\documentclass[twocolumn]{article}

\usepackage[letterpaper]{geometry}
\usepackage{ltexpprt}
\usepackage{amsmath,amsfonts,amssymb}
\usepackage[toc]{appendix}
\usepackage{array}
\usepackage{bbm,bm}
\usepackage{color}
\usepackage{comment}
\usepackage{enumitem}
\usepackage{epic}
\usepackage{footnote}
\usepackage[bottom]{footmisc}
\usepackage{graphics}
\usepackage{graphicx}
\usepackage{caption}
\usepackage{subcaption}
\usepackage{import}
\usepackage[utf8]{inputenc}
\usepackage{marginnote}
\usepackage{mathtools}
\usepackage{multirow}
\usepackage{psfrag}
\usepackage{setspace}
\usepackage{tabularx}
\usepackage[nice]{nicefrac}
\usepackage{booktabs}
\usepackage{minitoc}
\usepackage[dvipsnames]{xcolor}
\usepackage{pgfplots}
\usepackage{hyperref}
\usepackage{xurl}
\usepackage{booktabs}
\usepackage{breakurl}

\usepackage{algorithm}
\usepackage{algpseudocode}
\newtheorem{thm}{Theorem}[section]
\newtheorem{coro}{Corollary}[section]
\newtheorem{prop}{Proposition}[section]
\newtheorem{defn}{Definition}[section]
\newtheorem{lem}{Lemma}[section]
\newtheorem{rem}{Remark}[section]

\newtheorem{note}{Note}[section]

\usepackage{definitions}

\title{Improved Performance of Stochastic Gradients with Gaussian Smoothing}

\newcommand{\erf}{\textup{erf}}
\newcommand{\relu}{\textup{relu}}

\mathtoolsset{showonlyrefs=true}

\definecolor{grad1}{HTML}{FEC010}
\definecolor{grad2}{HTML}{F06923}
\definecolor{grad3}{HTML}{B9561D}
\definecolor{grad4}{HTML}{4A707A}
\definecolor{grad5}{HTML}{0A4252}

\usetikzlibrary{patterns,decorations.pathreplacing,plotmarks}

\begin{document}
\author{{\sffamily A.~Starnes\thanks{Lirio AI Research, Lirio, Inc., Knoxville, TN 37923 ({\tt astarnes@lirio.com, cwebster@lirio.com}).}}
\and {\sffamily C.~Webster}\footnotemark[1]
}

\date{}

\maketitle


\begin{abstract}

This paper formalizes and analyzes Gaussian smoothing applied to two prominent optimization methods: Stochastic Gradient Descent (GSmoothSGD) and Adam (GSmoothAdam) in deep learning. By attenuating small fluctuations, Gaussian smoothing lowers the risk of gradient-based algorithms converging to poor local minima. These methods simplify the loss landscape while boosting robustness to noise and improving generalization, helping base algorithms converge more effectively to global minima. Existing approaches often rely on zero-order approximations, which increase training time due to inefficiencies in automatic differentiation. To address this, we derive Gaussian-smoothed loss functions for feedforward and convolutional networks, improving computational efficiency. Numerical experiments demonstrate the enhanced performance of our smoothing algorithms over unsmoothed counterparts, confirming the theoretical benefits.
\end{abstract}

\doparttoc 
\faketableofcontents 


\section{Introduction}
\label{sec:intro}


Optimization problems in machine learning often involve minimizing the expectation or sum of stochastic functions dependent on underlying observations. In many cases, the total loss is computed as the average loss over individual data samples. However, a significant challenge arises when optimal solutions for specific samples fail to generalize to the broader dataset or other observations. Stochastic Gradient Descent (SGD) is a widely used method to address these issues by computing gradients from random samples of the data. Despite its advantages, like gradient descent (GD), SGD can become trapped in suboptimal local minima. The methods proposed in this paper aim to overcome this limitation by introducing GSmooth techniques that enhance the performance of stochastic gradient-based algorithms.

In general, given $f(\bmx;\omega):\mathbb{R}^d\times \Omega\to\mathbb{R}$ where $\Omega$ is a probability space, we want to solve

\vspace{-0.25cm}
\begin{equation}
    \min_{x\in\mathbb{R}^d}E\big(f(\bmx;\omega)\big),
\end{equation}
where $E(\cdot)$ is the expectation over $\Omega$ with respect to the probability measure.

Typically, we do not have unrestricted access to $\Omega$ and instead are given a collection of samples where for each sample $\omega\in\Omega$ we can compute $f(\bmx;\omega)$ for any $\bmx\in\mathbb{R}^d$.
The optimization problem then becomes minimizing the sample average rather than the overall expectation.
In particular, let $\{\omega_1,...,\omega_K\}\subseteq \Omega$ represent our sample (or training set) and denote
\begin{equation}
    f(\bmx;\omega_k)=f_k(\bmx):\mathbb{R}^d\to\mathbb{R}
\end{equation}
for each $k\in\{1,...,K\}$.
Our goal is to find a global minimum of $f$ that satisfies
\begin{equation}
\label{eqn:conditions_on_f}
    E(f_k(\bmx))=f(\bmx)
    \text{ and }
    E(\nabla f_k(\bmx))=\nabla f(\bmx)
\end{equation}
almost surely (a.s.).
Note that if $f(\bmx)=\frac{1}{K}\sum_{k=1}^{K}f_{k}(\bmx)$, then $f$ satisfies the conditions in \eqref{eqn:conditions_on_f}.
Of course, if the probability distribution of $\Omega$ is known, one can modify $\nicefrac{1}{K}$ in the average to represent the correct likelihood of each observation.
Additionally, we consider the training set as fixed throughout this paper.

SGD samples a point from the training set and iteratively updates using the gradient over the training set.
Explicitly, at each step, $k_t\sim\text{Unif}([K])$ (where $[K]=\{1,...,K\}$), and
\begin{equation}
    \bmx_{t} = \bmx_{t-1} - \lrate\nabla f_{k_t}(\bmx_{t-1}).
\end{equation}
One drawback of SGD is that estimates of the gradient near a minimum can vary significantly depending on how much the gradient varies across the dataset as well as the magnitude of the learning rate.
This is due to the fact that a minimizer of $f$, denoted $\bmx_*$, is unlikely to be the minimizer of any $f_{k}$, which means $\|\nabla f_{k}(\bmx)\|$ will be nonzero for $\bmx$ near $\bmx_*$.
Hence the iterates will vary even near the minimizer.
A commonly used approach to address this is to decrease the learning rate on some schedule.
However, with a smaller learning rate, convergence is slower.

A popular modification of SGD is Adam \cite{KingmaBa14}, which has coordinate-wise adaptive learning rates and uses momentum. 
Intuitively, these changes mean that Adam increases the speed of convergence and reduces the noise of the gradient.
The updates for the momentum, variance, and iterates of Adam are 
\begin{align}
    \bmm_{t}&=\beta_t\bmm_{t-1}+(1-\beta_t)\nabla f_{k_t}(\bmx_{t-1})\\
    \bmv_t&=\theta_t\bmv_{t-1}+(1-\theta_t)(\nabla f_{k_t}(\bmx_{t-1}))^2\\
    \bmx_t&=\bmx_{t-1}-\eta_t\frac{\bmm_{t}}{\sqrt{\bmv_t+\epsilon}},
\end{align}
where $(\nabla f_{k_t}(\bmx_{t-1}))^2$ is performed coordinate-wise. Applying Gaussian smoothing to Adam allows us to attempt to reduce the noise of the gradient even further.

For any function $g:\mathbb{R}^d\to\mathbb{R}$, we define the $\sigma$-Gaussian smoothed version of $g$ as
\begin{equation}
    \label{eq:f_sigma}
    g_\sigma(\bmx) 
    = \frac{1}{\pi^{\nicefrac{d}{2}}} \int_{\mathbb{R}^d} g(\bmx + \sigma\bmu) \, e^{-\|\bmu\|_2^2} \,\mathrm{d} \bmu
\end{equation}
An example of a Gaussian smoothed function can be seen in Figure~\ref{fig:picture_of_smoothing}.
This example illustrates a non-convex function that becomes convex when the smoothing value is sufficiently large, specifically for any $\sigma \geq\sqrt{\nicefrac{2}{3}}$.
Gaussian smoothing can enhance the behavior of functions by flattening out small fluctuations (e.g., see~\cite{mobahi2015link}).
Assuming the gradient and integral can switch, the gradient of $g_\sigma(\bmx)$ can be computed as
\begin{equation}
    \label{eq:fs_grad}
	\nabla g_\sigma(\bmx) 
	= \frac{2}{\sigma \pi^{\nicefrac{d}{2}}}
	    \int_{\mathbb{R}^d} \bmu g(\bmx + \sigma \bmu) \, e^{-\|\bmu\|_2^2} \, \mathrm{d}\bmu,
\end{equation}
which is the result of a change of variables to pass $\bmx$ into the exponential, switching the gradient with the integral, and then reversing the change of variables.
For $f$ as in \eqref{eqn:conditions_on_f}, if $f_\sigma(\bmx)\in\mathbb{R}$ for each $\bmx\in\mathbb{R}^d$, we have
\begin{equation}
    f_{\sigma}(\bmx)
    =\frac{1}{\pi^{\frac{d}{2}}}\int_{\mathbb{R}^d}E(f_{k}(\bmx+\sigma\bmu))e^{-\|\bmu\|^2}d\bmu
     =E(f_{k,\sigma}(\bmx))
\end{equation}
where $f_{k,\sigma}$ is the $\sigma$-smoothed version of $f_k$.
We use the notation $f_{k,\sigma}$ to denote the smoothed version of $f_k$ (rather than $f_{\sigma,k}$) because we first evaluate at $\omega_k$ and then smooth. It is not possible to smooth the $\omega$-component of $f$ and then evaluate at $\omega_k$ since we do not have access to every $\omega$.

\begin{figure*}
    \centering
    \def\s{2.5}
    \begin{tikzpicture}[scale=0.45, font=\LARGE]
        \begin{axis}[
            axis lines = center,
            xtick={0},
            ytick={0},
            xmin=-3.5,
            xmax=4.5,
            ymin=-14,
            ymax=4.5,
        ]
        \addplot [domain=-3.2:3.9, samples=500,latex-latex,dashed,thick]{0.25*x^4-0.6666*x^3-2.5*x^2+6*x+sin(deg(10*x))} node[right,pos=0.95] {$f$};
        \addplot [domain=-1.15:1.65, samples=100,latex-latex,thick,grad1]{
            0.25*x^4-0.6666*x^3+(0.75*\s^2-2.5)*x^2+(6-\s^2)*x+0.1875*\s^4-1.25*\s^2+sin(deg(10*x))*2.718^(-25*\s^2)
        } node[left,pos=0.1] {$f_{\s}$};
        \end{axis}
    \end{tikzpicture}
    \def\s{2}
    \begin{tikzpicture}[scale=0.45, font=\LARGE]
        \begin{axis}[
            axis lines = center,
            xtick={0},
            ytick={0},
            xmin=-3.5,
            xmax=4.5,
            ymin=-14,
            ymax=4.5,
        ]
        \addplot [domain=-3.2:3.9, samples=500,latex-latex,dashed,thick]{0.25*x^4-0.6666*x^3-2.5*x^2+6*x+sin(deg(10*x))} node[right,pos=0.95] {$f$};
        \addplot [domain=-1.9:2.28, samples=100,latex-latex,thick,grad2]{0.25*x^4-0.6666*x^3+(0.75*\s^2-2.5)*x^2+(6-\s^2)*x+0.1875*\s^4-1.25*\s^2+sin(deg(10*x))*2.72^(-25*\s^2)} node[right,pos=0.1] {$f_{\s}$};
        \end{axis}
    \end{tikzpicture}
    \def\s{1.5}
    \begin{tikzpicture}[scale=0.45, font=\LARGE]
        \begin{axis}[
            axis lines = center,
            xtick={0},
            ytick={0},
            xmin=-3.5,
            xmax=4.5,
            ymin=-14,
            ymax=4.5,
        ]
        \addplot [domain=-3.2:3.9, samples=500,latex-latex,dashed,thick]{0.25*x^4-0.6666*x^3-2.5*x^2+6*x+sin(deg(10*x))} node[right,pos=0.95] {$f$};
        \addplot [domain=-2.5:3, samples=100,latex-latex,thick,grad3]{0.25*x^4-0.6666*x^3+(0.75*\s^2-2.5)*x^2+(6-\s^2)*x+0.1875*\s^4-1.25*\s^2+sin(deg(10*x))*2.72^(-25*\s^2)} node[right,pos=0.1] {$f_{\s}$};
        \end{axis}
    \end{tikzpicture}
    \def\s{0.5}
    \begin{tikzpicture}[scale=0.45, font=\LARGE]
        \begin{axis}[
            axis lines = center,
            xtick={0},
            ytick={0},
            xmin=-3.5,
            xmax=4.5,
            ymin=-14,
            ymax=4.5,
        ]
        \addplot [domain=-3.2:3.9, samples=500,latex-latex,dashed,thick]{0.25*x^4-0.6666*x^3-2.5*x^2+6*x+sin(deg(10*x))} node[right,pos=0.95] {$f$};
        \addplot [domain=-3.15:3.8, samples=100,latex-latex,thick,grad4]{0.25*x^4-0.6666*x^3+(0.75*\s^2-2.5)*x^2+(6-\s^2)*x+0.1875*\s^4-1.25*\s^2+sin(deg(10*x))*2.72^(-25*\s^2)} node[right,pos=0.05] {$f_{\s}$};
        \end{axis}
    \end{tikzpicture}
    \def\s{0.2}
    \begin{tikzpicture}[scale=0.45, font=\LARGE]
        \begin{axis}[
            axis lines = center,
            xtick={0},
            ytick={0},
            xmin=-3.5,
            xmax=4.5,
            ymin=-14,
            ymax=4.5,
        ]
        \addplot [domain=-3.2:3.9, samples=500,latex-latex,dashed,thick]{0.25*x^4-0.6666*x^3-2.5*x^2+6*x+sin(deg(10*x))} node[right,pos=0.95] {$f$};
        \addplot [domain=-3.2:3.85, samples=500,latex-latex,thick,grad5]{0.25*x^4-0.6666*x^3+(0.75*\s^2-2.5)*x^2+(6-\s^2)*x+0.1875*\s^4-1.25*\s^2+sin(deg(10*x))*2.72^(-25*\s^2)} node[right,pos=0.04] {$f_{\s}$};
        \end{axis}
    \end{tikzpicture}
    \caption{Gaussian smoothing transforming $f(x)=\frac{1}{4}x^4-\frac{2}{3}x^3-\frac{5}{2}x^2+6x+\sin(10x)$ into a convex function}
    \label{fig:picture_of_smoothing}
\end{figure*}

We can now generalize SGD by using the gradient of the smoothed version of $f_k$ instead of the gradient of $f_k$ itself.
We call this modification Gaussian smoothed SGD (GSmoothSGD), which is presented in Algorithm~\ref{alg:gssgd}.
In order to improve efficiency, we combine Gaussian smoothing with Adam (GSmoothAdam), which can be found in Algorithm~\ref{alg:gsadam}.
GSmoothSGD is a generalization of several other modifications of SGD, but, to the best of our knowledge, no one has unified these into the same framework for the analysis of $L$-smooth functions 
(see Section~\ref{sec:related_works} for further discussion).

\begin{algorithm}[t]
    \caption{GSmoothSGD}
    \label{alg:gssgd}
    \begin{algorithmic}[1]
        \Require $f:\mathbb{R}^d\to\mathbb{R}$, $(\sigma_t)_{t=1}^{T}$, $\bmx_0\in\mathbb{R}^d$, $\lrate>0$
        \For{$t=1\to T$}
            \State $k_t\sim\text{Unif}([K])$
            \State $\bmx_t = \bmx_{t-1} - \lrate\nabla f_{k_t,\sigma_t}(\bmx_{t-1})$
        \EndFor 
    \end{algorithmic}
\end{algorithm}

\begin{algorithm}[b]
    \caption{GSmoothAdam}
    \label{alg:gsadam}
    \begin{algorithmic}[1]
        \Require{$\eta_t>0$, $\epsilon\geq 0$, $0\leq\beta_t\leq\beta<1$, $\theta_t\in(0,1)$, $x_0\in\mathbb{R}^d$, $\sigma_t\geq 0$ for $t=1,2,...$}
        \Ensure{$m_0=0$, $v_0=0$}
        \For{$t=1,2,...$}
            \State $k_t\sim\text{Unif}([K])$
            \State \label{lst:line:momentum} $\bmm_{t+1}=\beta_t\bmm_{t-1}+(1-\beta_t)\nabla f_{k_t,\sigma_t}(\bmx_{t-1})$
            \State \label{lst:line:variance} $\bmv_t=\theta_t\bmv_{t-1}+(1-\theta_t)(\nabla f_{k_t,\sigma_t}(\bmx_{t-1}))^2$
            \State \label{lst:line:iterate} $\bmx_t=\bmx_{t-1}-\eta_t\dfrac{\bmm_{t+1}}{\sqrt{\bmv_t+\epsilon}}$
        \EndFor
    \end{algorithmic}
\end{algorithm}


Our goal is to apply smooth stochastic gradient techniques to deep learning problems, specifically by training neural networks using Gaussian smoothing. To achieve this, we need to compute the expectation in \eqref{eq:fs_grad}, typically approximated via numerical integration, as direct computation is intractable due to high dimensionality. However, this approach is computationally expensive. Inspired by~\cite{mobahi2016training}, we instead derive the analytical form of the neural network after Gaussian smoothing. With this closed form, the smoothed network can be coded and trained as efficiently as its unsmoothed counterpart. Moreover, having the exact smoothed gradient, rather than an approximation, aligns with our theoretical results.

The main contributions of this paper are:
%
\begin{itemize}
\itemsep2pt
    \item 
    Formalize GSmoothSGD and prove convergence results for $L$-smooth functions and arbitrary smoothing parameter sequences (Section~\ref{sec:GSSGD});
    \item
    Introduce GSmoothAdam and prove almost sure convergence of gradients to a stationary point for L-smooth functions (Section~\ref{sec:GSAdam});
    \item 
    Derive the mathematical formulation of Gaussian-smoothed feedforward (FFNN) and convolutional (CNN) neural networks and implement the new architecture in Python (Section~\ref{sec:explicit_smoothing}); and 
    \item 
    Provide numerical evidence demonstrating the effectiveness of smoothing in stochastic gradient settings (Section~\ref{sec:numerics}).
\end{itemize}

\begin{rem}
Stochastic Variance Reduced Gradient (SVRG)~\cite{JohnsonZhang13} is another technique for reducing gradient variance near the minimum.
In Section~\ref{app:convergence_svrg}, we discuss Gaussian smoothing SVRG (GSmoothSVRG), which we include to show another example of how to apply Gaussian smoothing.
\end{rem}

\subsection{Related Works}
\label{sec:related_works}

The use of Gaussian smoothing has roots in gradient-free optimization, homotopy continuation, and partial differential equations.
Gaussian smoothing has been applied to the non-stochastic (deterministic) gradient setting.
For an overview of smoothing gradient descent, see, e.g.,  \cite{gsgd} and the references therein.

While not the focus of this paper, Gaussian smoothing is popular in gradient-free optimization.
One of the seminal papers, \cite{nesterov2017random}, proposes Gaussian smoothing in order to construct a zero-order SGD-type method to optimize, since \eqref{eq:fs_grad} can clearly be approximated with numerical integration techniques.
The work~\cite{nesterov2017random} also provides several of the foundational results that we utilize in our work, see, e.g., Lemma~\ref{lem:fsigmalsmoothandconvex}, and the discussion that follows, with a focus on using the gradient of the smoothed function as a surrogate for the original gradient with a fixed (small) smoothing parameter value.
A zero-order Adam method is proposed in~\cite{chen2019zo}, where they also show that this method converges.
For a good overview of zero-order methods including those related to Gaussian smoothing, see~\cite{liu2020primer}.

From the perspective of exploiting partial differential equations (PDEs) to assist with optimization, a common approach is to use a PDE where the initial condition is the objective function and, as time increases, the PDE transforms the objective function into an improved version of itself.
An alternate definition for $g_{\sigma}$ is to define a convolution between $g$ and a Gaussian kernel.
Functions of this form are solutions to the heat equation with initial condition given by $g$ (see Section 2.3 of \cite{evans2010partial}).

The PDE version of smoothing is a particular case of homotopy continuation, where the homotopy is given by the solution to the PDE and is often at least differentiable.
The standard optimization by homotopy continuation algorithm (OGHC) finds the minimizers of the homotopy at $t$, starting at $t=1$ and iteratively reducing $t$ to 0. 
The homotopy is chosen so that optimizing the homotopy at $t=1$ is very easy.
Our modification of SGD is a generalized form of OGHC, now allowing for $\sigma$ to change at will rather than when close enough to the inner loops minima.
The majority of the theoretical results for OGHC come from the works~\cite{mobahi2012seeing, mobahi2012gaussian, mobahi2016closed, mobahi2015theoretical, mobahi2015link}.
Aside from the theory, the effort~\cite{mobahi2016training} trains recurrent neural networks using OGHC.
In the abstract, \cite{mobahi2016training} mentions that this approach is ``applicable to FFNNs''.
This work is what primarily motivated us to explicitly smooth FFNNs as well as CNNs.

To the best of our knowledge, no other works have attempted to combine Gaussian smoothing with Adam.
On the other hand, a number of other papers have modified OGHC as well and proven convergence results; we discuss the three most related which can be viewed as particular examples of GSmoothSGD.
For perspective, our convergence results for GSmoothSGD only assume that $f$ is $L$-smooth (the proof can be modified to show similar results if $f$ is just Lipschitz as was done in~\cite{gsgd}).
The work \cite{hazan2016graduated} focuses on the noisy problem where evaluations of the function or its gradient has noise, i.e., $f(\bmx) + \xi$ or $\nabla f(\bmx)+\xi$ with $\xi$ bounded random variable.
The authors' modification of OGHC uses $\sigma_t=\frac{1}{2}\sigma_{t-1}$, which can be viewed as using GSmoothSGD where $\sigma_t$ repeats as many times as they run SGD.
They also add an additional step after the SGD steps where an average is taken over a decision set that decreases between iterations.
Their convergence results have very strong assumptions, in particular $\|\bmx_{\sigma}^*-\bmx_{\sigma/2}^*\|\leq\frac{\sigma}{2}$ (where $\bmx_{\sigma}^*$ minimizes $f_{\sigma}$) and $f_{\sigma}$ is strongly convex in a ball $B(\bmx_{\sigma}^*,3\sigma)$.
These assumptions restrict the possibilities for $f$, for more details see Section 2.3 of \cite{gsgd}.
Zero-th order Perturbed Stochastic gradient descent (ZPSGD) was proposed in \cite{jin2018local}, which follows GSmoothSGD with a fixed $\sigma$ value and approximates the gradient with Monte Carlo estimates of \eqref{eq:fs_grad}.
The convergence results in~\cite{jin2018local} assume $f$ is Lipschitz, $L$-smooth, and bounded.

The most related paper to our analysis of GSmoothSGD is~\cite{iwakiri2022single}, which proposes the Single Loop Gaussian Homotopy (SLGH).
The authors convert OGHC from a double loop into a single loop that follows GSmoothSGD where $\sigma_t$ is updated either by $\sigma_{t}=\gamma\sigma_{t-1}$ for some $0<\gamma<1$ or by a $\sigma$-gradient descent step of $f_{\sigma}$, i.e., using $\frac{\partial}{\partial\sigma}f_{\sigma}(\bmx_{t-1})|_{\sigma=\sigma_t}$.
From the perspective of the heat equation, their parameters for the second option are updated using an SGD step on $f_{\sigma}$.
Finally, the convergence results they present are the least restrictive of the ones discussed thus far, only assuming that $f$ is bounded (see discussion after Lemma 2.10 of~\cite{gsgd}) and both Lipschitz and $L$-smooth.
Despite only requiring one of their assumptions ($f$ $L$-smooth), we provide similar convergence results.

\section{Background and preliminaries}
\label{sec:background}

In this section, we provide the necessary results and definitions required to prove the convergence of GSmoothSGD and GSmoothAdam.
As is typically the case with optimization results, we often use convexity and $L$-smoothness, $f$ being $L$-smooth means that $\nabla f$ is $L$-Lipschitz (full definitions can be found in Definition~\ref{defn:background_defs} in the Appendix).
For the stochastic gradient setting, we need a few results from the deterministic gradient setting.
The first result, Lemma~\ref{lem:fsigmalsmoothandconvex} whose proof can be found in \cite{nesterov2017random}, shows that smoothing preserves convexity and $L$-smoothness.
The rest of the results, Lemma~\ref{lem:fsigmagreaterthanfnonconvex}-\ref{lem:differentsmoothingvalues}, show how smoothing impacts the values of $f$ and their proofs can be found in \cite{gsgd}.
\begin{lem}
Let $f:\mathbb{R}^d\to\mathbb{R}$ and $\tau\geq\sigma\geq 0$.
\begin{enumerate}[label={(\alph*)},ref={\thelem~(\alph*)}]
    \item\label{lem:fsigmalsmoothandconvex} If $f$ is convex or $L$-smooth, then $f_{\sigma}$ is also.
    \item\label{lem:fsigmagreaterthanfnonconvex} If $f$ is nonconstant and $f(\bmx)\geq m$, then $f_{\sigma}(\bmx)> m$.
    \item\label{lem:fsigmagreaterthanf} If $f$ is convex, then $f_{\sigma}(\bmx)\geq f(\bmx)$.
    \item\label{lem:differentsmoothingvalues} If $f$ is $L$-smooth then $|f_{\tau}(\bmx)-f_{\sigma}(\bmx)|\leq\frac{Ld}{4}(\tau^2-\sigma^2)$.
\end{enumerate}
\end{lem}
%


In our convergence analysis, we need to bound the gradient of the smoothed function using the original function's gradient. 
The first result of Lemma~\ref{lem:difference_between_smoothed_gradients} reverses the inequality from~\cite{nesterov2017random} Lemma 4 (restated in~\ref{app:background_results}).
Lemma~\ref{lem:difference_between_smoothed_gradients}'s second result allows us to compare the gradients for varying smoothing values.
The proof of this lemma is in Appendix~\ref{app:background_results}.

\begin{lem}
\label{lem:lemma4_nesterov_modification}
\label{lem:difference_between_smoothed_gradients}
If $f:\mathbb{R}^d\to\mathbb{R}$ is $L$-smooth, then for any $\sigma,\tau\geq 0$
\begin{enumerate}[label={(\alph*)},ref={\thelem~(\alph*)}]
    \item $\|\nabla f_{\sigma}(\bmx)\|^2\leq 2\|\nabla f(\bmx)\|^2+\frac{1}{4}L^2\sigma^2(6+d)^3$
    \item $\|\nabla f_{\tau}(\bmx)-\nabla f_{\sigma}(\bmx)\|\leq L\left(\frac{3+d}{2}\right)^{\nicefrac{3}{2}}\sqrt{|\tau^2-\sigma^2|}$
\end{enumerate}
\end{lem}

\section{Gaussian smoothing SGD (GSmoothSGD)}
\label{sec:GSSGD}

Here we prove convergence results for GSmoothSGD.
Further, if $\sigma_t=0$ (i.e., no smoothing occurs), then we recover the standard convergence guarantees as with SGD.
If $\sigma_t$ is constant, then we show that the expectation of the gradient from GSmoothSGD ends up in to a ball centered at zero whose radius depends of this constant.

\begin{thm}
\label{thm:GSSGD}
Consider GSmoothSGD in Algorithm~\ref{alg:gssgd}.
Assume that $f$ is $L$-smooth in $\bmx$ and $E(\nabla f_k(\bmx))=\nabla f(\bmx)$.
Let $f^*$ denote the minimum of $f$.
Suppose that $E(\|\nabla f_k\|^2)\leq\varbound$ for any $k\in [K]$. 
Then for some $\nu<T$, 
\begin{multline}
\label{eqn:gssgd_result}
    E(\|\nabla f(\bmx_{\nu})\|^2)\leq\frac{2(f(\bmx_0)-f^*)}{T\lrate}
    +2\varbound\\
    +\frac{1}{2T\lrate^2}\sum_{t=1}^{T}\Big(d|\sigma_{t}^2-\sigma_{t-1}^2|+\sigma_{t}^2(6+d)^3\Big).
\end{multline}
\end{thm}

The proof modifies the standard SGD proof (see, e.g., \cite{sgdconvergence}), where, in addition to focusing on changes between $\bmx_t$ and $\bmx_{t+1}$ in successive iterations, we also address changes between $\sigma_t$ and $\sigma_{t+1}$.
Since we have bounds on the difference between $f_{\sigma_t}$ and $f_{\sigma_{t+1}}$, as well as on $\|\nabla f_{\sigma_{t+1}}\|$ and $\|\nabla f\|$, there are two methods for adapting the standard proof.
Specifically, we can replace $f_{\sigma_{t+1}}$ with $f_{\sigma_t}$ or substitute $\nabla f_{\sigma_{t+1}}$ with $\nabla f$.
However, we must make these changes carefully since each time we do it, there is a cost that depends on the smoothing parameters (e.g., added $\frac{Ld}{4}|\sigma_t^2-\sigma_{t-1}^2|$ terms).
The full proof can be found in Appendix~\ref{app:GSSGD}.

Compared to the standard convergence results for SGD (the terms in \eqref{eqn:gssgd_result} on the left of the sum), we obtain additive bounds that depend on $\sigma_t$ (specifically, the two terms in the sum).
The first of these terms, $|\sigma_{t}^2-\sigma_{t-1}^2|$, indicates that there is a theoretical cost to changing the smoothing parameters.
The second term, $\sigma_{t}^2$, represents an inherent theoretical cost to smoothing, arising from the fact that the minima of $f_{\sigma}$ and $f$ are not necessarily the same (i.e., even if $\|\nabla f_{\sigma_{t+1}}(\bmx_t)\|$ is small, $\|\nabla f(\bmx_k)\|$ may still not be sufficiently small).
However, if we assume, for instance, that $(\sigma_t)_{t=1}^{\infty}$ is square-summable, then as $T \to \infty$, the right-hand side of \eqref{eqn:gssgd_result} still converges to $2\lambda$, as in the unsmoothed setting.

Theorem~\ref{thm:GSSGD} only asserts that the convergence rate of GSmoothSGD is not worse than that of SGD.
However, depending on the function being optimized, GSmoothSGD can be shown to outperform SGD.
Our next result demonstrates this by providing sufficient conditions under which GSmoothSGD converges, regardless of the initial point.
Before stating the result, we introduce the following definition: For $\sigma \geq 0$, the basin of attraction at $\sigma$, denoted by $\mathcal{B}(\sigma)$, is the set of initial points where GSmoothSGD converges to the global minimizer of $f_{\sigma}$ for some sequence of smoothing parameters. (The formal definition can be found in Definition~\ref{defn:def_basin_of_attraction}.)

\begin{prop}
\label{prop:gssgd_improvement}
Assume there exists $\Sigma>0$ so that, for $k\in [K]$, $f_{k,\Sigma}$ is convex and share a common minimizer.
Suppose there exists a decreasing sequence, $(\sigma_n)_{n=1}^{N}$, with $\sigma_1=\Sigma$ and $\sigma_N=0$ such that for $n=1,2,\dots,N-1$, $f_{\sigma_n}$ has a unique minimizer $\bmx_{\sigma_n}^{\star}$ that is in an open neighborhood of $\mathcal{B}(\sigma_{n+1})$.
Then $\mathbb{R}^d=\mathcal{B}(0)$.
\end{prop}

The proof can be found in the Supplementary Material, Section~\ref{app:GSSGD}.
The proof's key idea is that, for each $\sigma_n > 0$, a finite sequence of smoothing parameters can be found that brings GSmoothSGD close to the global minimizer of $f_{\sigma_n}$.
These sequences are then chained together, including a potentially infinite sequence of smoothing parameters for $\sigma_N = 0$, resulting in a complete sequence that enables GSmoothSGD to converge to the global minimum of $f$.

As a concrete example, shown in Figure~\ref{fig:gssgd_improvement_example}, we consider $f(\bm{x})=\frac{1}{2}(\bm{x}^4-2\bm{x}^2)+\sin(8\bm{x})$ (where we have only a single observation so $f=f_k$).
We can analytically compute $f_{\sigma}$ and see that it is convex for $\sigma\geq 0.85$.
We use GSmoothSGD on $\sigma_1=1$, $\sigma_2=\frac{1}{2}$, $\sigma_3=\frac{1}{4}$, and $\sigma_4=0$.
Starting at any $\bmx\in\mathbb{R}$, GSmoothSGD with constant smoothing value $\sigma_1$ will converge to the global minimum of $f_{\sigma_1}$.
Once the iterates of GSmoothSGD are close enough to $\bmx_{\sigma_1}^{\star}$, we switch from $\sigma_1$ to $\sigma_2$ and repeat the process.
The result is that the basin of attraction for the global minimizer of $f$ using GSmoothSGD is all of $\mathbb{R}$, compared to a strictly smaller subset using SGD.

Unfortunately, the proof is an existence statement and does not provide practical details on constructing the sequence of smoothing parameters.
The open neighborhood assumption holds if $f$ has a continuous second derivative and the Hessian of $f_{\sigma}$ is positive definite at its global minimum, a common assumption in homotopy continuation theory (see, e.g.,~\cite{wu1996effective} Theorem 3).
Additionally, the condition that $\bmx_{\sigma_n}^{\star} \in \mathcal{B}(\sigma_{n+1})$ is analogous to an assumption in the convergence proof of Homotopy Optimization with Perturbations and Ensembles (HOPE) (\cite{dunlavy2005homotopy} Theorem 3.1).

\begin{figure*}[t]
    \centering
    \def\scale{0.5}
    \begin{subfigure}[t]{0.3\linewidth}
    \centering
        \begin{tikzpicture}[scale=\scale, font=\LARGE]
            \begin{axis}[
                axis lines = center,
                xtick={0},
                ytick={0},
                xmin=-2.2,
                xmax=1.9,
                ymin=-1.85,
                ymax=4.1,
                width=3.3in,
                height=2.9in,
                scale only axis
            ]
            \def\s{1}
            \addplot [domain=-1.55:1.55, samples=100,latex-latex,very thick,grad1]{0.5*(x^4+(3*\s^2-2)*x^2 + 0.75*\s^4-\s^2)+2.72^(-16*\s^2)*sin(deg(8*x))} node[right,pos=0.01] {$f_{1}$};
            \def\s{0.5}
            \addplot [domain=-1.72:1.714, samples=100,latex-latex,very thick,grad2]{0.5*(x^4+(3*\s^2-2)*x^2 + 0.75*\s^4-\s^2)+2.72^(-16*\s^2)*sin(deg(8*x))} node[above,pos=0] {$f_{\frac{1}{2}}$};
            \def\s{0.25}
            \addplot [domain=-1.81:1.733, samples=100,latex-latex,very thick,grad3]{0.5*(x^4+(3*\s^2-2)*x^2 + 0.75*\s^4-\s^2)+2.72^(-16*\s^2)*sin(deg(8*x))} node[anchor=-45,pos=0] {$f_{\frac{1}{4}}$};
            \addplot [domain=-1.85:1.66, samples=500,latex-latex,very thick,grad4]{0.5*(x^4-2*x^2)+sin(deg(8*x))} node[left,pos=0] {$f$};
            \end{axis}
        \end{tikzpicture}
        \caption{Plots of $f_1$, $f_{\frac{1}{2}}$, $f_{\frac{1}{4}}$, and $f$}
    \end{subfigure}
    \begin{subfigure}[t]{0.69\linewidth}
    \centering
        \def\s{1}
        \begin{tikzpicture}[scale=\scale, font=\LARGE]
            \begin{axis}[
                axis lines = center,
                xtick={0},
                ytick={0},
                xmin=-1.8,
                xmax=0.75,
                ymin=-1.85,
                ymax=4.1,
                width=1.85in,
                height=2.9in,
                scale only axis
            ]
            \addplot [domain=-1.55:0.65, samples=100,latex-latex,very thick,grad1]{0.5*(x^4+(3*\s^2-2)*x^2 + 0.75*\s^4-\s^2)+2.72^(-16*\s^2)*sin(deg(8*x))} node[right,pos=0.02] {$f_{1}$};
            \addplot [grad1, mark = *] coordinates {(0,-0.125)};
            \addplot [dashed, very thick, grad1] coordinates {(-0.3,-0.4) (-0.3,0.5)};
            \addplot [dashed, very thick, grad1] coordinates {(0.3,-0.4) (0.3,0.5)};
            \filldraw[fill=white, draw = white] (axis cs: -0.5,-1.2) rectangle (axis cs: 0.5,-0.45);
            \draw [grad1,very thick,decoration={brace,mirror,raise=2pt},decorate] 
                  (axis cs:-0.3,-0.4) --
                    node[below=4pt] {$B(\bm{x}_{1}^{\star},\delta_1)$} 
                  (axis cs:0.3,-0.4);
            \addplot [mark = *, nodes near coords={$\bm{x}_0$}, every node near coord/.style={anchor=0}] coordinates {(-1.5,3.531)};
            \addplot [mark = *, nodes near coords={$\bm{x}_{N_1}$}, every node near coord/.style={anchor=20}] coordinates {(-0.2,-0.10)};
            \addplot [-latex, very thick] coordinates {(-1.37,3.17) (-0.33,0.26)};
            \end{axis}
        \end{tikzpicture}
        \def\s{0.5}
        \begin{tikzpicture}[scale=\scale, font=\LARGE]
            \begin{axis}[
                axis lines = center,
                xtick={0},
                ytick={0},
                xmin=-2,
                xmax=0.1,
                ymin=-1.85,
                ymax=4.1,
                width=1.69in,
                height=2.9in,
                scale only axis
            ]
            \addplot [domain=-1.72:0.05, samples=100,latex-latex,very thick,grad2]{0.5*(x^4+(3*\s^2-2)*x^2 + 0.75*\s^4-\s^2)+2.72^(-16*\s^2)*sin(deg(8*x))} node[above,pos=0] {$f_{\frac{1}{2}}$};
            \addplot [grad2, mark = *] coordinates {(-0.839,-0.301)};
            \addplot [dashed, very thick, grad2] coordinates {(-0.539,-0.5) (-0.539,0.4)};
            \addplot [dashed, very thick, grad2] coordinates {(-1.139,-0.5) (-1.139,0.4)};
            \draw [grad2,very thick,decoration={brace,mirror,raise=2pt},decorate] 
                  (axis cs:-1.139,-0.5) --
                    node[below=4pt] {$B(\bm{x}_{\frac{1}{2}}^{\star},\delta_{\frac{1}{2}})$} 
                  (axis cs:-0.539,-0.5);
            \addplot [mark = *] coordinates {(-0.2,-0.144)};
            \node[below] at (axis cs:-0.15,-0.144) {$\bm{x}_{N_1}$};
            \addplot [mark = *] coordinates {(-0.65,-0.26)};
            \node[above] at (axis cs:-0.8,-0.2) {$\bm{x}_{N_2}$};
            \addplot [-latex, very thick] coordinates {(-0.29,-0.17) (-0.56,-0.24)};
            \end{axis}
        \end{tikzpicture}
        \def\s{0.25}
        \begin{tikzpicture}[scale=\scale, font=\LARGE]
            \begin{axis}[
                axis lines = center,
                xtick={0},
                ytick={0},
                xmin=-2.2,
                xmax=0.1,
                ymin=-1.85,
                ymax=4.1,
                width=1.85in,
                height=2.9in,
                scale only axis
            ]
            \addplot [domain=-1.81:0.05, samples=100,latex-latex,very thick,grad3]{0.5*(x^4+(3*\s^2-2)*x^2 + 0.75*\s^4-\s^2)+2.72^(-16*\s^2)*sin(deg(8*x))} node[left,pos=0] {$f_{\frac{1}{4}}$};
            \addplot [grad3, mark = *] coordinates {(-0.977,-0.807)};
            \addplot [dashed, very thick, grad3] coordinates {(-0.677,-0.9) (-0.677,0.5)};
            \addplot [dashed, very thick, grad3] coordinates {(-1.277,-0.9) (-1.277,0.5)};
            \draw [grad3,very thick,decoration={brace,mirror,raise=2pt},decorate] 
                  (axis cs:-1.277,-0.9) --
                    node[below=4pt] {$B(\bm{x}_{\frac{1}{4}}^{\star},\delta_{\frac{1}{4}})$} 
                  (axis cs:-0.677,-0.9);
            \addplot [mark = *] coordinates {(-0.65,-0.0015)};
            \node[above] at (axis cs:-0.4,-0.0015) {$\bm{x}_{N_2}$};
            \addplot [mark = *] coordinates {(-0.8,-0.44)};
            \node[right] at (axis cs:-0.75,-0.55) {$\bm{x}_{N_3}$};
            \addplot [-latex, very thick] coordinates {(-0.68,-0.09) (-0.77,-0.35)};
            \end{axis}
        \end{tikzpicture}
        \begin{tikzpicture}[scale=\scale, font=\LARGE]
            \begin{axis}[
                axis lines = center,
                xtick={0},
                ytick={0},
                xmin=-2.2,
                xmax=0.1,
                ymin=-1.85,
                ymax=4.1,
                width=1.85in,
                height=2.9in,
                scale only axis
            ]
            \addplot [domain=-1.85:0.05, samples=500,latex-latex,very thick,grad4]{0.5*(x^4-2*x^2)+sin(deg(8*x))} node[left,pos=0] {$f$};
            \addplot [mark = *] coordinates {(-0.8,-0.551)};
            \node[above] at (axis cs:-0.72,-0.5) {$\bm{x}_{N_3}$};
            \addplot [mark = *] coordinates {(-0.9828,-1.4992)};
            \node[left] at (axis cs:-0.9828,-1.4) {$\bm{x}^{\star}$};
            \addplot [-latex, very thick] coordinates {(-0.82,-0.65) (-0.96,-1.4)};
            \addplot [mark = *] coordinates {(-1.5,0.817)};
            \node[above] at (axis cs:-1.6,0.817) {$\bm{x}_{0}$};
            \end{axis}
        \end{tikzpicture}
        \caption{Changing $\sigma_n$ when iterates are within open neighborhood of basin of attraction $\mathcal{B}(\sigma_{n+1})$, $\delta_{\sigma_n}$ chosen so that $B(\bm{x}_{\sigma_n}^{\star},\delta_{\sigma_n})\subseteq\mathcal{B}(\sigma_{n+1})$}
    \end{subfigure}
        \caption{
        GSmoothSGD outperforming SGD. The basin of attraction for $\bm{x}^{\star}$ for GSmoothSGD is all of $\mathbb{R}$ but is strictly smaller for SGD. SGD starting at $\bm{x}_0$ does not find the global minimum, whereas GSmoothSGD does.
        }
    \label{fig:gssgd_improvement_example}
\end{figure*}

\begin{rem}
Instead of assuming that $f_{k,\Sigma}$ is convex, we could assume that $f_{k,\Sigma}$ is convex in a ball containing $\bmx_{\Sigma}^{\star}$ (see~\cite{mobahi2012optimization} Corollary 9 for sufficient conditions of $f$ for this to happen).
Then extend $f_{k,\Sigma}$ to be convex everywhere (see~\cite{yan2012extension} Theorem 3.2), which would make the basin of attraction for this fully convex function $\mathbb{R}^d$.
This is a small change to GSmoothSGD, but still shows off the benefits of Gaussian smoothing.
\end{rem}

\section{Gaussian smoothing Adam (GSmoothAdam)}
\label{sec:GSAdam}

Now, we show that GSmoothAdam has the same almost sure convergence as its unsmoothed counterpart.

\begin{thm}
\label{thm:gsadam}
Consider GSmoothAdam in Algorithm~\ref{alg:gsadam}.
Let $f$ be $L$-smooth and $f^*$ denote the minimum of $f$.
Assume $E(f_k(\bmx))=f(\bmx)$ and $E(\nabla f_k(\bmx))=\nabla f(\bmx)$.
Suppose that $E(\|\nabla f_k\|^2)\leq\lambda$ for any $k\in[K]$,
$\sum_{t=1}^{\infty}\eta_t=\infty$,
$\quad\sum_{t=1}^{\infty}\eta_t^2<\infty$,
$\quad\sum_{t=1}^{\infty}\eta_t(1-\theta_t)<\infty$,
and there exists a non-increasing sequence of real numbers $(\alpha_t)_{t\geq 1}$ such that $\eta_t=\Theta(\alpha_t)$.
If $\sqrt{|\sigma_t^2-\sigma_{t+1}^2|}\leq\eta_t$ and $\sigma_t\to 0$,
then
$\lim_{t\to\infty}\|\nabla f(\bmx_t)\|^2=0$ a.s. and $\lim_{t\to\infty}E(\|\nabla f(\bmx_t)\|^2)=0$.
\end{thm}

The proof is adapted from \cite{he2023convergence} to incorporate Gaussian smoothing.
Several lemmas from \cite{he2023convergence} can be applied as they are and do not require modification, as they apply directly to $f_{\sigma}$.
The only proofs that require modification are those where $\sigma$ varies (i.e., when we transition from $\sigma_t$ to $\sigma_{t+1}$).
As in the proof of GSmoothSGD convergence (Theorem~\ref{thm:GSSGD}), there are two methods to modify the original proof: replace $f_{\sigma_t}$ with $f_{\sigma_{t+1}}$ or substitute $\nabla f_{\sigma_{t+1}}$ with $\nabla f$.
The restrictions on $\eta_t$ come directly from~\cite{he2023convergence}.
The conditions on $\sigma_t$ are due to the need to guarantee that the sum of $|\sigma_t^2-\sigma_{t+1}^2|$ converges.
Full details of the proof of Theorem~\ref{thm:gsadam} are provided in Appendix~\ref{app:proofs_gsmoothadam}.

Compared to the GSmoothSGD convergence result, there are no additional terms dependent on $\sigma_t$ in the results Theorem~\ref{thm:gsadam}.
Instead, for certain sequences of smoothing parameters, the convergence results are identical to those of unsmoothed Adam.

\section{Explicitly Smoothing Neural Networks}
\label{sec:explicit_smoothing}


Gaussian smoothing in deep learning is often implemented using gradient-free approaches, such as Monte Carlo estimates of~\eqref{eq:fs_grad}. This method is computationally expensive, as each point in the gradient-free approximation requires randomly perturbing the model, a process more complex than simple function queries. To circumvent the computational burden of integral approximation, we explicitly compute the output of Gaussian smoothing applied to a network. Our approach is inspired by \cite{mobahi2016training}, which derives the explicit form of a Gaussian-smoothed recurrent neural network. We start by introducing the notation for explicit smoothing and then present the main results, detailing how Gaussian smoothing modifies the formulation of a neural network.

For a neural network, we use $C$ and $L$ to denote the number of convolutional and dense layers, respectively.
Given a convolutional layer $C_l$, let $k_l$ be the kernel size and $s_l$ be the stride length.
Then $w_l=\frac{\text{width}(x_{l-1})-k_l}{s_l}+1$ gives the width of the output of the convolution with a single kernel.
When we explicitly smooth a convolutional layer, we end up with an additional regularization term that depends on the input to that layer.
Define the convolutional norm of $x$ with respect to $C_l$ as
$\|x\|_{C_l}^2=\|x*J_{k_l}\|^2$
where $J_{k_l}$ is the $k_l\times k_l$ matrix whose entries are all 1.

Two additional regularization terms come from smoothing the dense and activation layers, which we define as
\begin{align}
        r_1(x,\theta)
        &=\|\theta\textup{diag}(\sqrt{(h^2)_{\sigma}}(x))\|^2_F
            -\|\theta\textup{diag}(h_{\sigma}(x))\|^2_F\\
        &\qquad\qquad+\tfrac{1}{2}\sigma^2 \textup{dim}(x)\|\sqrt{(h^2)_{\sigma}}(x)\|^2\\
        \hspace{-3pt}r_2(x,\theta,l)
        &=\tfrac{1}{2}\sigma^2\|h_{\sigma}(x)\|_{C_l}^2
            +\|\sqrt{(h^2)_{\sigma}}(x)\|_F^2\\
        &\qquad\qquad-\|h_{\sigma}(x)\|_F^2
            +\tfrac{1}{2}\sigma^2w_l^2\|\theta\|_F^2,
\end{align}
where $\|\cdot\|_F$ is the Frobenius norm.

With our notation defined, we are ready to state the optimization problem that Gaussian smoothed neural networks solve.
\begin{thm}
\label{thm:smooth_ffnn}
\label{thm:smooth_cnn}
Smoothing the Constrained FFNN gives
\begin{align}
\begin{split}
    \min_{\theta,b}\quad&\sum_{n=1}^{N}\left(
        \|x_L^n-y^n\|^2
        +\sum_{l=2}^{L}\lambda_lr_1(x_{l-1}^n,\theta_l)
    \right)\\
    \textup{s.t.}\quad&x_1^n=\theta_1x_0^n+b_1\\
    &x_l^n=\theta_lh_{\sigma}(x_{l-1}^n)+b_l\text{ for }l=2,...,L
\end{split}
\end{align}
Smoothing the Constrained CNN formulation gives (removing constants)
\begin{align}
\begin{split}
    \min_{\theta,b}&\quad\sum_{n=1}^{N}\Bigg(
        \|x_{C+L}^n-y^n\|^2
        +\mathcal{R}_n
    \Bigg)
\end{split}
\end{align}
\begin{align}
\begin{split}
    \textup{s.t.}\quad&x_1^n=x_0^n*\theta_1+b_1\\
    &x_l^n=h_{\sigma}(x_{l-1}^n)*\theta_l+b_l\text{ for }l=2,...,C\\
    &x_l^n=\theta_lh_{\sigma}(x_{l-1}^n)+b_l\text{ for }l=C+1,...,C+L\\
    &\mathcal{R}_n=\sum_{l=2}^{C}\lambda_lr_2(x_{l-1}^n,\theta_l,l)+\sum_{l=C+1}^{C+L}\lambda_lr_1(x_{l-1}^n,\theta_l)
\end{split}
\end{align}
\end{thm}

\begin{table*}
    \caption{Unsmoothed and smoothed layers (layers have no activation functions unless otherwise specified.)}
    \label{tab:summary_of_nnet_smoothing_names}
    \centering
    \resizebox{0.75\textwidth}{!}{
    \begin{tabular}{cccc|c}
    \multirow{3}{*}{\parbox{0.1\linewidth}{\vspace{7pt}\centering\textbf{Original}\\\textbf{Layer}\\\textbf{Name}}}&
    \multirow{3}{*}{\parbox{0.1\linewidth}{\vspace{7pt}\centering\textbf{Layer}\\\textbf{Number}\\($\bm{l}$)}}&
                                                         \multicolumn{3}{c}{\textbf{Smoothed}}\\\cmidrule(lr){3-5}
                 &            &\multirow{2}{*}{\parbox{0.1\linewidth}{\centering\textbf{Layer}\\\textbf{Name}}}           &\multicolumn{2}{c}{\textbf{Regularizers}}\\\cmidrule(lr){4-5}
                  &&            &\textbf{Input/Output}&\textbf{Weights}\\\specialrule{.1em}{.05em}{.05em}
         \multirow{2}{*}{Dense}&$l=1$                     &Dense                    &$-$                                                                      &$-$\\\cline{2-5}
                               &$l>1$                     &Dense                    &$\frac{\sigma^2d_l}{2}\|\bmx_{l-1}\|^2$                               &$\frac{\sigma^2}{2}\|\theta_l\|^2_F$\\\specialrule{.1em}{.05em}{.05em}
         Activation            &\multirow{2}{*}{$l>1$}                 &Activation&\multirow{2}{*}{$\|\sqrt{(h^2)_{\sigma}}(\bmx_{l-1})\|^2+\|h_{\sigma}(\bmx_{l-1})\|^2$}&\multirow{2}{*}{$-$}\\
         ($h$)     &                     &($h_{\sigma}$)&&\\\specialrule{.1em}{.05em}{.05em}
         \multirow{2}{*}{Convolution}           &$l=1$                     &Convolution              &$-$                                                                      &$-$\\\cline{2-5}
                    &$l>1$                     &Convolution              &$\frac{\sigma^2}{2}\|\bmx_{l-1}\|_{C_l}^2$                            &$\frac{\sigma^2w_l^2}{2}\|\theta_l\|^2_F$\\\specialrule{.1em}{.05em}{.05em}
         Dropout ($p$)         &$l\geq 1$                 &Identity                 &$p\|\bmx_l\|^2$                                                       &$-$\\\specialrule{.1em}{.05em}{.05em}
         Average Pool          &$l\geq 1$                 &Average Pool             &$-$                                                                      &$-$\\\specialrule{.1em}{.05em}{.05em}
         Flatten               &$l\geq 1$                 &Flatten                  &$-$                                                                      &$-$
    \end{tabular}
    }
\end{table*}

\noindent For a FFNN, this result says that a regularization term ($r_1$) is added to the loss and the original activation function are smoothed.
For a CNN, the loss has regularization terms associated with the convolution ($r_2$) and dense ($r_1$) layers and each activation function is replaced by its smoothed version.
Unless one is willing to spend computation time to estimate the smoothed activation function, it needs to be analytically computed.
Due to its popularity and prevalence, we compute all necessary computations for ReLU (see Proposition~\ref{prop:results_from_mobahi}), but this can be done for other activation functions (e.g., see~\cite{mobahi2016training} Table 1).

In order to prove Theorem~\ref{thm:smooth_ffnn} and apply Gaussian smoothing to a (convolutional) neural network, we explicitly calculate the optimization problems that smoothed FFNNs and CNNs satisfy.
We do this (as done in \cite{mobahi2016training}), by first converting the constrained problems \eqref{eqn:constrained_ffnn} and \eqref{eqn:constrained_cnn} to their corresponding unconstrained problems.
Then we smooth the unconstrained problems and convert them back into a constrained formulation.
The proofs of these results can be found in Appendix~\ref{app:details_explicit_smoothing}.

Although Theorem~\ref{thm:smooth_ffnn} provides the mathematical foundation for explicitly smoothed neural networks, translating this theory into code is not straightforward.
We now explain how to implement Gaussian-smoothed neural networks in code.
To begin, we separate the dense layer from its activation function and arrange them sequentially, meaning the dense layer appears without an activation function, followed by a separate activation layer.
This approach prevents regularization terms from mixing the weights and inputs of a layer (from $r_1$ particularly) and instead applies regularization terms separately to the weights and inputs.
We apply the same method to convolutional layers.
Table~\ref{tab:summary_of_nnet_smoothing_names} shows the correspondence between the original layers and their smoothed counterparts, along with the associated regularization terms.
Further explanation of this process, along with examples, can be found in Appendix~\ref{app:numerics}.

\section{Numerical experiments}
\label{sec:numerics}



\begin{figure*}[t]
    \centering
    \begin{subfigure}{0.90\linewidth}
        \includegraphics[width=0.98\textwidth]{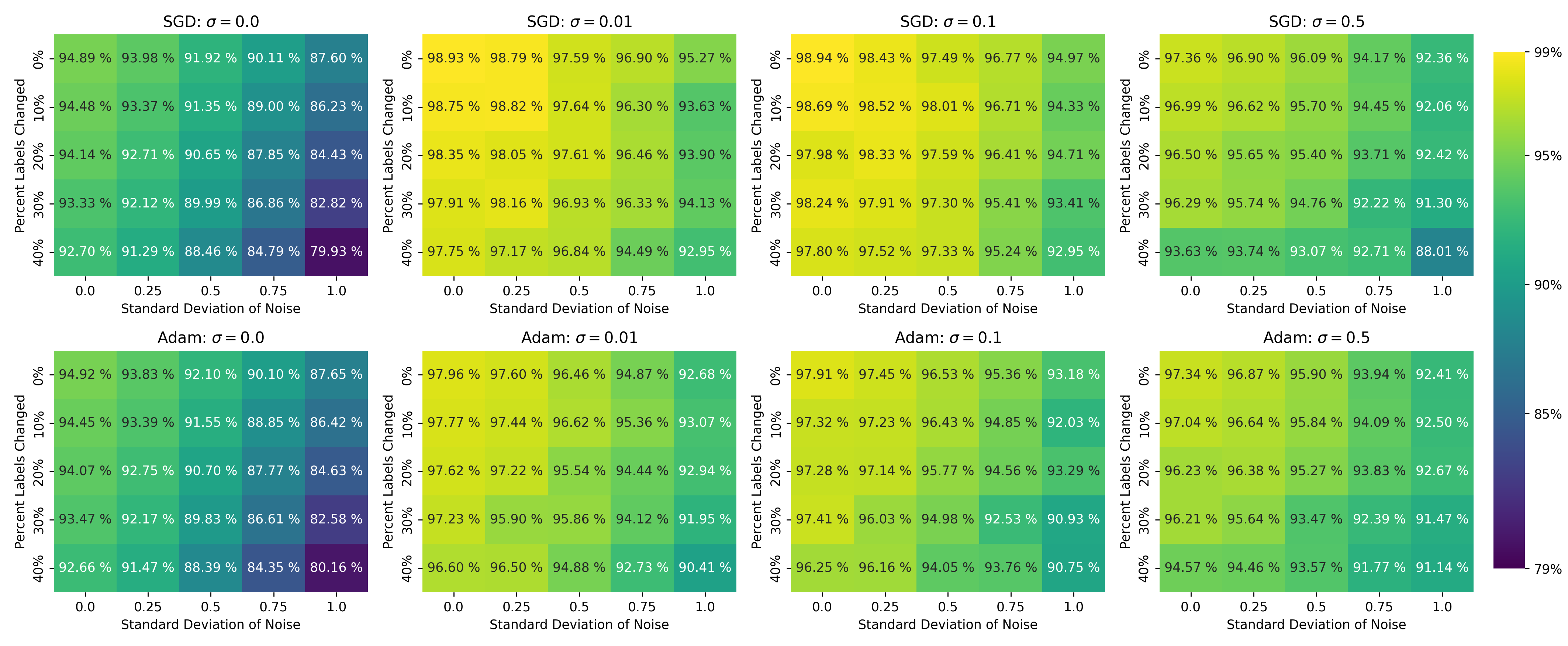}
        \caption{MNIST Experiments: SGD (top), Adam (bottom)}
        \label{fig:mnist_experiments}
    \end{subfigure}
    \begin{subfigure}{0.90\linewidth}
        \includegraphics[width=0.98\textwidth]{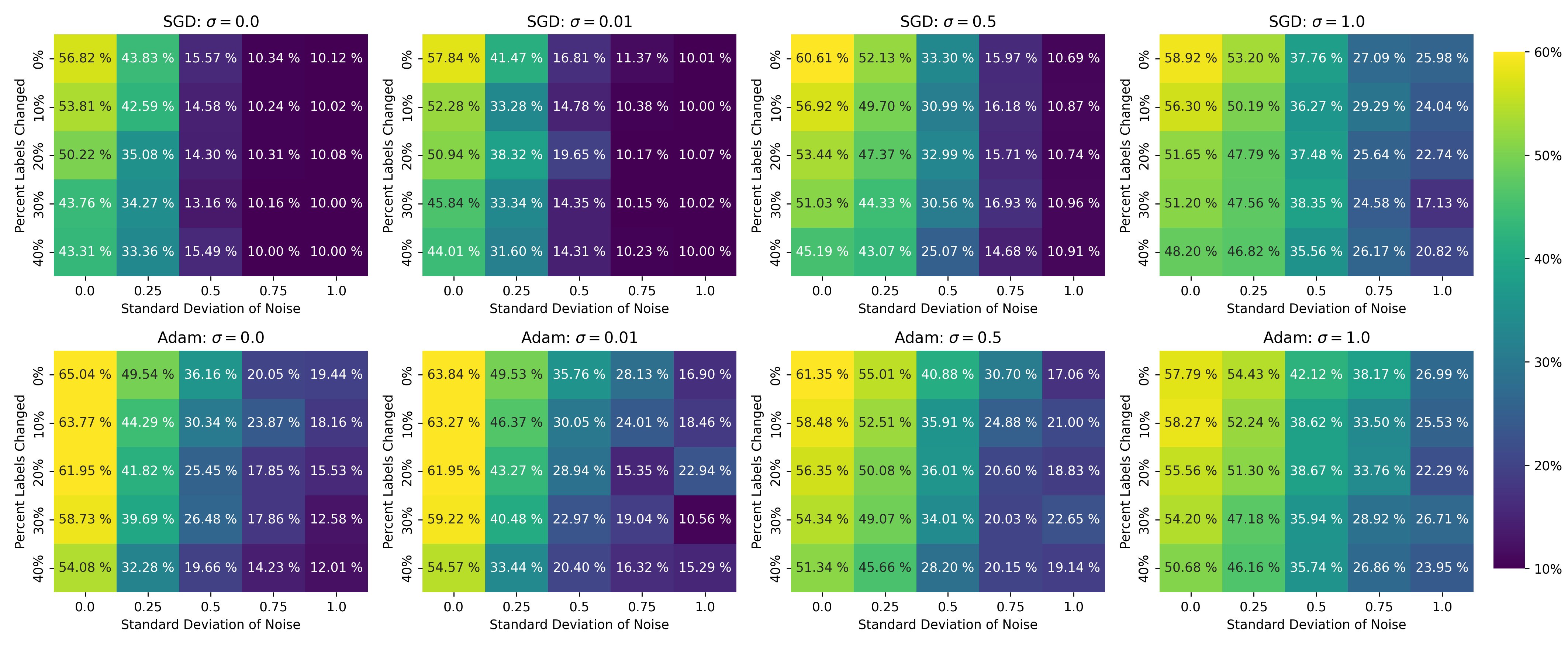}
        \caption{CIFAR-10 Experiments: SGD (top), Adam (bottom)}
        \label{fig:cifar10_experiments}
    \end{subfigure}
    \caption{Results from numerical experiments}
    \label{fig:experiments}
\end{figure*}

To apply our results, we train CNNs on the MNIST~\cite{LeCun2005TheMD} and CIFAR-10~\cite{Krizhevsky09learningmultiple} datasets with two types of noise.
First, we add Gaussian noise with a mean of 0 and standard deviations of 0, 0.25, 0.5, 0.75, and 1 to the training images (see Figure~\ref{fig:noisy_images} for examples).
Second, we randomly change the labels of 0\%, 10\%, 20\%, 30\%, and 40\% of the training images.
These combinations result in 25 variations of input and label noise.
The same network architecture was used for both the MNIST and CIFAR-10 experiments which consists of a convolutional layer with 32 kernels of size 4 and a stride of 1, followed by a ReLU activation and an average pooling layer with a window size and stride of 2.
This is followed by a dense layer with 128 nodes, a ReLU activation, and finally an output layer with 10 neurons, representing the 10 classes in the dataset.
This is summarized in Appendix Table~\ref{tab:cnn_architecture} and Table~\ref{tab:smooth_cnn_architecture} for the smoothed network.

All models are trained with a batch size of 1, using mean-squared error as the loss function, with the optimizer’s default parameters.
A hyperparameter search was conducted to tune the models using the MNIST dataset, utilizing 5-fold cross-validation on noise-free images, and the optimal hyperparameters can be found in Table~\ref{tab:smooth_cnn_hyperparameters}.
The training results of the experiments are evaluated on a test set without any noise.
The experiments were conducted using Python 3.10 and TensorFlow.
Full experimental details are provided in Section~\ref{app:numerics} of the Appendix.

We start by discussing the results of the MNIST experiments, which compare the performance of SGD, Adam, and their smoothed variants.
The key results are presented in Figure~\ref{fig:experiments}, while detailed results for all $\sigma$ values are available in Appendix Figure~\ref{fig:full_noise_mnist}.
For the unsmoothed CNN ($\sigma=0$), performance significantly decreases as the standard deviation increases, regardless of the optimizer used.
Although this is also true for the smoothed CNNs, the performance decline was not as pronounced.

For small $\sigma$-values (0.01, 0.1, 0.5), the smoothed networks consistently outperformed the unsmoothed network.
As $\sigma$ increases from 0.01 to 0.5, overall accuracy decreases in the SGD-trained models.
This suggests that on this dataset larger values of $\sigma$ are more susceptible to noise.
However, since the minimum occurs at $\sigma=0$, we expect smaller values of $\sigma$ to smooth the loss landscape sufficiently, resulting in better performance than larger $\sigma$ values.
In both the smoothed and unsmoothed cases, test accuracy exhibits greater variation as the standard deviation increases compared to the percentage of labels changed.
Surprisingly, the CNNs trained with SGD perform better than those trained with Adam.
This would likely change if the optimizer’s parameters were tuned.
However, the SGD models showed a more significant performance decline compared to Adam when input noise levels were high.

Switching to the CIFAR-10 experiments, the key results are shown in Figure~\ref{fig:cifar10_experiments} with full results shown in Appendix Figure~\ref{fig:full_noise_cifar10}.
Once again, we observe that image noise affects training more than label noise.
Smoothing the loss function improves stability for both optimizers by reducing their sensitivity to noise.
For both Adam and SGD, the best results occur when noise is added at $\sigma=1.0$, suggesting that more smoothing was necessary for the CIFAR-10 experiments compared to the MNIST experiments.
However, using a larger model to learn the CIFAR-10 labels seems warranted.
In contrast to the MNIST results, Adam generally performs better across all settings, particularly in high-noise conditions.
Although noise significantly impacts both optimizers, the effect is mitigated at higher $\sigma$ values, suggesting that loss-smoothing techniques effectively stabilize learning in noisy environments.

\section{Conclusions}
\label{sec:conclusions}

In this paper, we provided theoretical convergence proofs for smoothing two stochastic gradient algorithms, along with sufficient conditions for performance improvements. We first defined the GSmoothSGD algorithm and proved that it converges to a ball around the minimizer under specific sequences of smoothing parameters (e.g., when constant). We then introduced GSmoothAdam and demonstrated its almost-sure convergence to a stationary point, showing that smoothing does not impose additional computational costs. These results offer a general framework for applying smoothing to other stochastic gradient algorithms.

Our numerical experiments on neural networks further confirmed that smoothing improves performance and robustness to noise in training. To align with our theoretical findings, we explicitly derived Gaussian-smoothed versions of feedforward and convolutional neural networks, which can be adapted to a broad range of architectures. Future work will explore smoothing with other distributions beyond Gaussian, such as compactly supported distributions, to extend these results to cases involving only local Lipschitz constants of the gradient.

\newpage 
\bibliographystyle{plain}
\bibliography{paper}

\newpage

\onecolumn
\clearpage
\appendix


\addcontentsline{toc}{section}{Appendix} 
\part{Appendix} 

\parttoc 

\section{Proofs of Background Results}
\label{app:background_results}

We begin this section by stating the definitions of convexity and $L$-smoothness as well as mentioning some key identities that we use throughout our proofs, then we provide the proofs of Lemmas~\ref{lem:lemma4_nesterov_modification} and~\ref{lem:fstronglyconvexthensoisfsigma}.

\begin{defn}
\label{defn:background_defs}
Let $f:\mathbb{R}^d\to\mathbb{R}$.
\begin{enumerate}[label={(\alph*)},ref={\thedefn~(\alph*)}]
    \item We say $f$ is convex, if for $\bmx,\bmy\in\mathbb{R}^d$
        \begin{equation}
            f(t\bmx+(1-t)\bmy)\leq tf(\bmx)+(1-t)f(\bmy).
        \end{equation}
    \item We say $f$ is $\gamma$-strongly convex if
        \begin{equation}
            f(t\bmx+(1-t)\bmy)\leq tf(\bmx)+(1-t)f(\bmy)-\frac{\gamma}{2}t(1-t)\|\bmx-\bmy\|^2.
        \end{equation}
    \item We say $f$ is $L$-smooth, if for $\bmx,\bmy\in\mathbb{R}^d$
        \begin{equation}
            |f(\bmy)-f(\bmx)-\langle\nabla f(\bmx),\bmy-\bmx\rangle|\leq\frac{L}{2}\|\bmy-\bmx\|^2.
        \end{equation}
\end{enumerate}
\end{defn}

As in \cite{nesterov2017random}, it is often convenient to represent the gradient of the smoothed function as an integral difference\footnote{This comes from the fact that $\int_{\mathbb{R}^d}f(\bmx)\bmu e^{-\|\bmu\|^2}\;d\bmu=0$.}
\begin{equation}
\label{eqn:firstfinitedifferenceintegral}
    \nabla f_{\sigma}(\bmx)=\frac{2}{\pi^{\frac{d}{2}}\sigma}\int_{\mathbb{R}^d}\Big(f(\bmx+\sigma \bmu)-f(\bmx)\Big)\bmu e^{-\|\bmu\|^2}\;d\bmu.
\end{equation}
Another convenience from \cite{nesterov2017random}, is rewriting the gradient of the original function in an integral
\begin{equation}
    \nabla f(\bmx)=\frac{1}{\pi^{\frac{d}{2}}}\int_{\mathbb{R}^d}\langle\nabla f(\bmx),\bmu\rangle \bmu e^{-\|\bmu\|^2}\;d\bmu.
\end{equation}

We state Lemma 4 from~\cite{nesterov2017random}, which we use to bound the norm of the original function by the norm of the smoothed function.

\begin{lem}[Lemma 4~\cite{nesterov2017random}]
\label{lem:nesterov_lemma_4}
If $f:\mathbb{R}^d\to\mathbb{R}$ is $L$-smooth, then for any $\bmx\in\mathbb{R}^d$ and $\sigma\geq 0$
\begin{equation}
    \|\nabla f(\bmx)\|^2
    \leq 2\|\nabla f_{\sigma}(\bmx)\|^2+\frac{L^2\sigma^2}{4}(6+d)^3.
\end{equation}
\end{lem}

Now, we prove Lemma~\ref{lem:lemma4_nesterov_modification}.

\begin{proof}[Proof of Lemma~\ref{lem:lemma4_nesterov_modification}]
The first result can be shown just by a modification of the proof of Lemma~\ref{lem:nesterov_lemma_4}. Observe
\begin{align}
\begin{split}
    \|\nabla f_{\sigma}(\bmx)\|^2
    &=\left\|\frac{2}{\pi^{\frac{d}{2}}\sigma}\int_{\mathbb{R}^d}\big(f(\bmx+\sigma \bmu)-f(\bmx)\big)\bmu e^{-\|\bmu\|^2}du\right\|^2\\
    &=\left\|\frac{2}{\pi^{\frac{d}{2}}\sigma}\int_{\mathbb{R}^d}\Big[\big(f(\bmx+\sigma \bmu)-f(\bmx)-\sigma\langle\nabla f(\bmx),\bmu\rangle\big)+\sigma\langle\nabla f(\bmx),\bmu\rangle\Big]\bmu e^{-\|\bmu\|^2}d\bmu\right\|^2\\
    &\leq\frac{8}{\pi^{d}\sigma^2}\int_{\mathbb{R}^d}\big(f(\bmx+\sigma \bmu)-f(\bmx)-\sigma\langle\nabla f(\bmx),\bmu\rangle\big)^2\|\bmu\|^2e^{-\|\bmu\|^2}d\bmu+2\|\nabla f(\bmx)\|^2\\
    &\leq\frac{2L^2\sigma^2}{\pi^d}\int_{\mathbb{R}^d}\|\bmu\|^6e^{-\|\bmu\|^2}d\bmu+2\|\nabla f(\bmx)\|^2\\
    &\leq\frac{L^2\sigma^2}{4}(6+d)^3+2\|\nabla f(\bmx)\|^2,
\end{split}
\end{align}
where the last inequality comes from Lemma 1 of \cite{nesterov2017random} which states that for $p\geq 2$
\begin{equation}
    \frac{1}{\pi^{\nicefrac{p}{2}}}\int_{\mathbb{R}^d}\|\bmu\|^pe^{-\|\bmu\|^2}\;d\bmu
    \leq\left(\frac{p+d}{2}\right)^{\nicefrac{p}{2}}.
\end{equation}

For the second part of the lemma, we will assume, without loss of generality, $\sigma\leq\tau$. By Lemma 3 from~\cite{nesterov2017random}, for any $\eta>0$,
\begin{equation}
    \|\nabla f_{\eta}(\bmx)-\nabla f(\bmx)\|
    \leq L\eta\left(\frac{3+d}{2}\right)^{\nicefrac{3}{2}}.
\end{equation}
By Lemma 2.8 of \cite{gsgd}, if $\eta=\sqrt{\tau^2-\sigma^2}$, then $(f_{\sigma})_{\eta}=f_{\tau}$.
So,
\begin{align}
    \|\nabla f_{\tau}(\bmx)-\nabla f_{\sigma}(\bmx)\|
    =\|\nabla (f_{\sigma})_{\eta}-\nabla f_{\sigma}(\bmx)\|
    \leq L\eta\left(\frac{3+d}{2}\right)^{\nicefrac{3}{2}}
    =L\sqrt{\tau^2-\sigma^2}\left(\frac{3+d}{2}\right)^{\nicefrac{3}{2}}.
\end{align}
\end{proof}

\section{Proof of GSmoothSGD Convergence}
\label{app:GSSGD}


The proof of convergence of GSmoothSGD adapts the standard proof of SGD (see, e.g., \cite{sgdconvergence}).
Based on the background results, we have two ways to change between $f_{\sigma}$ and $f_{\tau}$ (where $\tau$ could be 0): switch at the function level (Lemma~\ref{lem:differentsmoothingvalues}) or switch at the gradient level (Lemma~\ref{lem:difference_between_smoothed_gradients}).
Our proof of SGD carefully chooses where to change between smoothing values in order to keep the added smoothing bound as small as possible.

\begin{proof}[Proof of Theorem~\ref{thm:GSSGD}]
Since $E(\|\nabla f_k\|^2)\leq\varbound$, there exists $\varbound_t$ so that $E(\|\nabla f_{k_t,\sigma_{t+1}}\|^2)\leq\varbound_{t+1}$ (see Lemma~\ref{lem:lemma4_nesterov_modification}). 
We begin by repeating typical analysis done in the SGD proof: 
\begin{align}
\begin{split}
    f_{\sigma_{t+1}}(\bmx_{t+1})
    &\stackrel{L-\text{smooth}}{\leq}f_{\sigma_{t+1}}(\bmx_t)+\langle\nabla f_{\sigma_{t+1}}(\bmx_t),\bmx_{t+1}-\bmx_t\rangle+\frac{L}{2}\|\bmx_{t+1}-\bmx_t\|^2\\
    &=f_{\sigma_{t+1}}(\bmx_t)+\langle\nabla f_{\sigma_{t+1}}(\bmx_t),-\lrate\nabla f_{k_t,\sigma_{t+1}}(\bmx_t)\rangle+\frac{L\lrate^2}{2}\|\nabla f_{k_t,\sigma_{t+1}}(\bmx_t)\|^2\\
    &=f_{\sigma_{t+1}}(\bmx_t)-\lrate\langle\nabla f_{\sigma_{t+1}}(\bmx_t),\nabla f_{k_t,\sigma_{t+1}}(\bmx_t)\rangle+\frac{L\lrate^2}{2}\|\nabla f_{k_t,\sigma_{t+1}}(\bmx_t)\|^2
\end{split}
\end{align}
Taking the expectation and using the gradient bound gives
\begin{align}
\begin{split}
    E(f_{\sigma_{t+1}}(\bmx_{t+1}))
    &\leq E(f_{\sigma_{t+1}}(\bmx_t))-\lrate E(\langle\nabla f_{\sigma_{t+1}}(\bmx_t),\nabla f_{k_t,\sigma_{t+1}}(\bmx_t)\rangle)+\frac{L\lrate^2}{2}\varbound_{t+1}.
\end{split}
\end{align}
Repeating the regular SGD proof but for $f_{\sigma_{t+1}}$, we have 
\begin{equation}
    E\big(\langle\nabla f_{\sigma_{t+1}}(\bmx_t),\nabla f_{k_t,\sigma_{t+1}}(\bmx_t)\rangle\big)
    =E\left(\left\langle\nabla f_{\sigma_{t+1}}(\bmx_t),E(\nabla f_{k,\sigma_{t+1}}(\bmx_t)|k_t=k)\right\rangle\right)
    =E(\|\nabla f_{\sigma_{t+1}}(\bmx_t)\|^2).
\end{equation}
This means that
\begin{equation}
\label{eqn:pf_of_gssgd_iterates_decrease}
    E(f_{\sigma_{t+1}}(\bmx_{t+1}))\leq E(f_{\sigma_{t+1}}(\bmx_t))-\lrate E(\|\nabla f_{\sigma_{t+1}}(\bmx_t)\|^2)+\frac{L\lrate^2}{2}\varbound_{t+1}.
\end{equation}
Rearranging gives 
\begin{align}
\begin{split}
    \lrate E(\|\nabla f_{\sigma_{t+1}}(\bmx_t)\|^2)
    &\leq E(f_{\sigma_{t+1}}(\bmx_t)-f_{\sigma_{t+1}}(\bmx_{t+1}))+\frac{L\lrate^2}{2}\varbound_{t+1}\\
    &\leq E(f_{\sigma_{t}}(\bmx_t)-f_{\sigma_{t+1}}(\bmx_{t+1}))+\frac{L\lrate^2}{2}\varbound_{t+1}+\frac{Ld}{4}|\sigma_{t+1}^2-\sigma_t^2|.
\end{split}
\end{align}
Summing over the steps shows (where $\sigma_0=0$ for notational convenience)
\begin{align}
\begin{split}
    \lrate\sum_{t=0}^{T-1}E(\|\nabla f_{\sigma_{t+1}}(\bmx_t)\|^2)
    &\leq\sum_{t=0}^{T-1}E(f_{\sigma_{t}}(\bmx_t)-f_{\sigma_{t+1}}(\bmx_{t+1}))+\frac{L\lrate^2}{2}\sum_{t=0}^{T-1}\varbound_{t+1}+\frac{Ld}{4}\sum_{t=0}^{T-1}|\sigma_{t+1}^2-\sigma_t^2|\\
    &=E(f_{\sigma_0}(\bmx_0)-f_{\sigma_{T}}(\bmx_{T}))+\frac{L\lrate^2}{2}\sum_{t=0}^{T-1}\varbound_{t+1}+\frac{Ld}{4}\sum_{t=0}^{T-1}|\sigma_{t+1}^2-\sigma_t^2|
\end{split}
\end{align}
We know that $f^*\leq f_{\sigma_T}(\bmx_T)$, so 
\begin{align}
\begin{split}
    \lrate\sum_{t=0}^{T-1}E(\|\nabla f_{\sigma_{t+1}}(\bmx_t)\|^2)
    &\leq E(f(\bmx_0)-f^*)+\frac{L\lrate^2}{2}\sum_{t=1}^{T}\varbound_{t}+\frac{Ld}{4}\sum_{t=1}^{T}|\sigma_{t}^2-\sigma_{t-1}^2|\\
    &=f(\bmx_0)-f^*+\frac{L\lrate^2}{2}\sum_{t=1}^{T}\varbound_{t}+\frac{Ld}{4}\sum_{t=1}^{T}|\sigma_{t}^2-\sigma_{t-1}^2|.
\end{split}
\end{align}
Taking the average gives
\begin{align}
\begin{split}
    \sum_{t=0}^{T-1}E(\|\nabla f_{\sigma_{t+1}}(\bmx_t)\|^2)
    &\leq\frac{f(\bmx_0)-f^*}{T\lrate}+\frac{L\lrate}{2T}\sum_{t=1}^{T}\varbound_{t}+\frac{Ld}{4T\lrate}\sum_{t=1}^{T}|\sigma_{t}^2-\sigma_{t-1}^2|\\
    &\stackrel{\lrate<\frac{1}{L}}{\leq}\frac{f(\bmx_0)-f^*}{T\lrate}+\frac{1}{2T}\sum_{t=1}^{T}\varbound_{t}+\frac{d}{4T\lrate^2}\sum_{t=1}^{T}|\sigma_{t}^2-\sigma_{t-1}^2|.
\end{split}
\end{align}
From Lemma~\ref{lem:nesterov_lemma_4} and using $L<\frac{1}{\lrate}$, we have
\begin{align}
\begin{split}
    \sum_{t=0}^{T-1}E(\|\nabla f(\bmx_t)\|^2)
    &\leq\frac{1}{T}\sum_{t=0}^{T-1}\left(2E(\|\nabla f_{\sigma_{t+1}}(\bmx_t)\|^2)+\frac{L^2\sigma_{t+1}^2}{4}(6+d)^3\right)\\
    &\leq\frac{2(f(\bmx_0)-f^*)}{T\lrate}+\frac{1}{T}\sum_{t=1}^{T}\varbound_{t}+\frac{d}{2T\lrate^2}\sum_{t=1}^{T}|\sigma_{t}^2-\sigma_{t-1}^2|+\frac{(6+d)^3}{4T\lrate^2}\sum_{t=1}^{T}\sigma_{t}^2.
\end{split}
\end{align}
Finally, from Lemma~\ref{lem:lemma4_nesterov_modification}, since $E(\|\nabla f_k\|^2)\leq\varbound$, we have that 
\begin{align}
\begin{split}
    E(\|\nabla f_{k_t,\sigma_{t+1}}(\bmx_t)\|^2)
    &\leq 2E(\|\nabla f_{k_t}(\bmx_t)\|^2)+\frac{L^2(6+d)^3}{4}\sigma_{t+1}^2\\
    &\leq 2\varbound+\frac{L^2(6+d)^3}{4}\sigma_{t+1}^2\\
    &\leq 2\varbound+\frac{(6+d)^3}{4\lrate^2}\sigma_{t+1}^2.
\end{split}
\end{align}
Combining the previous two equations yields
\begin{align}
\begin{split}
    \frac{1}{T}\sum_{t=0}^{T-1}E(\|\nabla f(\bmx_t)\|^2)
    &\leq\frac{2(f(\bmx_0)-f^*)}{T\lrate}+\frac{1}{T}\sum_{t=1}^{T}\varbound_{t}+\frac{d}{2T\lrate^2}\sum_{t=1}^{T}|\sigma_{t}^2-\sigma_{t-1}^2|+\frac{(6+d)^3}{4T\lrate^2}\sum_{t=1}^{T}\sigma_{t}^2\\
    &=\frac{2(f(\bmx_0)-f^*)}{T\lrate}+2\varbound+\frac{1}{2T\lrate^2}\sum_{t=1}^{T}\Big(|\sigma_{t}^2-\sigma_{t-1}^2|d+\sigma_{t}^2(6+d)^3\Big).
\end{split}
\end{align}
\end{proof}

We now turn towards proving Proposition~\ref{prop:gssgd_improvement}.
We begin with a formal definition of the basin of attraction.

\begin{defn}
\label{defn:def_basin_of_attraction}
Let $\sigma\geq 0$ and suppose $f_{\sigma}$ has a unique minimizer $\sstar$.
Let $(\sigma_n)_{n=1}^{\infty}\subseteq[\sigma,\infty)$ with $\sigma_t\to\sigma$, define
\begin{align}
    \mathcal{B}\big(\sigma,(\sigma_t)_{t=1}^{\infty};(k_t)_{t=1}^{\infty}\big) &= \{\bmx\in\mathbb{R}^d:\bmx_{t}\to\sstar\text{ where }\bmx_{t+1}=\bmx_{t}-\eta\nabla f_{k_t,\sigma_{t+1}}(\bmx_t)\text{ and }\bmx_0=\bmx\}\\
    \mathcal{B}\big(\sigma,(\sigma_t)_{t=1}^{\infty}\big) &=\bigcap \left\{\mathcal{B}\big(\sigma,(\sigma_t)_{t=1}^{\infty};(k_t)_{t=1}^{\infty}\big):k_t\sim\text{Unif}([K])\right\}\\
    \mathcal{B}(\sigma)&=\bigcup \left\{\mathcal{B}\big(\sigma,(\sigma_t)_{t=1}^{\infty}\big):(\sigma_n)_{n=1}^{\infty}\subseteq[\sigma,\infty)\text{ with }\sigma_t\to\sigma\right\}
\end{align}
We call $\mathcal{B}(\sigma)$ the basin of attraction of $\sstar$.
\end{defn}

While we assume that there are unique minimizers, the proof can be generalized to the setting where the minimizers form a connected, bounded set. Regardless, for simplicity, we assume uniqueness of the minimizers.

The proof of Proposition~\ref{prop:gssgd_improvement} chains together the iterates from traveling though $\sigma_1, \sigma_2, \dots, \sigma_N$ by running GSmoothSGD until the iterate is close enough to $\bmx_{\sigma_n}^{\star}$ so that it is also in the basin of attraction of $\bmx_{\sigma_{n+1}}^{\star}$.
The following proposition provides the proof of this chaining idea.
In particular, Proposition~\ref{prop:gssgd_improvement} is actually a corollary of the following result, which provides how we transition from $\mathcal{B}(\sigma_n)$ to $\mathcal{B}(\sigma_{n+1})$.

\begin{prop}
Assume $f_{\sigma}$ has a unique minimizer $\sstar$ for $0\leq \sigma\leq\Sigma$.
Let $0\leq\tau<\sigma\leq\Sigma$ be such that
\begin{equation}
    \bm{x}^{\star}_{\sigma}\in\mathcal{B}(\tau).
\end{equation}
Suppose further that there is an open set around $\sstar$ contained in $\mathcal{B}(\tau)$.
Then $\mathcal{B}(\sigma)\subseteq\mathcal{B}(\tau)$.
\end{prop}

\begin{proof}
Let $\bmx\in\mathcal{B}(\sigma)$, so there exists $(\alpha_t)_{t=1}^{\infty}\subseteq[\sigma,\infty)$ such that
\begin{equation}
    \bmx_t=\bmx_{t-1}-\eta\nabla f_{k_t,\alpha_{t+1}}(\bmx_{t-1})
\end{equation}
with $\bmx_0=\bmx$ and $k_t\sim\text{Unif}([K])$ satisfies $\bmx_t\to\sstar$.
Let $\bm{k} = (k_1,k_2,...)$ be a realization of samples where $k_t\sim\text{Unif}([K])$.
By assumption, there exists $R>0$ such that
\begin{equation}
    B(\sstar,R)\subseteq\mathcal{B}(\tau).
\end{equation}
Since $\bmx\in\mathcal{B}(\sigma_n)$, there exists $(\alpha_t)_{t=1}^{\infty}\subseteq[\sigma,\infty)$ so that if
\begin{equation}
    \bmx_t=\bmx_{t-1}-\eta\nabla f_{k_t,\alpha_{t+1}}(\bmx_{t-1})
\end{equation}
with $\bmx_0=\bmx$ then $\bmx_t\to\sstar$.
As such, there exists $N\in\mathbb{N}$ so that $\|\bmx_N-\sstar\|<R$.
Hence,
\begin{equation}
    \bmx_N\in B(\sstar,R)\subseteq\mathcal{B}(\tau).
\end{equation}
This means that there exists $(\beta_t)_{t=1}^{\infty}\subseteq[\tau,\infty)$ with $\beta_t\to\tau$ such that if
\begin{equation}
    \bmy_t=\bmy_{t-1}-\eta\nabla f_{k_{t+N},\beta_{t+1}}(\bmy_{t-1})
\end{equation}
with $\bmy_0=\bmx_N$ then $\bmy_t\to\bmx_{\tau}^{\star}$.
Define
\begin{equation}
    \gamma_t=\left\{\begin{array}{ll}
        \alpha_t\qquad&\text{for }t=1,...,N-1\\
        \beta_t&\text{for }t\geq N
    \end{array}\right.
\end{equation}
and
\begin{equation}
    \bmx_t=\bmx_{t-1}-\eta\nabla f_{k_t,\gamma_{t+1}}(\bmx_{t-1})
\end{equation}
with $\bmx_0=\bmx$.
Then $\bmx_t\to\bmx_{\tau}^{\star}$ and $(\gamma_t)_{t=1}^{\infty}\subseteq[\tau,\infty)$ with $\gamma_t\to\tau$.
Therefore,
\begin{equation}
    x\in\mathcal{B}(\tau,(\gamma_t)_{t=1}^{\infty};\bm{k})\subseteq\mathcal{B}(\tau).
\end{equation}
Since this can be done for any $\bm{k}$, the proof is complete.
\end{proof}

Practically, this does not provide a sequence of smoothing parameters because $(\beta_t)_{t=1}^{\infty}$ (and hence $(\gamma_t)_{t=1}^{\infty}$) depends on $\bm{k}$.
Next, we state a more general corollary of the previous proposition compared to Proposition~\ref{prop:gssgd_improvement}.
In fact, Proposition~\ref{prop:gssgd_improvement} is just a particular example of the following corollary that leverages the convexity of $f_{\Sigma}$ to increase the initial basin of attraction.

\begin{coro}
Suppose there exists an increasing sequence, $(\sigma_n)_{n=1}^{N}$, starting at 0 such that for $n=2,\dots,N$, $f_{\sigma_n}$ has a unique minimizer $\bmx_{\sigma_n}^{\star}$ and there exists $R_n>0$ so that
\begin{equation}
    B(\bm{x}^{\star}_{\sigma_{n}},R_n)\subseteq\mathcal{B}(\sigma_{n-1}).
\end{equation}
Then $\mathcal{B}(\Sigma)\subseteq\mathcal{B}(0)$.
\end{coro}

\begin{proof}
The previous proposition shows that $\mathcal{B}(\sigma_{n})\subseteq\mathcal{B}(\sigma_{n-1})$ for each $n=2,...,N$.
Therefore,
\begin{equation}
    \mathcal{B}(\Sigma)
    =\mathcal{B}(\sigma_1)
    \subseteq\mathcal{B}(\sigma_2)
    \subseteq\cdots
    \subseteq\mathcal{B}(\sigma_N)
    =\mathcal{B}(0).
\end{equation}
\end{proof}

Finally, we prove Proposition~\ref{prop:gssgd_improvement}.

\begin{proof}[Proof of Proposition~\ref{prop:gssgd_improvement}]
Since each $f_{k,\Sigma}$ is convex and they share a minimizer, $f_{\Sigma}$ is also convex with the same minimizer.
As such, $\mathcal{B}(\Sigma)=\mathbb{R}^d$.
By the previous corollary, we have
\begin{equation}
    \mathbb{R}^d
    =\mathcal{B}(\Sigma)
    \subseteq\mathcal{B}(0)
    \subseteq\mathbb{R}^d.
\end{equation}
\end{proof}

\section{Proof of Convergence of GSmoothAdam}
\label{app:proofs_gsmoothadam}

The proof that GSmoothAdam converges almost surely adapts the proof that Adam does from~\cite{he2023convergence}.\
In order to prove Theorem~\ref{thm:gsadam}, we needed to replicate almost the entirety of two of their other key results (Theorems 4 and 9). For this section, we adopt $\mathcal{F}_t$ as the notation for the sigma algebra generated by $\bmx_0,...,\bmx_t$.

As the first step to proving Theorem~\ref{thm:gsadam}, we provide the following smoothed version of Lemma 19 from~\cite{he2023convergence}.
This proof relies on Lemmas 16, 17 and 18 from~\cite{he2023convergence}, which do not need to be modified for smoothing.
Lemma 16 provides a uniform bound on $\bmm_t$, $\bmv_t$, and the Adam update to $\bmx_t$, where the subtle difference between the smoothed and unsmoothed results are in this bound.
In particular, the $M$ is the original statement, which bounds $\|\nabla f(\bmx)\|^2$, becomes
\begin{equation}
    M:=2\lambda+\frac{L^2(6+d)^3}{4}\max_t\sigma_t^2,
\end{equation}
using Lemma~\ref{lem:difference_between_smoothed_gradients} to bound $\|\nabla f_{\sigma}(\bmx)\|^2$.
Since we assume that $\sigma_t\to 0$, the sequence $(\sigma_t)$ is bounded and hence $M\in\mathbb{R}$.
We use this definition of $M$ throughout this section.
Lemmas 17 and 18 provide technical bounds needed in the proof of Lemma 19 that do not need to be adapted because $\sigma$ is fixed in both of them.

\begin{lem}[Lemma 19 \cite{he2023convergence}]
\label{lem:adam_lemma19}
Let $(\bmx_t)_{t\geq 1}$, $(\bmm_t)_{t\geq 1}$, and $(\bmv_t)_{t\geq 1}$ be the sequences generated by GSmoothAdam.
Let $f$ be $L$-smooth and $f^*$ denote the minimum of $f$.
Assume $E(f_k(\bmx))=f(\bmx)$ and $E(\nabla f_k(\bmx))=\nabla f(\bmx)$.
Suppose that $E(\|\nabla f_k\|^2)\leq\lambda$ for any $k\in[K]$.
Then for all $t\geq 1$, we have
\begin{equation}
    E\left(\left\langle
        \nabla f_{\sigma_{t+1}}(\bmx_t),
        \frac{\bmm_{t+1}}{\sqrt{\bmv_{t+1}+\epsilon}}
    \right\rangle\right)
    \geq\sum_{i=1}^{t+1}\prod_{j=i+1}^{t+1}\beta_jD_i
    +\sum_{i=1}^{t+1}\prod_{j=i+1}^{t+1}\beta_j\widetilde{D}_i\sqrt{|\sigma_{i}^2-\sigma_{i+1}^2|},
\end{equation}
where
\begin{equation}
    D_i
    =-L\beta\eta_{i}E\left(\left\|\frac{\bmm_{i}}{\sqrt{\bmv_{i}+\epsilon}}\right\|^2\right)
    +\frac{1-\beta_{i+1}}{\sqrt{M^2+\epsilon}}E(\|\nabla f_{\sigma_{i+1}}(\bmx_i)\|^2)
    -\frac{\sqrt{d}M^4}{\epsilon^{\nicefrac{3}{2}}}(1-\theta_{i+1})
\end{equation}
and
\begin{equation}
    \widetilde{D}_i=-\frac{\beta_i LM}{\sqrt{\epsilon}}\left(\frac{3+d}{2}\right)^{\nicefrac{3}{2}}.
\end{equation}
\end{lem}

\begin{proof}
We repeat the proof from \cite{he2023convergence}, but include the necessary changes for smoothing.
Let
\begin{align}
\label{eqn:lemma_19_0}
    \Theta_{t+1}
    &=E\left(\left\langle
        \nabla f_{\sigma_{t+1}}(\bmx_t),
        \frac{\bmm_{t+1}}{\sqrt{\bmv_{t+1}+\epsilon}}
    \right\rangle\right)\\
    &=\underbrace{
        E\left(\left\langle
            \nabla f_{\sigma_{t+1}}(\bmx_t),
            \frac{\bmm_{t+1}}{\sqrt{\bmv_{t}+\epsilon}}
        \right\rangle\right)
    }_{I}
    +\underbrace{
        E\left(\left\langle
            \nabla f_{\sigma_{t+1}}(\bmx_t),
            \frac{\bmm_{t+1}}{\sqrt{\bmv_{t+1}+\epsilon}}-\frac{\bmm_{t+1}}{\sqrt{\bmv_{t}+\epsilon}}
        \right\rangle\right)
    }_{II}.
\end{align}
Focusing on $I$, we have
\begin{align}
\label{eqn:lemma_19_1}
    &E\left(\left.
        \left\langle
            \nabla f_{\sigma_{t+1}}(\bmx_t),
            \frac{\bmm_{t+1}}{\sqrt{\bmv_{t+1}+\epsilon}}
        \right\rangle
    \right|\mathcal{F}_t\right)\\
    &=E\left(\left.
        \left\langle
            \nabla f_{\sigma_{t+1}}(\bmx_t),
            \frac{\beta_{t+1}\bmm_{t}+(1-\beta_{t+1})\nabla f_{k_t,\sigma_{t+1}}(\bmx_t)}{\sqrt{\bmv_{t+1}+\epsilon}}
        \right\rangle
    \right|\mathcal{F}_t\right)\\
    &=\beta_{t+1}\left\langle
        \nabla f_{\sigma_{t+1}}(\bmx_t),
        \frac{\bmm_{t}}{\sqrt{\bmv_{t+1}+\epsilon}}
    \right\rangle
    +(1-\beta_{t+1})\left\langle
        \nabla f_{\sigma_{t+1}}(\bmx_t),
        \frac{\nabla f_{\sigma_{t+1}}(\bmx_t)}{\sqrt{\bmv_{t+1}+\epsilon}}
    \right\rangle\\
    &=\beta_{t+1}\left\langle
        \nabla f_{\sigma_{t+1}}(\bmx_t),
        \frac{\bmm_{t}}{\sqrt{\bmv_{t+1}+\epsilon}}
    \right\rangle
    +(1-\beta_{t+1})\left\|
        \frac{(\nabla f_{\sigma_{t+1}}(\bmx_t))^2}{\sqrt{\bmv_{t+1}+\epsilon}}
    \right\|_1\\
    &=\beta_{t+1}\left\langle
        \nabla f_{\sigma_{t}}(\bmx_{t-1}),
        \frac{\bmm_{t}}{\sqrt{\bmv_{t+1}+\epsilon}}
    \right\rangle
    +(1-\beta_{t+1})\left\|
        \frac{(\nabla f_{\sigma_{t+1}}(\bmx_t))^2}{\sqrt{\bmv_{t+1}+\epsilon}}
    \right\|_1\\
    &\qquad\qquad\qquad-\beta_{t+1}\left\langle
        \nabla f_{\sigma_{t+1}}(\bmx_{t})-\nabla f_{\sigma_{t}}(\bmx_{t-1}),
        \frac{\bmm_{t}}{\sqrt{\bmv_{t+1}+\epsilon}}
    \right\rangle,
\end{align}
where $(\nabla f_{\sigma_{t+1}}(\bmx_t))^2$ is done coordinate-wise.
Focusing on the last term of the previous equation, we have
\begin{align}
    -\beta_{t+1}\left\langle
        \nabla f_{\sigma_{t+1}}(\bmx_{t})-\nabla f_{\sigma_{t}}(\bmx_{t-1}),
        \frac{\bmm_{t}}{\sqrt{\bmv_{t+1}+\epsilon}}
    \right\rangle
    &\geq-\beta
    \|\nabla f_{\sigma_{t+1}}(\bmx_{t})-\nabla f_{\sigma_{t}}(\bmx_{t-1})\|
    \left\|\frac{\bmm_{t}}{\sqrt{\bmv_{t+1}+\epsilon}}\right\|.
\end{align}
Note that
\begin{align}
    \|\nabla f_{\sigma_{t+1}}(\bmx_{t})-\nabla f_{\sigma_{t}}(\bmx_{t-1})\|
    &\leq
        \|\nabla f_{\sigma_{t}}(\bmx_{t})-\nabla f_{\sigma_{t}}(\bmx_{t-1})\|
        +\|\nabla f_{\sigma_{t+1}}(\bmx_{t})-\nabla f_{\sigma_{t}}(\bmx_{t})\|\\
    &\leq
        L\|\bmx_t-\bmx_{t-1}\|
        +L\left(\frac{3+d}{2}\right)^{\nicefrac{3}{2}}\sqrt{|\sigma_t^2-\sigma_{t+1}^2|}
\end{align}
using Lemma~\ref{lem:difference_between_smoothed_gradients} and the fact that $f_{\sigma_t}$ is L-smooth.
So,
\begin{align}
    &-\beta_{t+1}\left\langle
        \nabla f_{\sigma_{t+1}}(\bmx_{t})-\nabla f_{\sigma_{t}}(\bmx_{t-1}),
        \frac{\bmm_{t}}{\sqrt{\bmv_{t+1}+\epsilon}}
    \right\rangle\\
    &\geq-\beta
    L\|\bmx_t-\bmx_{t-1}\|
    \left\|\frac{\bmm_{t}}{\sqrt{\bmv_{t+1}+\epsilon}}\right\|
    -\beta
    L\left(\frac{3+d}{2}\right)^{\nicefrac{3}{2}}\sqrt{|\sigma_t^2-\sigma_{t+1}^2|}
    \left\|\frac{\bmm_{t}}{\sqrt{\bmv_{t+1}+\epsilon}}\right\|\\
    &=-\beta
    L\eta_{t}
    \left\|\frac{\bmm_{t}}{\sqrt{\bmv_{t+1}+\epsilon}}\right\|^2
    -\beta
    L\left(\frac{3+d}{2}\right)^{\nicefrac{3}{2}}\sqrt{|\sigma_t^2-\sigma_{t+1}^2|}
    \left\|\frac{\bmm_{t}}{\sqrt{\bmv_{t+1}+\epsilon}}\right\|.
\end{align}
Combining this with where we were in \eqref{eqn:lemma_19_1}, we have
\begin{align}
    &E\left(\left.
        \left\langle
            \nabla f_{\sigma_{t+1}}(\bmx_t),
            \frac{\bmm_{t+1}}{\sqrt{\bmv_{t+1}+\epsilon}}
        \right\rangle
    \right|\mathcal{F}_t\right)\\
    &\geq\beta_{t+1}\left\langle
        \nabla f_{\sigma_{t}}(\bmx_{t-1}),
        \frac{\bmm_{t}}{\sqrt{\bmv_{t+1}+\epsilon}}
    \right\rangle
    +(1-\beta_{t+1})\left\|
        \frac{(\nabla f_{\sigma_{t+1}}(\bmx_t))^2}{\sqrt{\bmv_{t+1}+\epsilon}}
    \right\|_1\\
    &\qquad\qquad\qquad-\beta
        L\eta_{t}
        \left\|\frac{\bmm_{t}}{\sqrt{\bmv_{t+1}+\epsilon}}\right\|^2
        -\beta_{t+1}
        L\left(\frac{3+d}{2}\right)^{\nicefrac{3}{2}}\sqrt{|\sigma_t^2-\sigma_{t+1}^2|}
        \left\|\frac{\bmm_{t}}{\sqrt{\bmv_{t+1}+\epsilon}}\right\|.
\end{align}
So,
\begin{align}
    I
    &=E\left(E\left(\left.
        \left\langle
            \nabla f_{\sigma_{t+1}}(\bmx_t),
            \frac{\bmm_{t+1}}{\sqrt{\bmv_{t+1}+\epsilon}}
        \right\rangle
    \right|\mathcal{F}_t\right)\right)\\
    &\geq
        -\beta_{t+1}E\left(\left\langle
            \nabla f_{\sigma_{t}}(\bmx_{t-1}),
            \frac{\bmm_{t}}{\sqrt{\bmv_{t+1}+\epsilon}}
        \right\rangle\right)
        +(1-\beta_{t+1})E\left(\left\|
            \frac{(\nabla f_{\sigma_{t+1}}(\bmx_t))^2}{\sqrt{\bmv_{t+1}+\epsilon}}
        \right\|_1\right)\\
    &\qquad\qquad\qquad
        -\beta L\eta_{t}E\left(\left\|\frac{\bmm_{t}}{\sqrt{\bmv_{t+1}+\epsilon}}\right\|^2\right)\\
    &\qquad\qquad\qquad
        -\beta_{t+1}L\left(\frac{3+d}{2}\right)^{\nicefrac{3}{2}}\sqrt{|\sigma_t^2-\sigma_{t+1}^2|}
        E\left(\left\|\frac{\bmm_{t}}{\sqrt{\bmv_{t+1}+\epsilon}}\right\|\right)\\
    &\geq
        \beta_{t+1}\Theta_{t}
        +\frac{1-\beta_{t+1}}{\sqrt{M^2+\epsilon}}E\left(\left\|
            \nabla f_{\sigma_{t+1}}(\bmx_t)
        \right\|^2\right)
        -\beta L\eta_{t}E\left(\left\|\frac{\bmm_{t}}{\sqrt{\bmv_{t+1}+\epsilon}}\right\|^2\right)\\
    &\qquad\qquad\qquad
        -\frac{\beta_{t+1}LM}{\sqrt{\epsilon}}\left(\frac{3+d}{2}\right)^{\nicefrac{3}{2}}\sqrt{|\sigma_t^2-\sigma_{t+1}^2|}
\end{align}
where the last inequality used Lemmas 16 and 17 from \cite{he2023convergence}.

Now focusing on $II$ in \eqref{eqn:lemma_19_0},
\begin{align}
    II
    &=E\left(\left\langle
        \nabla f_{\sigma_{t+1}}(\bmx_t),
        \frac{\bmm_{t+1}}{\sqrt{\bmv_{t+1}+\epsilon}}-\frac{\bmm_{t+1}}{\sqrt{\bmv_{t}+\epsilon}}
    \right\rangle\right)\\
    &=-E\left(\left\langle
        \nabla f_{\sigma_{t+1}}(\bmx_t),
        \frac{\bmm_{t+1}}{\sqrt{\bmv_{t}+\epsilon}}-\frac{\bmm_{t+1}}{\sqrt{\bmv_{t+1}+\epsilon}}
    \right\rangle\right)\\
    &\geq E\left(
        \|\nabla f_{\sigma_{t+1}}(\bmx_t)\|
        \left\|\frac{\bmm_{t+1}}{\sqrt{\bmv_{t}+\epsilon}}-\frac{\bmm_{t+1}}{\sqrt{\bmv_{t+1}+\epsilon}}\right\|
    \right)\\
    &\geq-\frac{\sqrt{d}M^4}{\epsilon^{\frac{3}{2}}}(1-\theta_t)
\end{align}
by Lemma 18 from \cite{he2023convergence}.

Therefore, using the definitions of $D_t$ and $\widetilde{D}_t$ from the statement of the lemma,
\begin{align}
    \Theta_{t+1}
    &\geq\beta_{t+1}\Theta_{t}
        +\frac{1-\beta_{t+1}}{\sqrt{M^2+\epsilon}}E\left(\left\|
            \nabla f_{\sigma_{t+1}}(\bmx_t)
        \right\|^2\right)
        -\beta L\eta_{t}E\left(\left\|\frac{\bmm_{t}}{\sqrt{\bmv_{t+1}+\epsilon}}\right\|^2\right)\\
    &\qquad\qquad\qquad
        -\frac{\sqrt{d}M^4}{\epsilon^{\frac{3}{2}}}(1-\theta_{t+1})
        -\frac{\beta_{t+1}LM}{\sqrt{\epsilon}}\left(\frac{3+d}{2}\right)^{\nicefrac{3}{2}}\sqrt{|\sigma_t^2-\sigma_{t+1}^2|}\\
    &=\beta_{t+1}\Theta_{t}
        +D_t
        +\widetilde{D}_t\sqrt{|\sigma_t^2-\sigma_{t+1}^2|}\\
\end{align}
Recursively applying this inequality, we see that
\begin{align}
    \Theta_{t+1}
    &\geq\prod_{i=1}^{t+1}\beta_i\Theta_0
        +\sum_{i=1}^{t+1}\prod_{j=i+1}^{t+1}\beta_jD_i
        +\sum_{i=1}^{t+1}\prod_{j=i+1}^{t+1}\beta_j\widetilde{D}_i\sqrt{|\sigma_i^2-\sigma_{i+1}^2|}\\
    &=
        \sum_{i=1}^{t+1}\prod_{j=i+1}^{t+1}\beta_jD_i
        +\sum_{i=1}^{t+1}\prod_{j=i+1}^{t+1}\beta_j\widetilde{D}_i\sqrt{|\sigma_i^2-\sigma_{i+1}^2|}
\end{align}
since $\Theta_0=0$ and we used $\prod_{j=t+2}^{t+1}\beta_{t}=1$ for notational convenience.
\end{proof}

The next result, which is the GSmoothAdam analogue of Theorem~\ref{thm:GSSGD} for GmoothSGD, shows that the minimum gradient of the iterates converges to 0.

\begin{thm}[Theorem 4 from \cite{he2023convergence}]
\label{thm:adam_theorem4}
Let $f$ be $L$-smooth and $f^*$ denote the minimum of $f$.
Assume $E(f_k(\bmx))=f(\bmx)$ and $E(\nabla f_k(\bmx))=\nabla f(\bmx)$.
Suppose that $E(\|\nabla f_k\|^2)\leq\lambda$ for any $k\in[K]$ and $(\alpha_t)_{t\geq 1}$ is a non-increasing real sequence.
Let $(\bmx_t)_{t\geq 1}$ be generated by GSmoothAdam with $\eta_t=\Theta(\alpha_t)$ (i.e., there exist $C_0,\widetilde{C_0}>0$ such that $C_0\alpha_t\leq\eta_t\leq\widetilde{C_0}\alpha_t$). Then for $T\geq 1$, we have
\begin{multline}
    \min_{1\leq t\leq T}\Big(E(\|\nabla f(\bmx_t)\|^2)\Big)
    \sum_{t=1}^{T}\eta_t
    \leq
        2C_1(f_{\sigma_1}(\bmx_1)-f^*)
        +2C_2\sum_{t=1}^{T}\eta_t(1-\theta_t)
        +2C_3\sum_{t=1}^{T}\eta_t^2\\
        +\frac{Ld\sqrt{M^2+\epsilon}}{2(1-\beta)}\sum_{t=1}^{T}|\sigma_t^2-\sigma_{t+1}^2|
        +\frac{L^2(6+d)^3}{4}\sum_{t=1}^{T}\eta_t\sigma_{t+1}^2,
\end{multline}
where
\begin{equation}
    C_1=\frac{\sqrt{M^2+\epsilon}}{1-\beta},\;
    C_2=\frac{\widetilde{C_0}\sqrt{d}M^4\sqrt{M^2+\epsilon}}{\epsilon^{\nicefrac{3}{2}}C_0(1-\beta)^2},\;
    C_3=\frac{2\widetilde{C_0}^2LM^2\sqrt{M^2+\epsilon}}{\epsilon C_0^2(1-\beta)^2}.
\end{equation}
\end{thm}

\begin{proof}
Again, we will repeat the proof from \cite{he2023convergence}, but with the necessary changes for smoothing.
Repeating what was done at the beginning of the GSmoothSGD proof, we have
\begin{align}
    f_{\sigma_{t+1}}(\bmx_{t+1})
    &\leq f_{\sigma_{t+1}}(\bmx_t)+\langle\nabla f_{\sigma_{t+1}}(\bmx_t),\bmx_{t+1}-\bmx_t\rangle+\frac{L}{2}\|\bmx_{t+1}-\bmx_t\|^2\\
    &=
        f_{\sigma_{t+1}}(\bmx_t)
        -\eta_{t+1}\left\langle\nabla f_{\sigma_{t+1}}(\bmx_t),\frac{\bmm_{t+1}}{\sqrt{\bmv_{t+1}+\epsilon}}\right\rangle
        +\frac{L\eta_{t+1}^2}{2}\left\|\frac{\bmm_{t+1}}{\sqrt{\bmv_{t+1}+\epsilon}}\right\|^2.
\end{align}
This means
\begin{align}
    E(f_{\sigma_{t+1}}(\bmx_{t+1}))
    &\leq
        E(f_{\sigma_{t+1}}(\bmx_t))
        -\eta_{t+1}E\left(\left\langle\nabla f_{\sigma_{t+1}}(\bmx_t),\frac{\bmm_{t+1}}{\sqrt{\bmv_{t+1}+\epsilon}}\right\rangle\right)\\
    &\qquad
        +\frac{L\eta_{t+1}^2}{2}E\left(\left\|\frac{\bmm_{t+1}}{\sqrt{\bmv_{t+1}+\epsilon}}\right\|^2\right)\\
    &\leq
    E(f_{\sigma_{t+1}}(\bmx_t))
        -\eta_{t+1}\sum_{i=1}^{t+1}\prod_{j=i+1}^{t+1}\beta D_i
        -\eta_{t+1}\sum_{i=1}^{t+1}\prod_{j=i+1}^{t+1}\beta\widetilde{D}_i\sqrt{|\sigma_i^2-\sigma_{i+1}^2|}\\
    &\qquad
        +\frac{L\eta_{t+1}^2}{2}E\left(\left\|\frac{\bmm_{t+1}}{\sqrt{\bmv_{t+1}+\epsilon}}\right\|^2\right).
\end{align}
Now, repeating the analysis that was done in \cite{he2023convergence}, the smoothed version of their equation (23) is
\begin{multline}
    E(f_{\sigma_{t+1}}(\bmx_{t+1}))
    \leq
        E(f_{\sigma_{t+1}}(\bmx_t))
        -\frac{(1-\beta)\eta_{t+1}}{\sqrt{M^2+\epsilon}}E(\|\nabla f_{\sigma_{t+1}}(\bmx_t)\|^2)
        +\frac{\beta L\eta_{t+1}M^2}{\epsilon}\sum_{i=1}^{t+1}\beta^{t-i}\eta_{i}\\
        +\frac{L\eta_{t+1}^2M^2}{2\epsilon}
        +\frac{\sqrt{d}M^4\eta_{t+1}}{\epsilon^{\nicefrac{3}{2}}}\sum_{i=1}^{t+1}\beta^{t-i}(1-\theta_{i})
        -\eta_{t+1}\sum_{i=1}^{t+1}\prod_{j=i+1}^{t+1}\beta_j\widetilde{D}_i\sqrt{|\sigma_i^2-\sigma_{i+1}^2|}.
\end{multline}
As such, rearranging and summing (as in \cite{he2023convergence}),
\begin{multline}
    \frac{(1-\beta)}{\sqrt{M^2+\epsilon}}\sum_{t=1}^{T}\eta_{t+1}E(\|\nabla f_{\sigma_{t+1}}(\bmx_t)\|^2)
    \leq
        \sum_{t=1}^{T}\Big(E(f_{\sigma_{t+1}}(\bmx_t))-E(f_{\sigma_{t+1}}(\bmx_{t+1})\Big)\\
        +\frac{L\eta_{t+1}^2M^2}{2\epsilon}
        +\frac{\beta L\eta_{t+1}M^2}{\epsilon}\sum_{i=1}^{t+1}\beta^{t-i}\eta_{i}\\
        +\frac{\sqrt{d}M^4\eta_{t+1}}{\epsilon^{\nicefrac{3}{2}}}\sum_{i=1}^{t}\beta^{t-i}(1-\theta_{i})
        -\eta_{t+1}\sum_{i=1}^{t}\prod_{j=i+1}^{t}\beta_j\widetilde{D}_i\sqrt{|\sigma_i^2-\sigma_{i+1}^2|}.
\end{multline}
Now,
\begin{align}
    E(f_{\sigma_{t+1}}(\bmx_t))-E(f_{\sigma_{t+1}}(\bmx_{t+1}))
    \leq E(f_{\sigma_t}(\bmx_t)-f_{\sigma_{t+1}}(\bmx_{t+1}))+\frac{Ld}{4}|\sigma_t^2-\sigma_{t+1}^2|,
\end{align}
which means
\begin{multline}
    \frac{(1-\beta)}{\sqrt{M^2+\epsilon}}\sum_{t=1}^{T}\eta_{t+1}E(\|\nabla f_{\sigma_{t+1}}(\bmx_t)\|^2)
    \leq
        f_{\sigma_1}(\bmx_1)-f^*
        +\frac{L\eta_k^2M^2}{2\epsilon}\\
        +\frac{\beta L\eta_{t+1}M^2}{\epsilon}\sum_{i=1}^t\beta^{t-i}\eta_{i-1}
        +\frac{\sqrt{d}M^4\eta_{t+1}}{\epsilon^{\nicefrac{3}{2}}}\sum_{i=1}^{t+1}\beta^{t-i}(1-\theta_{i})\\
        -\eta_{t+1}\sum_{i=1}^{t}\prod_{j=i+1}^{t}\beta_j\widetilde{D}_i\sqrt{|\sigma_i^2-\sigma_{i+1}^2|}
        +\frac{Ld}{4}\sum_{t=1}^{T}|\sigma_t^2-\sigma_{t+1}^2|.
\end{multline}
Again repeating the analysis in \cite{he2023convergence} (and carrying along the additional additive $\sigma$ terms), our version of their equation (30) is
\begin{multline}
    \sum_{t=1}^{T}\eta_{t+1}E(\|\nabla f_{\sigma_{t+1}}(\bmx_t)\|^2)
    \leq
        C_1(f_{\sigma_1}(\bmx_1)-f^*)
        +C_2\sum_{t=1}^{T}\eta_{t+1}(1-\theta_{t+1})
        +C_3\sum_{t=1}^{T}\eta_{t+1}^2\\
        -\frac{\sqrt{M^2+\epsilon}}{1-\beta}\sum_{t=1}^{T}\sum_{i=1}^{t}\prod_{j=i+1}^{t}\eta_{t+1}\beta_j\widetilde{D}_i\sqrt{|\sigma_i^2-\sigma_{i+1}^2|}
        +\frac{Ld\sqrt{M^2+\epsilon}}{4(1-\beta)}\sum_{t=1}^{T}|\sigma_t^2-\sigma_{t+1}^2|.
\end{multline}
Although it will improve the estimate, for simplicity since all of the terms of the fourth term of the previous equation are positive, we have
\begin{multline}
\label{eqn:thm4_1}
    \sum_{t=1}^{T}\eta_{t+1}E(\|\nabla f_{\sigma_{t+1}}(\bmx_t)\|^2)
    \leq
        C_1(f_{\sigma_1}(\bmx_1)-f^*)
        +C_2\sum_{t=1}^{T}\eta_{t+1}(1-\theta_{t+1})
        +C_3\sum_{t=1}^{T}\eta_{t+1}^2\\
        +\frac{Ld\sqrt{M^2+\epsilon}}{4(1-\beta)}\sum_{t=1}^{T}|\sigma_t^2-\sigma_{t+1}^2|.
\end{multline}

Now, applying Lemma 4 from \cite{nesterov2017random} and finishing the proof from \cite{he2023convergence}, we have
\begin{align}
\label{eqn:thm4_2}
    &\min_{1\leq t\leq T}\Big(E(\|\nabla f(\bmx_t)\|^2)\Big)\sum_{t=1}^{T}\eta_{t+1}\\
    &\leq\min_{1\leq t\leq T}\left(
        2E(\|\nabla f_{\sigma_{t+1}}(\bmx_t)\|^2)
        +\frac{L^2\sigma_{t+1}^2(6+d)^3}{4}
    \right)\sum_{t=1}^{T}\eta_{t+1}\\
    &=\sum_{t=1}^{T}\eta_{t+1}\min_{1\leq t\leq T}\left(
        2E(\|\nabla f_{\sigma_{t+1}}(\bmx_t)\|^2)
        +\frac{L^2\sigma_{t+1}^2(6+d)^3}{4}
    \right)\\
    &\leq\sum_{t=1}^{T}\eta_{t+1}\left(
        2E(\|\nabla f_{\sigma_{t+1}}(\bmx_t)\|^2)
        +\frac{L^2\sigma_{t+1}^2(6+d)^3}{4}
    \right)\\
    &=
        2\sum_{t=1}^{T}\eta_{t+1}E(\|\nabla f_{\sigma_{t+1}}(\bmx_t)\|^2)
        +\frac{L^2(6+d)^3}{4}\sum_{t=1}^{T}\eta_{t+1}\sigma_{t+1}^2.
\end{align}
Combining \eqref{eqn:thm4_1} and \eqref{eqn:thm4_2} gives the result.
\end{proof}

The next result provides almost sure convergence of GSmoothAdam.
The proof primarily relies on the analysis in the proof of the previous theorem.

\begin{thm}[Theorem 9 from \cite{he2023convergence}]
\label{thm:adam_theorem9}
Let $f=\frac{1}{K}\sum_{k=1}^{K}f_{k}$ be $L$-smooth and $f^*$ denote the minimum of $f$.
Suppose that $E(\|\nabla f_k\|^2)\leq\lambda$ for any $k\in[K]$.
Let $(\bmx_t)_{t\geq 1}$ be generated by GSmoothAdam.
Suppose
\begin{equation}
    \sum_{t=1}^{\infty}\frac{\eta_t}{\sum_{i=1}^{t-1}\eta_i}=\infty,
    \quad\sum_{t=1}^{\infty}\eta_t^2<\infty,
    \quad\sum_{t=1}^{\infty}\eta_t(1-\theta_t)<\infty,
\end{equation}
and $\eta_t$ is decreasing.
If $\sum_{t=1}^{\infty}|\sigma_t^2-\sigma_{t+1}^2|<\infty$, then for any $T\geq 1$, we have
\begin{equation}
    \min_{1\leq t\leq T}\|\nabla f_{\sigma_{t+1}}(\bmx_t)\|^2=o\left(\frac{1}{\sum_{t=1}^T\eta_t}\right)\;\text{a.s.}
\end{equation}
\end{thm}

\begin{proof}
The proof is exactly the same, just replace $f(\bmx^k)$ with $f_{\sigma_{t+1}}(\bmx_t)$.
Also, in order to use Lebesgue's Monotone Convergence Theorem, we need to use our assumption about the summability of $|\sigma_t^2-\sigma_{t+1}^2|$.
\end{proof}

With most of the heavy lifting completed in the proofs of the previous results, we are ready to show that the iterates from GSmoothAdam converge to a stationary point a.s. in $\bmx$.

\begin{proof}[Proof of Theorem~\ref{thm:gsadam}]
Once again, this proof follows the same structure as in \cite{he2023convergence} with changes for smoothing.
Since $f$ satisfies the assumptions of Theorem~\ref{thm:adam_theorem9}, from the proof of Theorem~\ref{thm:adam_theorem9} we have that
\begin{equation}
    \sum_{t=1}^{\infty}\eta_{t+1}\|\nabla f_{\sigma_{t+2}}(\bmx_{t+1})\|^2<\infty\text{ a.s.}
\end{equation}
Since $f$ is $L$-smooth,
\begin{align}
    \Big|\|\nabla f_{\sigma_{t+2}}(\bmx_{t+1})\|-\|\nabla f_{\sigma_{t+1}}(\bmx_{t})\|\Big|
    &\leq\|\nabla f_{\sigma_{t+2}}(\bmx_{t+1})-\nabla f_{\sigma_{t+1}}(\bmx_{t})\|\\
    &\leq L\|\bmx_{t+1}-\bmx_t\|+L\left(\frac{3+d}{2}\right)^{\nicefrac{3}{2}}\sqrt{|\sigma_{t+2}^2-\sigma_{t+1}^2|}\\
    &=L\eta_{t+1}\left\|\frac{\bmm_{t+1}}{\sqrt{\bmv_{t+1}+\epsilon}}\right\|+L\left(\frac{3+d}{2}\right)^{\nicefrac{3}{2}}\sqrt{|\sigma_{t+2}^2-\sigma_{t+1}^2|}\\
    &\leq\frac{LM}{\sqrt{\epsilon}}\eta_{t+1}+L\left(\frac{3+d}{2}\right)^{\nicefrac{3}{2}}\sqrt{|\sigma_{t+2}^2-\sigma_{t+1}^2|}\text{ a.s.}\\
    &\leq\frac{LM}{\sqrt{\epsilon}}\eta_{t+1}+L\left(\frac{3+d}{2}\right)^{\nicefrac{3}{2}}\eta_{t+1}\text{ a.s.}\\
    &=\eta_{t+1}\left(\frac{LM}{\sqrt{\epsilon}}+L\left(\frac{3+d}{2}\right)^{\nicefrac{3}{2}}\right).
\end{align}
By Lemma 21 of \cite{he2023convergence},
\begin{equation}
     \lim_{t\to\infty}\|\nabla f_{\sigma_{t+1}}(\bmx_t)\|^2=0\text{ a.s.}.
\end{equation}
By Lemma~\ref{lem:difference_between_smoothed_gradients}, since $\sigma_t\to 0$, we know that
\begin{equation}
     \lim_{t\to\infty}\|\nabla f(\bmx_t)\|^2=0\text{ a.s.}.
\end{equation}
Since $\|\nabla f(\bmx_t)\|\leq M$ a.s., by Lebesgue's Dominated Convergence Theorem, we have that
\begin{equation}
    \lim_{t\to\infty}E(\|\nabla f(\bmx_t)\|^2)
    =E\left(\lim_{t\to\infty}\|\nabla f(\bmx_t)\|^2\right)
    =0.
\end{equation}
\end{proof}

\section{Details Regarding Explicitly Smoothing Neural Networks}
\label{app:details_explicit_smoothing}

In order to calculate the explicit form that a smoothed neural network has, we need to be able to combine all of the model's components (including its constraint equations) into a single function that we can take the expectation of.
We begin our analysis by writing out the constrained and unconstrained optimization problems for both FFNNs and CNNs.

A FFNN satisfies the following constrained optimization problem:
\begin{align}
\begin{split}
\label{eqn:constrained_ffnn}
    \min_{\theta,b}\quad&\sum_{n=1}^{N}\|x_L^n-y^n\|^2\\
    \text{subject to}\quad&x_1^n=\theta_1x_0^n+b_1\\
    &x_l^n=\theta_lh(x_{l-1}^n)+b_l\text{ for }l=2,...,L
\end{split}
\end{align}
Converting this to an unconstrained optimization problem means that a FFNN satisfies the following:
\begin{align}
\begin{split}
\label{eqn:unconstrained_ffnn}
    \min_{\theta,b}\sum_{n=1}^{N}
        \|x_L^n-y^n\|^2
        +\lambda_1\|\theta_1x_0^n+b_1-x_1^n\|^2
        +\sum_{l=2}^{L}\lambda_l\|\theta_lh(x_{l-1}^n)+b_l-x_l^n\|^2.
\end{split}
\end{align}

A CNN satisfies the following constrained optimization problem:
\begin{align}
\begin{split}
\label{eqn:constrained_cnn}
    \min_{\theta,b}\quad&\sum_{n=1}^{N}\|x_{L+C}^n-y^n\|^2\\
    \text{subject to}\quad&x_1^n=x_0^n*\theta_1+b_1\\
    &x_l^n=h(x_{l-1}^{n})*\theta_l+b_l\text{ for }2\leq l\leq C\\
    &x^n_{l}=\theta_{l}h(x^n_{l-1})+b_{l}\text{ for }C+1\leq l\leq C+L
\end{split}
\end{align}
Again converting this into an unconstrained problem shows that a CNN satisfies:
\begin{multline}
\label{eqn:unconstrained_cnn}
    \min_{\theta,b}\sum_{n=1}^{N}
        \|x_{L+C}^n-y^n\|^2
        +\lambda_1\|x_0^n*\theta_1+b_1-x_1^n\|_F^2
        +\sum_{l=2}^{C}\lambda_l\|h(x_{l-1}^{n})*\theta_l+b_l-x_l^n\|_F^2\\
        +\sum_{l=C+1}^{C+L}\lambda_l\|\theta_{l}h(x^n_{l-1})+b_{l}l-x_l^n\|^2.
\end{multline}
To smooth these loss functions, we need to smooth each term of its sum.
The summary of the smoothed terms and the propositions where they are proven can be found in Table~\ref{tab:summary_of_nnet_smoothing}.

\begin{table}[t]
    \caption{Summary of layer-wise Gaussian smoothing (constants removed)}
    \label{tab:summary_of_nnet_smoothing}
    \centering
    $\begin{array}{c|c|c|c|c}
        \multicolumn{1}{c}{\multirow{2}{*}{\parbox{0.08\linewidth}{\vspace{5pt}\centering\textbf{Layer}\\\textbf{No. }($\bm{l}$)}}}
        &\multicolumn{1}{c}{\multirow{2}{*}{\parbox{0.15\linewidth}{\vspace{5pt}\centering\textbf{Original}\\\textbf{Constraint}}}}
        &\multicolumn{2}{c}{\textbf{Smoothed}}&\\\cmidrule(lr){3-4}
        \multicolumn{1}{c}{}&\multicolumn{1}{c}{}&\multicolumn{1}{c}{\textbf{Constraint}}&\multicolumn{1}{c}{\textbf{Regularizer}}&\textbf{Prop.}\\\specialrule{.1em}{.05em}{.05em}
        l=1&
            x_l=\theta_lh(x_{l-1})+b_l&
            x_l=\theta_lh_{\sigma}(x_{l-1})+b_l&
            &\text{\ref{prop:smoothing_dense_layer}}
            \\\hline
        \multirow{2}{*}{$l>1$}&
            \multirow{2}{*}{$x_l=\theta_lh(x_{l-1})+b_l$}&
            \multirow{2}{*}{$x_l=\theta_lh_{\sigma}(x_{l-1})+b_l$}&
            \|\theta_l\text{diag}(\sqrt{(h^2)_{\sigma}}(x_{l-1}))\|^2_F\qquad\qquad
            &\multirow{2}{*}{\ref{prop:smoothing_dense_layer}}\\
            &&&\qquad\qquad+\|\theta_l\text{diag}(h_{\sigma}(x_{l-1}))\|^2_F&
            \\\hline
        l>1&
            x_l=\theta_lx_{l-1}+b_l&
            x_l=\theta_l\sigma(x_{l-1})+b_l&
            \frac{\sigma^2}{2}\|\theta_l\|_F^2+\frac{\sigma^2d_l}{2}\|x_{l-1}\|^2
            &\text{\ref{prop:smoothing_dense_layer}}
            \\\hline
        l=1&
            x_l=x_{l-1}*\theta_l+b_l-x_l&
            x_l=x_{l-1}*\theta_l+b_l-x_l&
            &\text{\ref{prop:smoothing_convolutional_layer}}
            \\\hline
        l>1&
            x_l=x_{l-1}*\theta_l+b_l-x_l&
            x_l=x_{l-1}*\theta_l+b_l-x_l&
            \frac{\sigma^2w_l^2}{2}\|\theta_l\|^2_F+\frac{\sigma^2}{2}\|x_{l-1}\|_{C_l}^2
            &\text{\ref{prop:smoothing_convolutional_layer}}
            \\\hline
        l\geq 1&
            x_l=\text{Drop}(x_{l-1})&
            x_l=x_{l-1}&
            p\|x_l\|^2
            &\text{\ref{prop:smoothing_dropout_layer}}
            \\\hline
        l\geq 1&
            x_l=\text{AvgPool}(x_{l-1})&
            x_l=\text{AvgPool}(x_{l-1})&
            &\text{\ref{prop:smoothing_pooling_layer}}
    \end{array}$
\end{table}

In order to smooth the functions from the unconstrained problems, we need the following results primarily from \cite{mobahi2016training}.


\begin{prop}
\label{prop:results_from_mobahi}
The following are the results from Gaussian smoothing:
\begin{enumerate}[label={(\alph*)},ref={\thedefn~(\alph*)}]
    \item \cite{mobahi2016training} $\relu_{\sigma}(x)=\frac{x}{2}\left(1+\erf(\nicefrac{x}{\sigma})\right)+\frac{\sigma}{2\sqrt{\pi}}e^{-\nicefrac{x^2}{\sigma^2}}$
    \item $\relu^2_{\sigma}(x)=\frac{1}{4}(1+\erf(\nicefrac{x}{\sigma}))(\sigma^2+2x^2)+\frac{\sigma x}{2\sqrt{\pi}}e^{-\nicefrac{x^2}{\sigma^2}}$
    \item \cite{mobahi2016training} For a matrix $A$, vector $b$, and function $h$,
    \begin{equation}
        \Big(\|Ah(\bmy)+b\|^2\Big)_{\sigma}(\bmx)
        =\Big\|Ah_{\sigma}(\bmx)+b\Big\|^2+\Big\|A\textup{diag}\Big(\sqrt{(h^2)_{\sigma}}(\bmx)\Big)\Big\|^2_F-\Big\|A\textup{diag}(h_{\sigma}(\bmx))\Big\|^2_F
    \end{equation}
    \item $(\|\bmy\|^2)_{\sigma}(\bmx)=\|\bmx\|^2+\frac{\sigma^2d}{2}$
\end{enumerate}
\end{prop}

\begin{proof}[Proof of (b)]
We have the following:
\begin{align}
    (\relu^2)_{\sigma}(x)
    &=\frac{1}{\sqrt{\pi}}\int_{\mathbb{R}}\Big(\relu(x+\sigma u)\Big)^2e^{-u^2}\;du\\
    &=\frac{1}{\sqrt{\pi}}\int_{-\nicefrac{x}{\sigma}}^{\infty}(x+\sigma u)^2e^{-u^2}\;du\\
    &=\frac{1}{\sqrt{\pi}}\int_{-\nicefrac{x}{\sigma}}^{\infty}(x^2+2\sigma xu+\sigma^2 u^2)e^{-u^2}\;du\\
    &=\frac{x^2}{2}\Big(1+\erf(\nicefrac{x}{\sigma})\Big)-\frac{\sigma x}{\sqrt{\pi}}e^{-\nicefrac{x^2}{\sigma^2}}+\frac{\sigma x}{2\sqrt{\pi}}e^{-\nicefrac{x^2}{\sigma^2}}+\frac{\sigma^2}{4}\Big(1+\erf(\nicefrac{x}{\sigma})\Big)\\
    &=\frac{1}{4}\Big(1+\erf(\nicefrac{x}{\sigma})\Big)(\sigma^2+2x^2)+\frac{\sigma x}{2\sqrt{\pi}}e^{-\nicefrac{x^2}{\sigma^2}}
\end{align}
\end{proof}

\begin{proof}[Proof of (d)]
Observe:
\begin{align}
    \frac{1}{\sqrt{\pi}}\int_{\mathbb{R}^d}\|\bmx+\sigma \bmu\|^2e^{-\|\bmu\|^2}\;du
    &=\frac{1}{\sqrt{\pi}}\int_{\mathbb{R}^d}(\|\bmx\|^2+2\sigma \langle \bmx,\bmu\rangle + \sigma^2\|\bmu\|^2)e^{-\|\bmu\|^2}\;du\\
    &=\|\bmx\|^2+0+\frac{\sigma^2d}{2}.
\end{align}
\end{proof}

\subsection{Smoothing Terms in Neural Network Unconstrained Optimization}
\label{app:details_explicit_smoothing_unconstrained}

In this section, we smooth each of the terms from the FFNN and CNN unconstrained problems.
Let us represent the unconstrained problem as
\begin{equation}
    \min_{x,\theta,b}\sum_{n=1}^{N}f(x,\theta,b)
\end{equation}
where $f$ is a sum of certain norms.
Since smoothing is a linear operator, we can focus on a single data point and drop the sum.
Hence, we focus on smoothing $f(x,\theta,b)$.
Then
\begin{multline}
    (f*k_{\sigma})(x,\theta,b)\\
    =\underbrace{\int_{\mathbb{R}^{d_L}} \cdots \int_{\mathbb{R}^{d_1}}}_{\substack{\text{smoothing}\\\text{wrt }x}}
    \underbrace{\int_{\mathbb{R}^{d_{L}\times d_{L-1}}} \cdots \int_{\mathbb{R}^{d_1\times d_0}}}_{\substack{\text{smoothing}\\\text{wrt }\theta}}
    \underbrace{\int_{\mathbb{R}^{d_L}} \cdots \int_{\mathbb{R}^{d_1}}}_{\substack{\text{smoothing}\\\text{wrt }b}}
    f(x+\sigma u_x,\theta+\sigma U_{\theta},b+\sigma u_b)\\
    \cdot e^{-\|u_b\|^2}e^{-\|U_{\theta}\|_F^2}e^{-\|u_x\|^2}du_bdU_{\theta}du_x.
\end{multline}
Since $f\geq 0$, we are free to switch the order of integration.
As such, the process that we will take is to sequentially smooth with respect to the variables.

A summary of the results in this section can be found in Table~\ref{tab:summary_of_nnet_smoothing}.

Based on the unconstrained problems, we only ever need to focus on a layer and its input.
We use the notation $x$ for the input to a layer and $x_+$ for the output. As such, the weights and biases of a layer are denoted by $\theta_+$ and $b_+$, respectively.


\begin{prop}[Smoothing Dense Layer]
\label{prop:smoothing_dense_layer}
Let $x\in\mathbb{R}^d$, $x_+,b_+\in\mathbb{R}^{d_+}$, $\theta_+\in\mathbb{R}^{d\times d_+}$, and
\begin{equation}
    l(b_+,x_+,\theta_+,x )
    =\|\theta_+h(x )+b_+-x_+\|^2.
\end{equation}
If $x=x_0$ represents the input data (i.e. we cannot smooth with respect to $x$), then
\begin{equation}
    l_{\sigma}(b_1,x_1,\theta_1,x_0)
    =\|\theta_1h(x_{0})+b_1-x_1\|^2
    +\frac{\sigma^2d_1}{2}\|h(x_{0})\|^2
    +\sigma^2 d_1.
\end{equation}
Otherwise
\begin{multline}
    l_{\sigma}(b_+,x_+,\theta_+,x )
    =\|\theta_+h_{\sigma}(x )+b_+-x_+\|^2
    +\|\theta_+\textup{diag}(\sqrt{(h^2)_{\sigma}}(x ))\|^2_F\\
    -\|\theta\textup{diag}(h_{\sigma}(x ))\|^2_F
    +\sigma^2 d_+\left(1+\frac{1}{2}\|\sqrt{(h^2)_{\sigma}}(x )\|^2\right).
\end{multline}
Furthermore, if $h(x)=x$, then
\begin{equation}
    l_{\sigma}(b_+,x_+,\theta_+,x )
    =\|\theta_+x+b_+-x_+\|^2
    +\frac{\sigma^2}{2}\|\theta_+\|_F^2
    +\frac{\sigma^2d_+}{2}\|x\|^2
    +\frac{\sigma^4dd_+}{4}+\sigma^2d_+.
\end{equation}
\end{prop}

\begin{proof}
First, we smooth with respect to $b_+$:
\begin{align}
    &\frac{1}{\pi^{\nicefrac{d_+}{2}}}\int_{\mathbb{R}^{d_+}}\|\theta_+h(x )+(b_++\sigma u)-x_+\|^2e^{-\|u\|^2}\;du\\
    &=\frac{1}{\pi^{\nicefrac{d_+}{2}}}\int_{\mathbb{R}^{d_+}}\left(
        \|\theta_+h(x )+b_+-x_+\|^2
        +2\sigma\Big\langle \theta_+h(x )+b_+-x_+,u\Big\rangle
        +\sigma^2\|u\|^2
    \right)e^{-\|u\|^2}\;du\\
    &=\|\theta_+h(x )+b_+-x_+\|^2+\frac{\sigma^2 d_+}{2}
\end{align}

Second, we smooth with respect to $x_+$. Since $x_+$ plays the same role as $b_+$, smoothing with respect to $x_+$ is analogous to smoothing with respect to $b_+$, so we end up with 
\begin{align}
    \|\theta_+h(x )+b_+-x_+\|^2
    +\sigma^2 d_+.
\end{align}

Third, we smooth with respect to $\theta_+$. Focusing on the first term in the previous equation, let $d'=\text{dim}(\theta_+)=d_+ \times d$, then
\begin{align}
    &\frac{1}{\pi^{\nicefrac{d'}{2}}}\int_{\mathbb{R}^{d'}}\|(\theta_++\sigma U)h(x )+b_+-x_+\|^2e^{-\|U\|_F^2}\;dU\\
    &=\frac{1}{\pi^{d'}}\int_{\mathbb{R}^{d'}}\Big(
        \|\theta_+h(x )+b_+-x_+\|^2
        +\sigma^2\|Uh(x )\|^2
    \Big)e^{-\|U\|_F^2}\;dU\\
    &=\|\theta_+h(x )+b_+-x_+\|^2
    +\frac{\sigma^2}{\pi^{\nicefrac{d'}{2}}}\sum_{i=1}^{d_+}\sum_{j=1}^{d }\int_{\mathbb{R}^{d'}}U^{ij}U^{ik}\big(h(x )\big)_j\big(h(x )\big)_ke^{-\|U\|_F^2}\;dU\\
    &=\|\theta_+h(x )+b_+-x_+\|^2
    +\frac{\sigma^2}{\pi^{\nicefrac{d'}{2}}}\sum_{j=1}^{d }\left[
        \big(h(x )\big)_j^2
        \sum_{i=1}^{d_+}\frac{\pi^{\nicefrac{d'}{2}}}{2}
    \right]\\
    &=\|\theta_+h(x )+b_+-x_+\|^2
    +\frac{\sigma^2d_+}{2}\|h(x )\|^2.
\end{align}
So, after smoothing with respect to $b$, $\theta$, and $x_+$, we have
\begin{align}
\label{eqn:ffnn_second_term_smoothing_after_b_theta_xl}
    \|\theta_+h(x )+b_+-x_+\|^2
    +\frac{\sigma^2d_+}{2}\|h(x )\|^2
    +\sigma^2 d_+.
\end{align}

Finally, we smooth with respect to $x $.
By Mobahi's proposition, the first term of \eqref{eqn:ffnn_second_term_smoothing_after_b_theta_xl} smoothes to
\begin{align}
    \|\theta_+h_{\sigma}(x )+b_+-x_+\|^2
    +\|\theta_+\text{diag}(\sqrt{(h^2)_{\sigma}}(x ))\|^2_F 
    -\|\theta_+\text{diag}(h_{\sigma}(x ))\|^2_F.             
\end{align}
Again, by Mobahi's proposition, the second term of \eqref{eqn:ffnn_second_term_smoothing_after_b_theta_xl} smoothes to
\begin{multline}
    \Big\|h_{\sigma}(x )\Big\|^2+\Big\|\textup{diag}\Big(\sqrt{(h^2)_{\sigma}}(x )\Big)\Big\|^2_F-\Big\|\textup{diag}(h_{\sigma}(x ))\Big\|^2_F\\
    =\|h_{\sigma}(x )\|^2
    +\|\sqrt{(h^2)_{\sigma}}(x )\|^2
    -\|h_{\sigma}(x )\|^2
    =\|\sqrt{(h^2)_{\sigma}}(x )\|^2.
\end{multline}
Combining these two, we have that after smoothing with respect to everything,
\begin{multline}
    \|\theta_+h_{\sigma}(x )+b_+-x_+\|^2
    +\|\theta_+\text{diag}(\sqrt{(h^2)_{\sigma}}(x ))\|^2_F
    -\|\theta_+\text{diag}(h_{\sigma}(x ))\|^2_F\\
    +\sigma^2 d_+\left(1+\frac{1}{2}\|\sqrt{(h^2)_{\sigma}}(x )\|^2\right).
\end{multline}

To prove the furthermore statement, we jump back to \eqref{eqn:ffnn_second_term_smoothing_after_b_theta_xl}, which means that if $h$ is the identity function, then after smoothing with respect to $b$, $\theta$, and $x_+$, we have
\begin{align}
\label{eqn:ffnn_furthermore_starting_point}
    \|\theta_+x+b_+-x_+\|^2
    +\frac{\sigma^2d_+}{2}\|x\|^2
    +\sigma^2 d_+.
\end{align}
Once again, we smooth with respect to $x$.
The first term smoothes as follows:
\begin{align}
    &\frac{1}{\pi^{\nicefrac{d}{2}}}\int_{\mathbb{R}^d}\|\theta_+(x+\sigma u)+b_+-x_+\|^2e^{-\|u\|^2}\;du\\
    &=\|\theta_+x+b_+-x_+\|^2
    +\frac{\sigma^2}{\pi^{\nicefrac{d}{2}}}\int_{\mathbb{R}^d}\|\theta_+u\|^2e^{-\|u\|^2}\;du\\
    &=\|\theta_+x+b_+-x_+\|^2
    +\frac{\sigma^2}{\pi^{\nicefrac{d}{2}}}\sum_{i=1}^{d_+}\sum_{j,k=1}^{d}(\theta_+)_{ij}(\theta_+)_{ik}\int_{\mathbb{R}^d}u_ju_ke^{-\|u\|^2}\;du\\
    &=\|\theta_+x+b_+-x_+\|^2
    +\frac{\sigma^2}{\pi^{\nicefrac{d}{2}}}\sum_{i=1}^{d_+}\sum_{j=1}^{d}(\theta_+)_{ij}^2\int_{\mathbb{R}^d}u_j^2e^{-\|u\|^2}\;du\\
    &=\|\theta_+x+b_+-x_+\|^2
    +\frac{\sigma^2}{2}\sum_{i=1}^{d_+}\sum_{j=1}^{d}(\theta_+)_{ij}^2\\
    &=\|\theta_+x+b_+-x_+\|^2
    +\frac{\sigma^2}{2}\|\theta_+\|_F^2.
\end{align}
The second term of \eqref{eqn:ffnn_furthermore_starting_point}, smoothes to
\begin{equation}
    \frac{\sigma^2d_+}{2}\left(\|x\|^2+\frac{\sigma^2d}{2}\right)
    =\frac{\sigma^2d_+}{2}\|x\|^2+\frac{\sigma^4dd_+}{4},
\end{equation}
which has been done several times already in this proof. Therefore, we end up with
\begin{equation}
    l_{\sigma}(b_+,x_+,\theta_+,x )
    =\|\theta_+x+b_+-x_+\|^2
    +\frac{\sigma^2}{2}\|\theta_+\|_F^2
    +\frac{\sigma^2d_+}{2}\|x\|^2
    +\frac{\sigma^4dd_+}{4}+\sigma^2d_+.
\end{equation}
\end{proof}

Before we smooth the convolutional layer, we want to write out the ``convolutional norm'':
\begin{equation}
    \|x\|_{C_+}^2
    =\sum_{a,b=0}^{w_+-1}\sum_{i,j=1}^{k_+}\Big(x^{(as_++i)(bs_++j)}\Big)^2
    =\|x*J_{k_+}\|^2.
\end{equation}

\begin{prop}[Smoothing Convolutional Layer]
\label{prop:smoothing_convolutional_layer}
Let $x\in\mathbb{R}^d$, $x_+,b_+\in\mathbb{R}^{d_+}$, $\theta_+\in\mathbb{R}^{c_+\times c_+}$, and
\begin{equation}
    l(\theta_+,x_+,x)
    =\|x*_+\theta_++b_+-x_+\|_F^2
\end{equation}
were $*_+$ represents the convolution with stride $s_+$.
Let $C_+$ represent the convolutional layer with these parameters.
If $x=x_0$ represents the input data (i.e. we cannot smooth with respect to $x$), then
\begin{align}
    l_{\sigma}(\theta_+,x_+,x)
    =\|x*_+\theta_++b_+-x_+\|_F^2+\frac{\sigma^2}{2}\|x\|^2_{C_+}+\sigma^2d_+.
\end{align}
Otherwise,
\begin{align}
    l_{\sigma}(\theta_+,x_+,x)
    =\|x *_+\theta_++b_+-x_+\|_F^2+\frac{\sigma^2}{2}w_+^2\|\theta_+\|_F^2
    +\frac{\sigma^2}{2}\|x \|^2_{C_+}+\frac{\sigma^4}{4}\|1\|^2_{C_+}
    +\sigma^2d_+.
\end{align}
\end{prop}

\begin{proof}
For the proof, we will simplify notation and write $*$ instead of $*_+$ since all of the convolutions will use the same stride.
Recall that $w_+=\left(\frac{\text{width}(x)+c_+}{s_+}+1\right)$.
First, we smooth with respect to $b_+$:
\begin{align}
    &\frac{1}{\pi^{\nicefrac{d_+}{2}}}\int_{\mathbb{R}^{{d_+}}}\|x*\theta_++(b_++\sigma U)-x_+\|_F^2e^{-\|U\|_F^2}\;dU\\
    &=\frac{1}{\pi^{\nicefrac{d_+}{2}}}\int_{\mathbb{R}^{{d_+}}}\Big(
        \|x*\theta_++b_+-x_+\|_F^2
        +2\sigma\langle x*\theta_++b_+-x_+,U\rangle_F
        +\sigma^2\|U\|_F^2
    \Big)e^{-\|U\|_F^2}\;dU\\
    &=\|x*\theta_++b_+-x_+\|_F^2+\frac{\sigma^2d_+}{2}.
\end{align}

\noindent Second, we smooth with respect to $\theta_+$. For first term from the previous equation, we have
\begin{multline}
    \frac{1}{\pi^{\nicefrac{c_+^2}{2}}}\int_{\mathbb{R}^{c_+^2}}\|x *(\theta_++\sigma U)+b_+-x_+\|_F^2e^{-\|U\|_F^2}\;dU\\
    =\|x *\theta_++b_+-x_+\|_F^2
        +\frac{\sigma^2}{\pi^{\nicefrac{c_+^2}{2}}}\int_{\mathbb{R}^{c_+^2}}\|x *U\|_F^2e^{-\|U\|_F^2}\;dU
\end{multline}
Focusing on the convolution term, (here $a',b'=0,s_+,2s_+,...,w_+s_+$)
\begin{align}
\label{eqn:convolution_inner_product_0}
    &\frac{\sigma^2}{\pi^{\nicefrac{c_+^2}{2}}}\int_{\mathbb{R}^{c_+^2}}\|x *U\|_F^2e^{-\|U\|_F^2}\;dU\\
    &=\frac{\sigma^2}{\pi^{\nicefrac{c_+^2}{2}}}\int_{\mathbb{R}^{c_+^2}}
        \sum_{a',b'}\left(\sum_{i,j=1}^{c_+}U^{ij}x^{(a'+i)(b'+j)}\right)^2
    e^{-\|U\|_F^2}\;dU\\
    &=\frac{\sigma^2}{\pi^{\nicefrac{c_+^2}{2}}}\int_{\mathbb{R}^{c_+^2}}
        \sum_{a',b'}\left(\sum_{i,j=1}^{c_+}\sum_{k,l=1}^{c_+}U^{ij}U^{kl}x ^{(a'+i)(b'+j)}x ^{(a'+k)(b'+l)}
    \right)e^{-\|U\|_F^2}\;dU\\
    &=\frac{\sigma^2}{\pi^{\nicefrac{c_+^2}{2}}}\sum_{a',b'}\sum_{i,j=1}^{c_+}\big(x ^{(a'+i)(b'+j)}\big)^2\int_{\mathbb{R}^{c_+^2}}(U^{ij})^2e^{-\|U\|_F^2}\;dU\\
    &=\frac{\sigma^2}{2}\|x \|_{C_+}^2.
\end{align}
So after smoothing with respect to $b_+$ and $\theta_+$, we have
\begin{align}
    \|x *\theta_++b_+-x_+\|_F^2+\frac{\sigma^2}{2}\|x \|^2_{C_+}+\frac{\sigma^2d_+}{2}.
\end{align}

Third, we smooth with respect to $x_+$, but as with $b_+$ we end up with
\begin{align}
\label{eqn:smoothing_convolutional_layer_1}
    \|x *\theta_++b_+-x_+\|_F^2+\frac{\sigma^2}{2}\|x\|^2_{C_+}+\sigma^2d_+.
\end{align}

Finally, we smooth with respect to $x$. For the first term of \eqref{eqn:smoothing_convolutional_layer_1}, we have
\begin{multline}
    \frac{1}{\pi^{\frac{d }{2}}}\int_{\mathbb{R}^{d }}\|(x +\sigma U)*\theta_++b_+-x_+\|_F^2e^{-\|U\|_F^2}\;dU\\
    =\|x *\theta_++b_+-x_+\|_F^2
    +\frac{\sigma^2}{\pi^{\frac{d}{2}}}\int_{\mathbb{R}^{d}}\|U*\theta_+\|_F^2e^{-\|U\|_F^2}\;dU
\end{multline}
Now, for the second term, we repeat the process in \eqref{eqn:convolution_inner_product_0} and arrive at
\begin{align}
    &\frac{1}{\pi^{\frac{d }{2}}}\int_{\mathbb{R}^{d }}\|(x +\sigma U)*\theta_++b_+-x_+\|_F^2e^{-\|U\|_F^2}\;dU\\
    &=\|x *\theta_++b_+-x_+\|_F^2+\frac{\sigma^2}{2}\sum_{a,b=1}^{w_+}\sum_{i,j=1}^{c_+}(\theta_+^{ij})^2\\
    &=\|x *\theta_++b_+-x_+\|_F^2+\frac{\sigma^2}{2}w_+^2\sum_{i,j=1}^{c_+}(\theta_+^{ij})^2\\
    &=\|x *\theta_++b_+-x_+\|_F^2+\frac{\sigma^2}{2}w_+^2\|\theta_+\|_F^2.
\end{align}
For the second term of \eqref{eqn:smoothing_convolutional_layer_1},
\begin{align}
    &\frac{1}{\pi^{\frac{d }{2}}}\int_{\mathbb{R}^d}\|x +\sigma U\|^2_{C_+}e^{-\|U\|_F^2}\;dU\\
    &=\frac{1}{\pi^{\frac{d }{2}}}\int_{\mathbb{R}^d}\sum_{a,b=1}^{w_+}\sum_{i,j=1}^{c_+}
        \Big(x ^{(a+i)(b+j)}+\sigma U^{(a+i)(b+j)}\Big)^2
    e^{-\|U\|_F^2}\;dU\\
    &=\sum_{a,b=1}^{w_+}\sum_{i,j}\Big(x ^{(a+i)(b+j)}\Big)^2+\frac{\sigma^2}{\pi^{\frac{d }{2}}}\sum_{a,b=1}^{w_+}\sum_{i,j}\frac{\pi^{\frac{d }{2}}}{2}\\
    &=\|x \|^2_{C_+}+\frac{\sigma^2}{2}\|1\|^2_{C_+}.
\end{align}
So, after smoothing with respect to all variables, we end up with
\begin{align}
    \|x *\theta_++b_+-x_+\|_F^2+\frac{\sigma^2}{2}w_+^2\|\theta_+\|_F^2
    +\frac{\sigma^2}{2}\|x \|^2_{C_+}+\frac{\sigma^4}{4}\|1\|^2_{C_+}
    +\sigma^2d_+.
\end{align}
\end{proof}

\begin{prop}[Smoothing Dropout Layer]
\label{prop:smoothing_dropout_layer}
Let $p\in [0,1]$ and $\textup{drop}$ be given by $\mathcal{P}(\textup{drop}(x) = 0) = p$ and $\mathcal{P}(\textup{drop}(x) = x) = 1-p$.
For $x,x_+\in\mathbb{R}^d$, define the unconstrained loss associated with the dropout layer as
\begin{equation}
    l(x,x_+)=\|\textup{drop}(x)-x_+\|^2.
\end{equation}
Then
\begin{equation}
    l_{\sigma}(x,x_+)=p\|x_+\|^2+(1-p)\|x-x_+\|^2+\left(1-\tfrac{p}{2}\right)\sigma^2d.
\end{equation}
\end{prop}

\begin{proof}
Let $x,x_+\in\mathbb{R}^d$.
Note that if $x$ and $x_+$ are matrices, the norms and inner products in this proof are the Frobenius-type.
First, we will smooth with respect to $x_+$. So,
\begin{align}
    E_u[l(x,x_++\sigma u)]
    &=\frac{p}{\pi^{\nicefrac{d}{2}}}\int_{\mathbb{R}^d}\|x_++\sigma u\|^2e^{-\|u\|^2}\;du
    +\frac{1-p}{\pi^{\nicefrac{d}{2}}}\int_{\mathbb{R}^d}\|x - (x_++\sigma u)\|^2e^{-\|u\|^2}\;du.
\end{align}
The first integral can be computed as
\begin{align}
    \frac{p}{\pi^{\nicefrac{d}{2}}}\int_{\mathbb{R}^d}\|x_++\sigma u\|^2e^{-\|u\|^2}\;du
    &=p\|x_+\|^2+\frac{p\sigma^2d}{2}.
\end{align}
The second integral is
\begin{align}
    \frac{1-p}{\pi^{\nicefrac{d}{2}}}\int_{\mathbb{R}^d}\|x - (x_++\sigma u)\|^2e^{-\|u\|^2}\;du
    &=(1-p)\|x-x_+\|^2+\frac{(1-p)\sigma^2d}{2}.
\end{align}
So,
\begin{align}
    E_u[l(x,x_++\sigma u)]
    &=p\|x_+\|^2+\frac{p\sigma^2d}{2}
    +(1-p)\|x-x_+\|^2+\frac{(1-p)\sigma^2d}{2}\\
    &=p\|x_+\|^2+(1-p)\|x-x_+\|^2+\frac{\sigma^2d}{2}.
\end{align}

Second, we will smooth with respect to $x$. We have
\begin{align}
    &E_u(l(x+\sigma u,x_+))\\
    &=\frac{1}{\pi^{\nicefrac{d}{2}}}\int_{\mathbb{R}^d} \left(
        p\|x_+\|^2+(1-p)\|x+\sigma u-x_+\|^2+\frac{\sigma^2d}{2}
    \right)e^{-\|u\|^2}\;du\\
    &=p\|x_+\|^2
    +\frac{\sigma^2d}{2}
    +\frac{1-p}{\pi^{\nicefrac{d}{2}}}\int_{\mathbb{R}^d}\|x+\sigma u-x_+\|^2e^{-\|u\|^2}\;du\\
    &=p\|x_+\|^2
    +\frac{\sigma^2d}{2}
    +(1-p)\left(
        \|x-x_+\|^2
        +\frac{\sigma^2d}{2}
    \right)\\
    &=(1-p)\|x-x_+\|^2
    +p\|x_+\|^2
    +(2-p)\frac{\sigma^2d}{2}.
\end{align}
This shows the claimed result.
\end{proof}

\begin{prop}[Smoothing Average Pooling]
\label{prop:smoothing_pooling_layer}
Define
\begin{equation}
    \Big(\textup{pool}(x)\Big)_i = \frac{1}{k}\sum_{j=1}^{k}x^{i_j},
\end{equation}
that is the pooling layer averages over $x^{i_j}$ for $j=1,...,k$.
Let
\begin{equation}
    l(x,x_+)=\|\textup{pool}(x)-x_+\|^2.
\end{equation}
Then
\begin{equation}
    l_{\sigma}(x,x_+)=\|\textup{pool}(x)-x_+\|^2
    + \frac{\sigma^2 d_+}{2k} + \frac{\sigma^2 d_+}{2}.
\end{equation}
\end{prop}

\begin{proof}
First, we smooth with respect to $x_+$, which has been done numerous times throughout this note and we end up with
\begin{equation}
    \|\textup{pool}(x)-x_+\|^2 + \frac{\sigma^2 d_+}{2}.
\end{equation}
Second, we smooth with respect to $x$:
\begin{multline}
    \frac{1}{\pi^{\nicefrac{d}{2}}}\int_{\mathbb{R}^d}\left\|
        \textup{pool}(x+\sigma u)-x_{+}
    \right\|^2e^{-\|u\|^2}\;du\\
    =\|\textup{pool}(x+\sigma u)-x_{+}\|^2
    +\frac{\sigma^2}{\pi^{\nicefrac{d}{2}}}\int_{\mathbb{R}^d}\left\|
        \textup{pool}(u)
    \right\|^2e^{-\|u\|^2}\;du.
\end{multline}
Focusing on the last term,
\begin{align}
    \int_{\mathbb{R}^d}\left\|
        \textup{pool}(u)
    \right\|^2e^{-\|u\|^2}\;du
    &=\sum_{i=1}^{d_+}\int_{\mathbb{R}^d}\left(
        \frac{1}{k}\sum_{j=1}^{k}u^{i_j}
    \right)^2e^{-\|u\|^2}\;du\\
    &=\frac{1}{k^2}\sum_{i=1}^{d_+}\sum_{j,m=1}^{k}\int_{\mathbb{R}^d}
        u^{i_j}u^{i_m}e^{-\|u\|^2}\;du\\
    &=\frac{1}{k^2}\sum_{i=1}^{d_+}\sum_{j=1}^{k}\int_{\mathbb{R}^d}
        (u^{i_j})^2e^{-\|u\|^2}\;du\\
    &=\frac{d_+\pi^{\nicefrac{d}{2}}}{2k}.
\end{align}
So,
\begin{equation}
    \frac{1}{\pi^{\nicefrac{d}{2}}}\int_{\mathbb{R}^d}\left\|
        \textup{pool}(x+\sigma u)-x_{+}
    \right\|^2e^{-\|u\|^2}\;du
    =\|\textup{pool}(x+\sigma u)-x_{+}\|^2+\frac{d_+\sigma^2}{2k}.
\end{equation}
Combining this with the constant term from the previous equation gives the result.
\end{proof}

\begin{note}[Comment about pooling layers]
If a max pooling layer is used, then the smoothing function is difficult (if not impossible to compute). In particular, let $\mathcal{I}$ is the index that a single part of the pooling layer takes the maximum over and
\begin{equation}
    h(x)=\max_{i,j\in\mathcal{I}}x^{ij}.
\end{equation}
Since $h$ is nonlinear, we need to use Proposition~\ref{prop:results_from_mobahi} (e) which requires computing $h_{\sigma}$.
Then
\begin{align}
    h_{\sigma}(x)
    &=\frac{1}{\pi^{\nicefrac{|\mathcal{I}|}{2}}}\int_{\mathbb{R}^{|\mathcal{I}|}}\max_{i,j\in\mathcal{I}}\Big(x^{ij}+\sigma U_{ij}\Big)e^{-\|U\|_F^2}\;dU\\
    &=E\left(\max_{i,j\in\mathcal{I}}X_{ij}\right)
\end{align}
where $X_{ij}\sim\mathcal{N}(x_l^{ij},\sigma^2)$.
Even in the case where all of the $x_l^{ij}$ are the same, this does not have a closed form.
Hence, we cannot compute $h_{\sigma}$.
\end{note}

\subsection{Proofs of Mathematical Formulation of Smoothed Neural Networks}
\label{app:proofs_of_smooth_nnets}

Now that we have smoothed all of the components of both of the unconstrained problems \eqref{eqn:unconstrained_ffnn} and \eqref{eqn:unconstrained_cnn}, we are ready to explicitly write out their smoothed mathematical formulations.

\begin{proof}[Proof of Theorem~\ref{thm:smooth_ffnn}]
The unconstrained FFNN is given in \eqref{eqn:unconstrained_ffnn}, which finds the minimum of the sum (over $n$) of
\begin{align}
\begin{split}
        \|x_L^n-y^n\|^2
        +\lambda_1\|\theta_1x_0^n+b_1-x_1^n\|^2
        +\sum_{l=2}^{L}\lambda_l\|\theta_lh(x_{l-1}^n)+b_l-x_l^n\|^2.
\end{split}
\end{align}
Based on the previous section, this smoothes to
\begin{align}
\begin{split}
    \|x_L^n-y^n\|^2&+\frac{\sigma^2d_L}{2}
    +\lambda_1\left(\|\theta_1x_0^n+b_1-x_1^n\|^2+\frac{\sigma^2d_1}{2}\|x_0\|^2+\sigma^2d_1\right)\\
    &+\sum_{l=2}^{L}\lambda_l\Bigg(\|\theta_lh(x_{l-1}^n)+b_l-x_l^n\|^2
    +\|\theta_l\textup{diag}(\sqrt{(h^2)_{\sigma}}(x_{l-1}))\|^2_F\\
    &\phantom{+\sum_{l=2}^{L}\lambda_l\Bigg(}-\|\theta_l\textup{diag}(h_{\sigma}(x_{l-1}))\|^2_F
    +\sigma^2 d_l\left(1+\frac{1}{2}\|\sqrt{(h^2)_{\sigma}}(x_{l-1})\|^2\right)\Bigg).
\end{split}
\end{align}
Now, we can reconstrain to get the result. The proof for the CNN case is exactly the same.
\end{proof}

\section{Additional Details of Numerical Experiments and Practical Guide to Smooth Neural Network Implementation}
\label{app:numerics}

We begin with a description of how to implement the explicitly smooth neural networks.
In order to use TensorFlow's built-in layer regularizers, we avoid regularization terms that mix weights and layer inputs (e.g., $\|\theta\textup{diag}(\sqrt{(h^2)_{\sigma}}(x))\|^2_F$).
This is why took a non-standard approach and we separated the activation function from the layer.
Hence, we split up terms like $\|\theta h(x)+b\|$ into two terms like $\|h(x)-y\|$ and $\|\theta y+b\|$.
Then we use the results from Section~\ref{app:details_explicit_smoothing_unconstrained} to smooth.
The original unconstrained terms and their smoothed constraints and regularization terms can be found in Table~\ref{tab:example_of_explicit_smoothing}.

\begin{table}[t]
    \caption{Example converting CNN layers to smoothed counterparts (additive constants omitted).}
    \label{tab:example_of_explicit_smoothing}
    \centering
    $\begin{array}{c|cc}
        \multicolumn{1}{c}{\multirow{2}{*}{\parbox{0.15\linewidth}{\vspace{5pt}\centering\textbf{Original}\\\textbf{Constraint}}}}
        &\multicolumn{2}{c}{\textbf{Smoothed}}\\\cmidrule(lr){2-3}
        \multicolumn{1}{c}{}&\multicolumn{1}{c}{\textbf{Constraint}}&\multicolumn{1}{c}{\textbf{Regularizer}}\\\specialrule{.1em}{.05em}{.05em}
        \|x_0*\theta_1+b_1-x_1\|^2_F       &\|x_0*\theta_1+b_1 - x_1\|^2_F&\\
        \|h(x_1)-x_1^h\|^2_F               &\|h_{\sigma}(x_1)-x_1^h\|_F^2&+\|\sqrt{(h^2)_{\sigma}}(x_1)\|_F^2-\|h_{\sigma}(x_1)\|_F^2\\
        \|\textup{Pool}(x_1^h)-x_2\|^2_F   &\|\textup{Pool}(x_1^h)-x_2\|_F^2&\\
        \|\textup{Flatten}(x_2)-x_3\|^2    &\|\textup{Flatten}(x_2)-x_3\|^2&\\
        \|\theta_4x_3+b_4-x_4\|^2          &\|\theta_4x_3+b_4-x_4\|^2&+\frac{\sigma^2}{2}\|\theta_4\|_F^2+\frac{\sigma^2d_4}{2}\|x_3\|^2\\
        \|h(x_4)-x_4^h\|^2                 &\|h_{\sigma}(x_4)-x_4^h\|^2&+\|\sqrt{(h^2)_{\sigma}}(x_4)\|^2-\|h_{\sigma}(x_4)\|^2\\
        \|\theta_5x_4^h+b_5-x_5\|^2        &\|\theta_5x_4^h+b_5-x_5\|^2&+\frac{\sigma^2}{2}\|\theta_5\|_F^2+\frac{\sigma^2d_5}{2}\|x_4\|^2
    \end{array}$
\end{table}

Practically, here are the steps that we follow to get the explicitly smooth network:
\begin{enumerate}
    \item Decide on the network structure making sure to split activation functions into their own terms (as mentioned before)
    \item Use Table~\ref{tab:summary_of_nnet_smoothing_names} to convert to the smoothed network
    \item Create code for smooth network using appropriate regularization terms
\end{enumerate}

Moving onto the details of our experiments, for all of our experiments we use the CNN architecture in Table~\ref{tab:cnn_architecture}.
The resulting smooth CNN architecture can be found in Table~\ref{tab:smooth_cnn_architecture}.
All training used a batch size of 1 and early stopping with a patience of 2 epochs based on validation loss (up to 25 epochs).
For the smooth architecture, four regularization weights need to be chosen, one for each of the following: first ReLU layer ($\lambda_2$), first dense layer ($\lambda_5$), second ReLU layer ($\lambda_6$), and output layer ($\lambda_7$).
Based on the construction of our particular CNN, no regularization term was needed for the convolutional layer.
A coarse hyperparameter search was done for the learning rate (options: $10^{-3}$, $10^{-4}$, and $10^{-5}$) and regularization coefficients (options: $10^{-3}$, $10^{-5}$, $10^{-7}$, $10^{-9}$) using 5-fold cross validation on the training set.
The optimal hyperparameters that were used for the experiments are shown in Table~\ref{tab:smooth_cnn_hyperparameters}.
The default values for any additional hyperparameters available in Tensorflow, but not mentioned here were used for these methods.
Additionally, the hyperparameters for when $\sigma=0.01$ are those found for $\sigma=0.1$, no additional search was performed for this $\sigma$-value.
Finally, the learning rate for the CIFAR-10 experiments using Adam were reduced by a factor of $10^{-1}$ based on additional 5-fold cross validation on the training set using only the unsmoothed CNN.
Code for the experiments is available at \url{https://github.com/acstarnes/GSmoothSGD}.

We ran the MNIST and CIFAR-10 experiments using $\sigma=0,0.01,0.1,0.5,1$.
Figure~\ref{fig:noisy_mnist_images} contains example of the noisy MNIST images and Figure~\ref{fig:noisy_cifar10_images} shows examples for CIFAR-10.
The full heatmaps of the experiments are shown in Figures~\ref{fig:full_noise_mnist} and~\ref{fig:full_noise_cifar10}.

\begin{table}[t]
    \caption{CNN architecture for experiments}
    \label{tab:cnn_architecture}
    \centering
    \begin{tabular}{cc|l}
         \textbf{Layer}&\textbf{Layer}&\textbf{Layer}\\
         \textbf{Number}&\textbf{Name}&\textbf{Properties}\\\specialrule{.1em}{.05em}{.05em}
         0&Input&MNIST or CIFAR10 images\\\specialrule{.1em}{.05em}{.05em}
         1&Convolution&32 $4\times 4$ convolutions with stride of 1 and `valid' padding\\\specialrule{.1em}{.05em}{.05em}
         2&Activation&ReLU\\\specialrule{.1em}{.05em}{.05em}
         3&Pooling&Average pooling with $2\times 2$ pool size and stride of 2\\\specialrule{.1em}{.05em}{.05em}
         4&Flatten&\\\specialrule{.1em}{.05em}{.05em}
         5&Dense&128 neurons\\\specialrule{.1em}{.05em}{.05em}
         6&Activation&ReLU\\\specialrule{.1em}{.05em}{.05em}
         7&Dense&10 neurons\\
    \end{tabular}
\end{table}

\begin{table}[t]
    \caption{Smooth CNN architecture for experiments. ``Activity'' and ``Kernel'' indicate regularization terms applied to the layer's output and weights, respectively. The regularization coefficients are $\lambda_l$ where $l$ indicates the layer where the regularization term comes from.}
    \label{tab:smooth_cnn_architecture}
    \centering
    \begin{tabular}{cc|l}
         \textbf{Layer}&\textbf{Layer}&\textbf{Layer}\\
         \textbf{Number}&\textbf{Name}&\textbf{Properties}\\\specialrule{.1em}{.05em}{.05em}
         0&Input&MNIST or CIFAR10 images\\\specialrule{.1em}{.05em}{.05em}
         \multirow{2}{*}{1}&\multirow{2}{*}{Convolution}&32 $4\times 4$ convolutions with stride of 1 and `valid' padding\\\cline{3-3}
         &&Activity ($\lambda_2$): $\|\sqrt{(\relu^2)_{\sigma}}(x)\|_F^2+\|\relu_{\sigma}(x)\|_F^2$\\\specialrule{.1em}{.05em}{.05em}
         2&Activation&ReLU\\\specialrule{.1em}{.05em}{.05em}
         3&Pooling&Average pooling with $2\times 2$ pool size and stride of 2\\\specialrule{.1em}{.05em}{.05em}
         4&Flatten&Activity ($\lambda_5$): $\frac{\sigma^2\cdot 128}{2}\|x\|$\\\specialrule{.1em}{.05em}{.05em}
         \multirow{3}{*}{5}&\multirow{3}{*}{Dense}&128 neurons\\\cline{3-3}
         &&Kernel ($\lambda_5$): $\frac{\sigma^2}{2}\|\theta\|_F^2$\\\cline{3-3}
         &&Activity ($\lambda_6$): $\|\sqrt{(\relu^2)_{\sigma}}(x)\|_F^2+\|\relu_{\sigma}(x)\|_F^2$\\\specialrule{.1em}{.05em}{.05em}
         \multirow{2}{*}{6}&\multirow{2}{*}{Activation}&ReLU\\\cline{3-3}
         &&Activity($\lambda_7$): $\frac{\sigma^2\cdot 5}{2}\|x\|$\\\specialrule{.1em}{.05em}{.05em}
         \multirow{2}{*}{7}&\multirow{2}{*}{Dense}&10 neurons\\\cline{3-3}
         &&Kernel($\lambda_7$): $\frac{\sigma^2}{2}\|\theta\|_F^2$\\
    \end{tabular}
\end{table}

\begin{table}[t]
    \caption{Smooth CNN hyperparameters tuned from 5-fold cross validation on MNIST experiments. The regularization coefficients are $\lambda_l$ where $l$ indicates the layer where the regularization term comes from with Table~\ref{tab:smooth_cnn_architecture} showing the layer where they are used.}
    \label{tab:smooth_cnn_hyperparameters}
    \centering
    \begin{tabular}{cccccccc}
        &&\textbf{MNIST}&\textbf{CIFAR-10}&\multicolumn{4}{c}{\textbf{Regularization}}\\
        &&\textbf{Learning}&\textbf{Learning}&\multicolumn{4}{c}{\textbf{Coefficients}}\\\cline{5-8}
        \textbf{Optimizer}&$\bm{\sigma}$&\textbf{Rate}&\textbf{Rate}&$\bm{\lambda}_{\bm{2}}$&$\bm{\lambda}_{\bm{5}}$&$\bm{\lambda}_{\bm{6}}$&$\bm{\lambda}_{\bm{7}}$\\\specialrule{.1em}{.05em}{.05em}
        SGD &0  &$10^{-4}$&$10^{-4}$&$-$      &$-$      &$-$      &$-$      \\\hline
        SGD &0.1&$10^{-2}$&$10^{-2}$&$10^{-5}$&$10^{-7}$&$10^{-5}$&$10^{-7}$\\\hline
        SGD &0.5&$10^{-3}$&$10^{-3}$&$10^{-7}$&$10^{-7}$&$10^{-7}$&$10^{-5}$\\\hline
        SGD &1  &$10^{-4}$&$10^{-4}$&$10^{-7}$&$10^{-7}$&$10^{-7}$&$10^{-5}$\\\specialrule{.1em}{.05em}{.05em}
        Adam&0  &$10^{-3}$&$10^{-4}$&$-$      &$-$      &$-$      &$-$      \\\hline
        Adam&0.1&$10^{-3}$&$10^{-4}$&$10^{-7}$&$10^{-7}$&$10^{-5}$&$10^{-5}$\\\hline
        Adam&0.5&$10^{-3}$&$10^{-4}$&$10^{-7}$&$10^{-7}$&$10^{-7}$&$10^{-5}$\\\hline
        Adam&1  &$10^{-3}$&$10^{-4}$&$10^{-7}$&$10^{-7}$&$10^{-7}$&$10^{-7}$
    \end{tabular}
\end{table}
\clearpage

\begin{figure}
    \begin{subfigure}{\linewidth}
        \centering
        \includegraphics[width=0.98\linewidth]{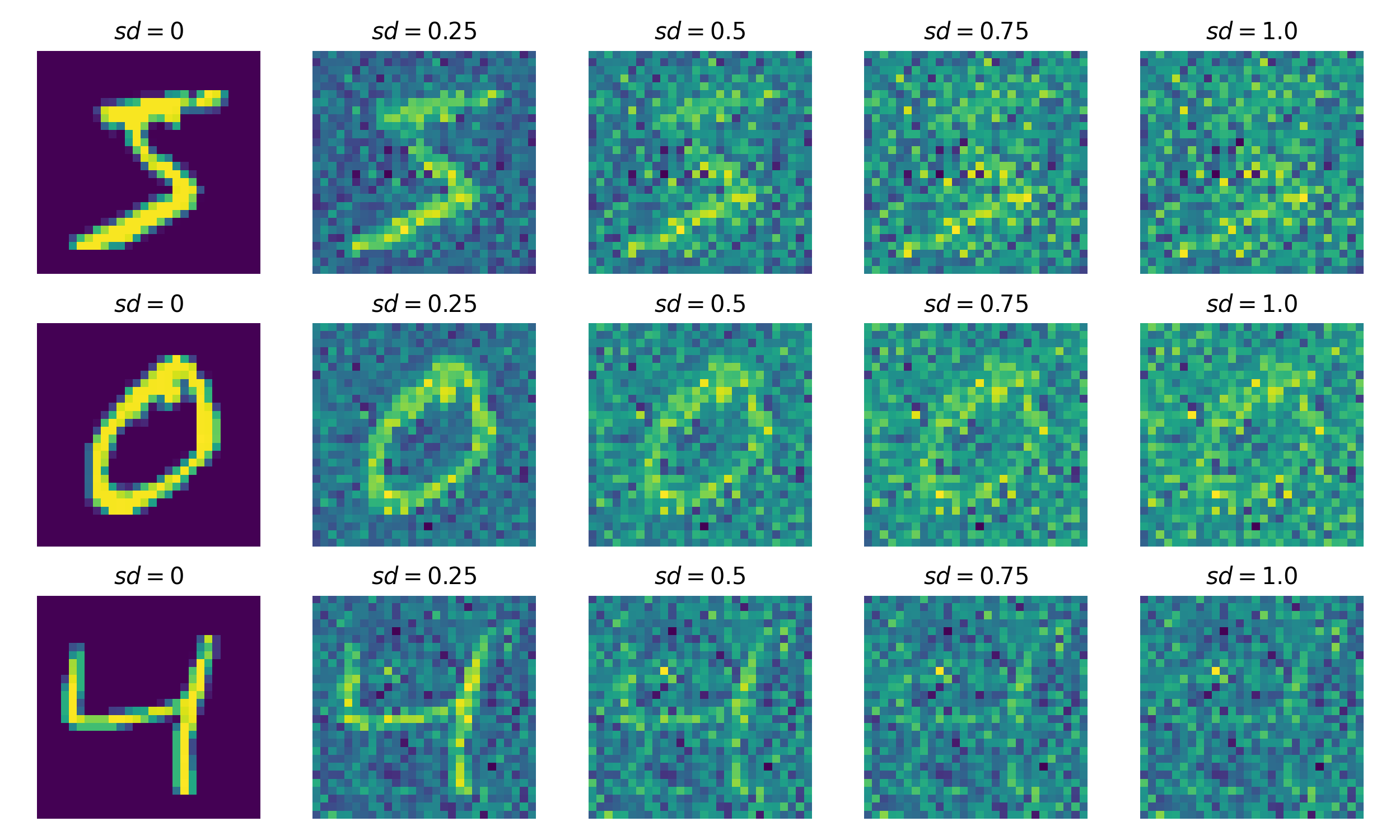}
        \caption{MNIST Noisy Images}
        \label{fig:noisy_mnist_images}
    \end{subfigure}
    \begin{subfigure}{\linewidth}
        \centering
        \includegraphics[width=0.98\linewidth]{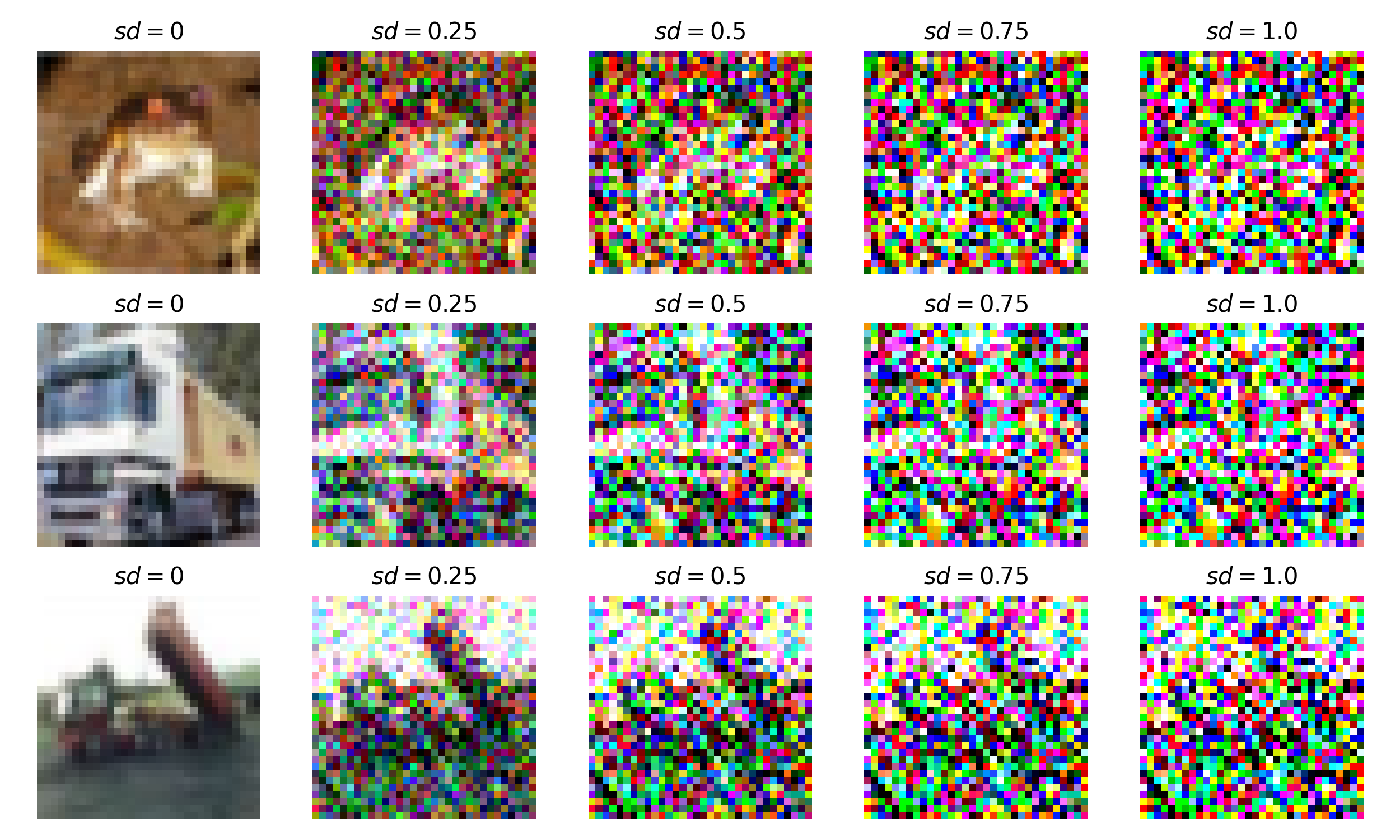}
        \caption{CIFAR-10 Noisy Images}
        \label{fig:noisy_cifar10_images}
    \end{subfigure}
    \caption{Examples of noisy images of MNIST and CIFAR-10 for noise standard deviations (sd)}
    \label{fig:noisy_images}
\end{figure}
\clearpage

\begin{figure}[t]
    \centering
    \includegraphics[height=0.95\textheight]{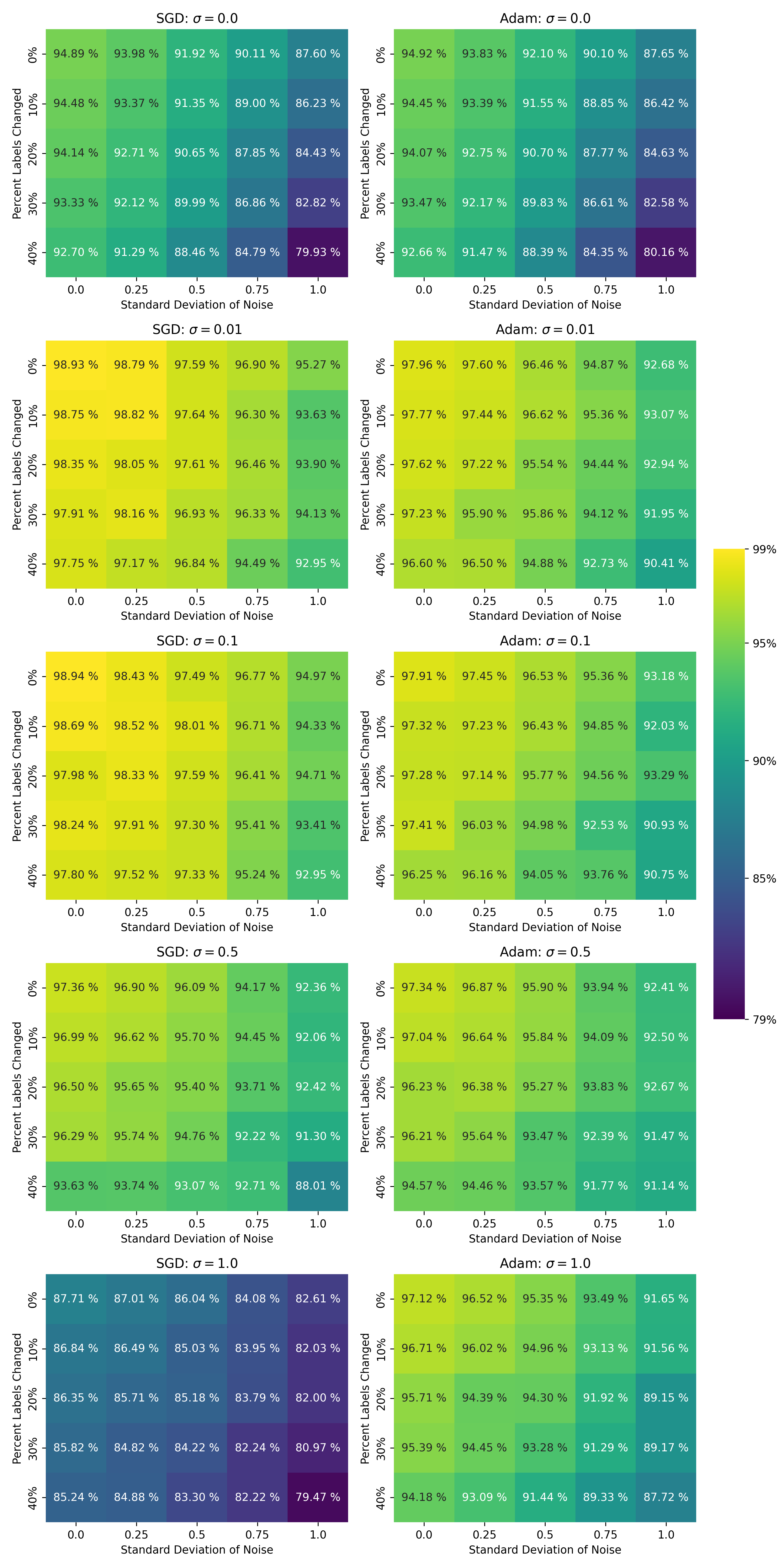}
    \caption{Test Accuracy for 25 SGD and Adam Experiments using MNIST}
    \label{fig:full_noise_mnist}
\end{figure}
\clearpage

\begin{figure}[t]
    \centering
    \includegraphics[height=0.95\textheight]{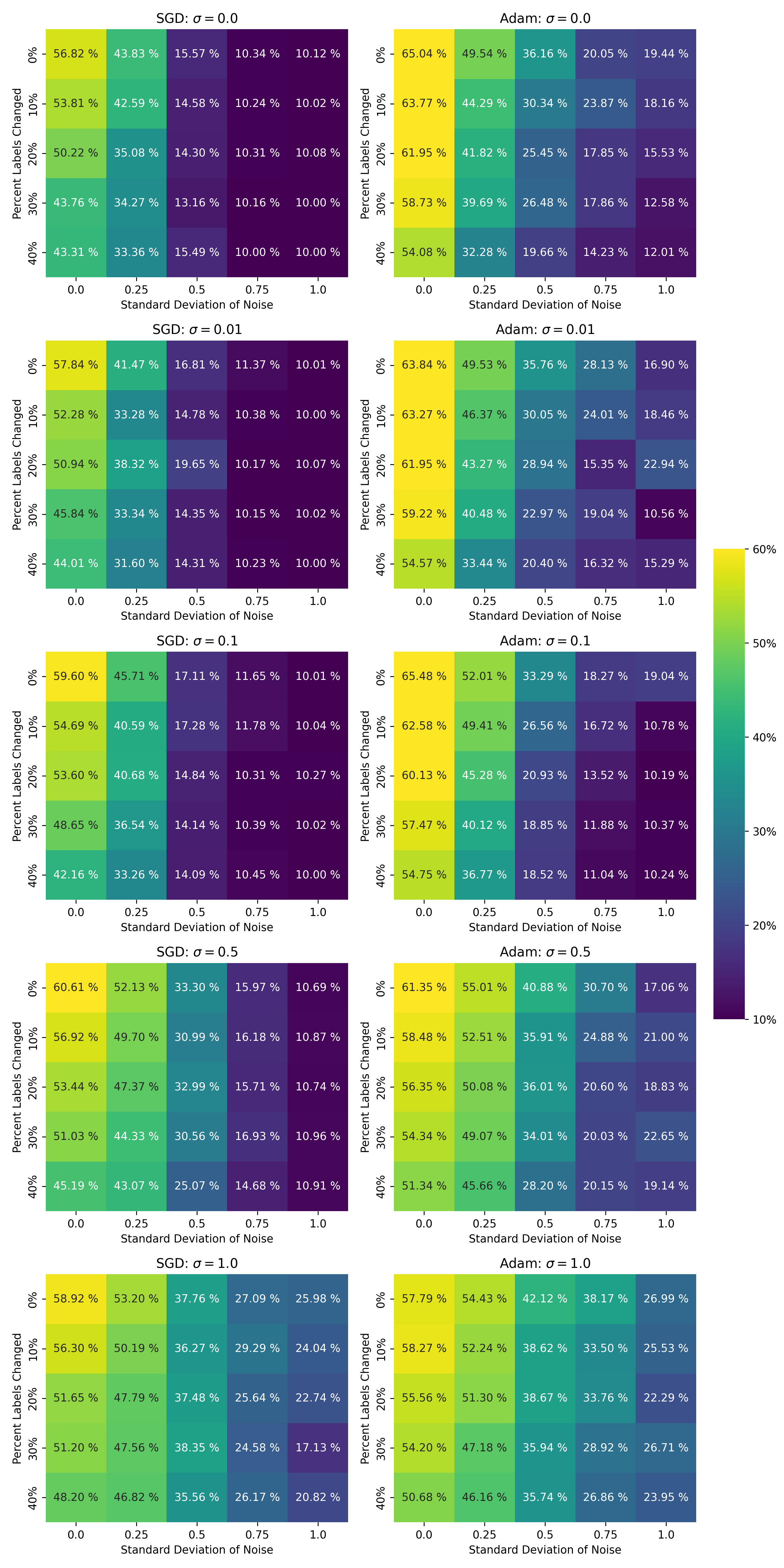}
    \caption{Test Accuracy for 25 SGD and Adam Experiments using CIFAR-10}
    \label{fig:full_noise_cifar10}
\end{figure}
\clearpage

\section{Gaussian smoothing stochastic variance reduced gradient (GSmoothSVRG)}
\label{app:convergence_svrg}

\begin{algorithm}[t]
    \caption{GSmoothSVRG}
    \label{alg:ssvrg}
    \begin{algorithmic}[1]
        \Require{$\widetilde{\bmx}_0\in\mathbb{R}^d$, $\sigma_s\geq 0$ for $s=0,1,...$}
        \For{$s=1,2,...$}
            \State $\widetilde{\bmx}=\widetilde{\bmx}_{s-1}$, $\bmx_0=\widetilde{\bmx}$, $\tau=\sigma_s$
            \State $\widetilde{\bmmu}_{\sigma_s}=\frac{1}{K}\sum_{i=1}^{K}\nabla f_{i,\sigma_s}(\widetilde{\bmx})$
            \For{$t=1,...,m$}
                \State $i_t\sim\text{Unif}[K]$
                \State $\bmv_t^{\sigma_s,\tau}=\nabla f_{i_t,\sigma_s}(\bmx_{t-1})-\nabla f_{i_t,\tau}(\widetilde{\bmx})+\widetilde{\bmmu}_{\tau}$
                \State $\bmx_t=\bmx_{t-1}-\lrate \bmv_t^{\sigma_s,\tau}$
            \EndFor
            \State $\widetilde{\bmx}_{s}=\bmx_t$ for $t\sim\text{Unif}[K]$
        \EndFor
    \end{algorithmic}
\end{algorithm}

Another modification of SGD that addresses this the variation in the gradient near a minimum is stochastic variance reduced gradient (SVRG) \cite{JohnsonZhang13}, which modifies the gradient used in the update step using inner and outer loops.
The outer loop computes a control variate that reduces the variance of the iterates in the inner loop.
The inner loop performs SGD steps, but with the control variate added to each step.
So, the inner update becomes
\begin{align}
    \bmv_{t} &= \nabla f_{k_t}(\bmx_{t-1})-\nabla f_{k_t}(\widetilde{\bmx})+\widetilde{\mu}\\
    \bmx_{t} &= \bmx_{t-1} - \lrate \bmv_{t},
\end{align}
where $\widetilde{x}$ is the output of the previous inner iteration and $\widetilde{\mu}$ is the full gradient at $\widetilde{x}$.
Since the motivation for SVRG is variance reduction, just like SGD, SVRG has a tendency converge to non-global minima, making it another good candidate for using Gaussian smoothing.
GSmoothSGD suffers the same variance issues as SGD, so in order to combine the benefits of variance reduction and smoothing, we propose Gaussian smoothed SVRG (GSmoothSVRG) which can be found in Algorithm~\ref{alg:ssvrg}.

In this section, we provide convergence results for GSmoothSVRG and then provide numerical experiments that show how the variance is reduced compared to both SGD and GSmoothSGD.

\subsection{Convergence of GSmoothSVRG}

This section provides the proof of Theorem~\ref{thm:ssvrg} and the additional background results needed for the proof.
The original convergence result for SVRG is stated for strongly convex function, we will do the same for GSmoothSVRG, which means we need to show that smoothing preserves strong convexity as well.
Our second result shows just that.

\begin{lem}
\label{lem:fstronglyconvexthensoisfsigma}
If $f$ is $\gamma$-strongly convex, then so is $f_{\sigma}$.
\end{lem}

The proof is the same as the proof when $f$ is convex (see Lemma~\ref{lem:fsigmalsmoothandconvex}) where only a minor modification is made for the strong convexity which now includes the quadratic function that $f$ grows as fast as. We still provide the proof here.

\begin{proof}[Proof of Lemma~\ref{lem:fstronglyconvexthensoisfsigma}]
Assume that $f$ is $\gamma$-strongly convex.
Then
\begin{align}
    f_{\sigma}(t\bmx+(1-t)\bmy)
    &=\frac{1}{\pi^{\nicefrac{d}{2}}}\int_{\mathbb{R}^d}f(t\bmx+(1-t)\bmy+\sigma\bmu)e^{-\|\bmu\|^2}\;du\\
    &=\frac{1}{\pi^{\nicefrac{d}{2}}}\int_{\mathbb{R}^d}f\Big(t(\bmx+\sigma\bmu)+(1-t)(\bmy+\sigma\bmu)\Big)e^{-\|\bmu\|^2}\;du\\
    &\leq\frac{1}{\pi^{\nicefrac{d}{2}}}\int_{\mathbb{R}^d}\Big(
        tf(\bmx+\sigma\bmu)
        +(1-t)f(\bmy+\sigma\bmu)\\
        &\qquad\qquad+\frac{\gamma}{2}t(1-t)\|(\bmx+\sigma\bmu)-(\bmy+\sigma\bmu)\|^2
    \Big)e^{-\|\bmu\|^2}\;du\\
    &=tf_{\sigma}(\bmx)+(1-t)f_{\sigma}(\bmy)+\frac{\gamma}{2}t(1-t)\|\bmx-\bmy\|^2
\end{align}
shows that $f_{\sigma}$ is also $\gamma$-strongly convex.
\end{proof}

We now show that adding variance reduction using SVRG to GSmoothSGD or equivalently smoothing SVRG does not change the convergence rate from the original SVRG for strongly convex and $L$-smooth functions.
In particular, compared with GSmoothSGD, GSmoothSVRG will converge for a fixed learning rate.

\begin{thm}\label{thm:ssvrg}
Consider GSmoothSVRG in Algorithm~\ref{alg:ssvrg}.
Assume $f_i$ is convex and $L$-smooth and $f=\frac{1}{K}\sum_{k=1}^{K}f_{k}$ is $\gamma$-strongly convex (for some $\gamma>0$).
Assume $m$ is sufficiently large so that
\begin{equation}
    \alpha=\frac{1+2L\lrate^2m}{\lrate\gamma(1-2L\lrate)m}<1.
\end{equation}
Then
\begin{equation}
    E(f(\widetilde{\bmx}_s)-f(\bmx_*))
    \leq\alpha^s E(f_{\sigma_0}(\widetilde{\bmx}_0)-f(\bmx_*))
    +\frac{Ld}{2}\sum_{i=1}^{s}\alpha^i\max(0,\sigma_{i-1}^{2}-\sigma_i^2).
\end{equation}
\end{thm}

The proof of this theorem is broken into four lemmas whose proofs follow the same structure as in~\cite{JohnsonZhang13} with adaptions for smoothing.
In particular, we mimic the proof that SVRG converges for strongly convex and $L$-smooth functions from~\cite{JohnsonZhang13}, which is broken into four lemmmas, and smooth when necessary.
The four lemmas build on each other culminating in the fact that the iterates are getting closer to the minimum.
The proof of the theorem then iteratively applies this result to get the claimed bound.
We begin by justifying the first lemma which states a bound between the difference in the gradients of the iterates and the minimum.

\begin{lem}\label{lem:ssvrgstep1}
For each $i\in\{1,...,K\}$, let $f_i$ be $L$-smooth and convex. 
Then for any $\sigma\geq 0$,
\begin{equation}
    E(\|\nabla f_{i,\sigma}(\bmx)-\nabla f_{i}(\bmx_*)\|^2)
    \leq 2L(f_{\sigma}(\bmx)-f(\bmx_*)).
\end{equation}
Furthermore, for $\sigma\geq\tau\geq 0$,
\begin{equation}
    E(\|\nabla f_{i,\sigma}(\bmx)-\nabla f_{i,\tau}(\bmx_*^{\tau})\|^2)
    \leq 4L(f_{\sigma}(\bmx)-f(\bmx_*))
\end{equation}
where $\bmx_*^{\tau}$ is the minimizer of $f_{\tau}$.
\end{lem}

We can adapt the above to get the following statements as well:
\begin{equation}
    E(\|\nabla f_{i,\sigma}(\bmx)-\nabla f_{i,\tau}(\bmx_*)\|^2)\\
    \leq 2L(f_{\sigma}(\bmx)-f(\bmx_*))+\frac{1}{2}\tau^2L^2d
\end{equation}
\begin{equation}
    E(\|\nabla f_{i,\sigma}(\bmx)-\nabla f_{i,\tau}(\bmx)\|^2)\leq 4L(f_{\sigma}(\bmx)-f(\bmx_*))
\end{equation}

\begin{proof}
Note that since each $f_k$ is convex, $f$ and $f_{k,\sigma}$ are convex for any $\sigma\geq 0$. This means that $f_{\sigma}$ is also convex for any $\sigma\geq 0$.
Let
\begin{equation}
    g_i^{\sigma}(\bmx)=f_{i,\sigma}(\bmx)-f_i(\bmx_*)-\langle\nabla f_i(\bmx_*),\bmx-\bmx_*\rangle.
\end{equation}
Then since $f_i$ is convex,
\begin{equation}
    f_{i,\sigma}(\bmx)-f_i(\bmx_*)
    \geq f_i(\bmx)-f_i(\bmx_*)
    \geq \langle\nabla f_i(\bmx_*),\bmx-\bmx_*\rangle.
\end{equation}
This means $g_i^{\sigma}(\bmx)\geq 0$ for any $i$ and $\sigma$.
Since $f_{i,\sigma}$ is $L$-smooth, so is $g_i^{\sigma}$.
So,
\begin{equation}
    0
    \leq g_i^{\sigma}\left(\bmx-\tfrac{1}{L}\nabla g_i^{\sigma}(\bmx)\right)
    \leq g_i^{\sigma}(\bmx)-\frac{1}{2L}\|g_i^{\sigma}(\bmx)\|^2
\end{equation}
and rearranging we have
\begin{equation}
    \|g_i^{\sigma}(\bmx)\|^2
    \leq 2Lg_i^{\sigma}(\bmx).
\end{equation}
Since
\begin{equation}
    \nabla g_i^{\sigma}(\bmx)=\nabla f_{i,\sigma}(\bmx)-\nabla f_i(\bmx_*),
\end{equation}
we have
\begin{equation}
    \|\nabla f_{i,\sigma}(\bmx)-\nabla f_i(\bmx_*)\|^2
    \leq 2L\Big(f_{i,\sigma}(\bmx)-f_i(\bmx_*)-\langle\nabla f_i(\bmx_*),\bmx-\bmx_*\rangle\Big).
\end{equation}
Therefore, taking the expectation over $i$,
\begin{align}
\begin{split}
    E(\|\nabla f_{i,\sigma}(\bmx)-\nabla f_i(\bmx_*)\|^2)
    &\leq 2LE(f_{i,\sigma}(\bmx)-f_i(\bmx_*)-\langle\nabla f_i(\bmx_*),\bmx-\bmx_*\rangle)\\
    &=2L(f_{\sigma}(\bmx)-f(\bmx_*)).
\end{split}
\end{align}
The furthermore statement can be seen by
\begin{align}
\begin{split}
    E(\|\nabla f_{i,\sigma}(\bmx)-\nabla f_{i,\tau}(\bmx_*^{\tau})\|^2)
    &\leq E(\|\nabla f_{i,\sigma}(\bmx)-\nabla f_{i}(\bmx_*)\|^2)
    + E(\|\nabla f_{i,\tau}(\bmx_*^{\tau})-\nabla f_{i}(\bmx_*)\|^2)\\
    &\leq 2L(f_{\sigma}(\bmx)-f(\bmx_*))+2L(f_{\tau}(\bmx_*^{\tau})-f(\bmx_*))\\
    &\leq 2L(f_{\sigma}(\bmx)-f(\bmx_*))+2L(f_{\sigma}(\bmx)-f(\bmx_*))\\
    &=4L(f_{\sigma}(\bmx)-f(\bmx_*)),
\end{split}
\end{align}
since $f_{\sigma}(\bmx)\geq f_{\tau}(\bmx_*^{\tau})$.
\end{proof}

Next, we bound the GSmoothSVRG gradient update by a linear combination of function outputs.
The previous lemma is applied in the proof in order to derive the bound.

\begin{lem}\label{lem:ssvrgstep2}
For each $i\in\{1,...,K\}$, let $f_i$ be $L$-smooth and convex.
For $\sigma\geq\tau\geq 0$,
\begin{equation}
    E(\|\bmv_t\|^2|\bmx_{t-1})
    \leq 4L(f_{\sigma}(\bmx_{t-1})-f(\bmx_*)+f_{\sigma}(\widetilde{\bmx})-f(\bmx_*)).
\end{equation}
\end{lem}

Recall that
\begin{equation}
    \bmv_t^{\sigma,\tau}=\nabla f_{i_t,\sigma}(\bmx_{t-1})-\nabla f_{i_t,\tau}(\widetilde{\bmx})+\widetilde{\bmmu}_{\tau}.
\end{equation}
This means
\begin{align}
\begin{split}
    \label{eqn:expectationofvt}
    E(\bmv_t^{\sigma,\tau}|\bmx_{t-1})
    &=\nabla f_{\sigma}(\bmx_{t-1})-\nabla f_{\tau}(\widetilde{\bmx})+\nabla f_{\tau}(\widetilde{\bmx})\\
    &=\nabla f_{\sigma}(\bmx_{t-1}).
\end{split}
\end{align}

\begin{proof}
Observe
\begin{align}
\begin{split}
    E(\|\bmv_t\|^2|\bmx_{t-1})
    &=E(\|\nabla f_{i_t,\sigma}(\bmx_{t-1})-\nabla f_{i_t,\tau}(\widetilde{\bmx})+\widetilde{\bmmu}_{\tau}\|^2|\bmx_{t-1})\\
    &\leq 
        E(\|\nabla f_{i_t,\sigma}(\bmx_{t-1})-\nabla f_{i_t,\tau}(\bmx_*^{\tau})\|^2|\bmx_{t-1})\\&\qquad
      + E(\|\nabla f_{i_t,\tau}(\bmx_*^{\tau})-\nabla f_{i_t,\tau}(\widetilde{\bmx})+\widetilde{\bmmu}_{\tau}\|^2|\bmx_{t-1})\\
    &\stackrel{(1)}{\leq}
        E(\|\nabla f_{i_t,\sigma}(\bmx_{t-1})-\nabla f_{i_t,\tau}(\bmx_*^{\tau})\|^2|\bmx_{t-1})\\&\qquad
      + E(\|\nabla f_{i_t,\tau}(\bmx_*^{\tau})-\nabla f_{i_t,\tau}(\widetilde{\bmx})-E(f_{i_t,\tau}(\bmx_*^{\tau})-\nabla f_{i_t,\tau}(\widetilde{\bmx}))\|^2|\bmx_{t-1})\\
    &\stackrel{(2)}{\leq}
        E(\|\nabla f_{i_t,\sigma}(\bmx_{t-1})-\nabla f_{i_t,\tau}(\bmx_*^{\tau})\|^2|\bmx_{t-1})
      + E(\|\nabla f_{i_t,\tau}(\bmx_*^{\tau})-\nabla f_{i_t,\tau}(\widetilde{\bmx})\|^2|\bmx_{t-1})\\
    &\stackrel{\text{Lem~\ref{lem:ssvrgstep1}}}{\leq}
        4L(f_{\sigma}(\bmx_{t-1})-f(\bmx_*))
      + 4L(f_{\tau}(\widetilde{\bmx})-f(\bmx_*))\\
    &\stackrel{(3)}{\leq}
        4L(f_{\sigma}(\bmx_{t-1})-f(\bmx_*)+f_{\sigma}(\widetilde{\bmx})-f(\bmx_*))
\end{split}
\end{align}
where step (1) is due to $E(\nabla f_{i_t,\tau}(\bmx_*^{\tau}))=0$, step (2) follows from $E(\|\xi-E(\xi)\|^2)=E(\|\xi\|^2)-\|E(\xi)\|^2\leq E(\|\xi\|^2)$ for any random vector $\xi$, and step (3) is because $f$ is convex and $\sigma\geq\tau$.
\end{proof}

With this bound on the GSmoothSVRG gradient update, we are ready to bound the expected value of the iterates.

\begin{lem}\label{lem:ssvrgstep3}
For each $i\in\{1,...,K\}$, let $f_i$ be $L$-smooth and convex.
For $\sigma\geq\tau\geq 0$,
\begin{equation}
    2\lrate(1-2L\lrate)mE(f_{\sigma}(\widetilde{\bmx}_s)-f(\bmx_*))
    \leq E(\|\bmx_0-\bmx_*\|^2)+4L\lrate^2mE(f_{\sigma}(\widetilde{\bmx})-f(\bmx_*)).
\end{equation}
\end{lem}

\begin{proof}
First,
\begin{align}
\begin{split}
    E(\|\bmx_t-\bmx_*\|^2|\bmx_{t-1})
    &\stackrel{\text{def. }\bmx_t}{=}\|\bmx_{t-1}-\bmx_*\|^2-2\lrate\langle \bmx_{t-1}-\bmx_*,E(\bmv_t|\bmx_{t-1})\rangle+\lrate^2E(\|\bmv_t\|^2|\bmx_{t-1})\\
    &\stackrel{\text{eqn. \eqref{eqn:expectationofvt}}}{=}\|\bmx_{t-1}-\bmx_*\|^2-2\lrate\langle \bmx_{t-1}-\bmx_*,\nabla f_{\sigma}(\bmx_{t-1})\rangle+\lrate^2E(\|\bmv_t\|^2|\bmx_{t-1})\\
    &\stackrel{\text{Lem. \ref{lem:ssvrgstep2}}}{\leq}
        \|\bmx_{t-1}-\bmx_*\|^2
        -2\lrate\langle \bmx_{t-1}-\bmx_*,\nabla f_{\sigma}(\bmx_{t-1})\rangle\\&\qquad
        +4L\lrate^2(f_{\sigma}(\bmx_{t-1})-f(\bmx_*)+f_{\sigma}(\widetilde{\bmx})-f(\bmx_*))\\
    &\stackrel{\text{conv.}}{\leq}
        \|\bmx_{t-1}-\bmx_*\|^2
        -2\lrate(f_{\sigma}(\bmx_{t-1})-f(\bmx_*))\\&\qquad
        +4L\lrate^2(f_{\sigma}(\bmx_{t-1})-f(\bmx_*)+f_{\sigma}(\widetilde{\bmx})-f(\bmx_*))\\
    &=
        \|\bmx_{t-1}-\bmx_*\|^2
        -2\lrate(1-2L\lrate)(f_{\sigma}(\bmx_{t-1})-f(\bmx_*))
        +4L\lrate^2(f_{\sigma}(\widetilde{\bmx})-f(\bmx_*)).
\end{split}
\end{align}
Since $P(\widetilde{\bmx}_s=\bmx_t)=\frac{1}{m}$ for $t=0,...,m-1$, then
\begin{equation}
    mE(f_{\sigma}(\widetilde{\bmx}_s)|\bmx_0,...,\bmx_{m-1}) = \sum_{t=0}^{m-1}f_{\sigma}(\bmx_t).
\end{equation}
So, summing over the $m$ steps gives
\begin{align}
\begin{split}
    E(\|\bmx_m-\bmx_*\|^2|\bmx_0,...,\bmx_{m-1})
    &\leq
        \|\bmx_0-\bmx_*\|^2
        -2\lrate(1-2L\lrate)mE(f_{\sigma}(\widetilde{\bmx}_s)-f(\bmx_*)|\bmx_0,...,\bmx_{m-1})\\&\qquad
        +4L\lrate^2m(f_{\sigma}(\widetilde{\bmx})-f(\bmx_*)).
\end{split}
\end{align}
Rearranging shows
\begin{multline}
    E(\|\bmx_m-\bmx_*\|^2|\bmx_0,...,\bmx_{m-1})
    +2\lrate(1-2L\lrate)mE(f_{\sigma}(\widetilde{\bmx}_s)-f(\bmx_*)|\bmx_0,...,\bmx_{m-1})
    \\\leq
        \|\bmx_0-\bmx_*\|^2
        +4L\lrate^2m(f_{\sigma}(\widetilde{\bmx})-f(\bmx_*)).
\end{multline}
Since $\|\bmx_m-\bmx_*\|^2\geq 0$,
\begin{equation}
    2\lrate(1-2L\lrate)mE(f_{\sigma}(\widetilde{\bmx}_s)-f(\bmx_*)|\bmx_0,...,\bmx_{m-1})
    \leq
        \|\bmx_0-\bmx_*\|^2
        +4L\lrate^2m(f_{\sigma}(\widetilde{\bmx})-f(\bmx_*)).
\end{equation}
Finally, taking the expectation gives
\begin{align}
\begin{split}
    &2\lrate(1-2L\lrate)mE(f_{\sigma}(\widetilde{\bmx}_s)-f(\bmx_*))\\
    &=2\lrate(1-2L\lrate)mE(E(f_{\sigma}(\widetilde{\bmx}_s)-f(\bmx_*)|\bmx_0,...,\bmx_{m-1}))\\
    &\leq
        E(E(\|\bmx_0-\bmx_*\|^2|\bmx_0,...,\bmx_{m-1}))
        +4L\lrate^2mE(f_{\sigma}(\widetilde{\bmx})-f(\bmx_*)|\bmx_0,...,\bmx_{m-1}))\\
    &=
        E(\|\bmx_0-\bmx_*\|^2)
        +4L\lrate^2mE(f_{\sigma}(\widetilde{\bmx})-f(\bmx_*)).
\end{split}
\end{align}
\end{proof}

All of the work so far now culminates in showing that, for strongly convex functions, expected difference between the iterates and minimum decreases by a multiplicative factor smaller than one each time.

\begin{lem}\label{lem:ssvrgstep4}
Assume $f_i$ is convex and $L$-smooth and $f$ is $\gamma$-strongly convex (for some $\gamma>0$).
For $\sigma\geq\tau\geq 0$, if $f$ is $\gamma$-strongly convex, then
\begin{equation}
    E(f_{\sigma}(\widetilde{\bmx}_s)-f(\bmx_*))
    \leq \frac{1+2L\lrate^2m}{\lrate\gamma(1-2L\lrate)m}E(f_{\sigma}(\widetilde{\bmx})-f(\bmx_*)).
\end{equation}
\end{lem}

\begin{proof}
Since $f$ is $\gamma$-strongly convex, so is $f_{\sigma}$.
As $\bmx_0=\widetilde{\bmx}$,
\begin{equation}
    E(\|\bmx_0-\bmx_*\|^2)
    \leq\frac{2}{\gamma}E(f_{\sigma}(\widetilde{\bmx})-f_{\sigma}(\bmx_*))
    \leq\frac{2}{\gamma}E(f_{\sigma}(\widetilde{\bmx})-f(\bmx_*)).
\end{equation}
Combining this with Lemma~\ref{lem:ssvrgstep3} shows
\begin{align}
\begin{split}
    2\lrate(1-2L\lrate)mE(f_{\sigma}(\widetilde{\bmx}_s)-f(\bmx_*))
    &\leq E(\|\bmx_0-\bmx_*\|^2)+4L\lrate^2mE(f_{\sigma}(\widetilde{\bmx})-f(\bmx_*))\\
    &\leq\frac{2}{\gamma}E(f_{\sigma}(\widetilde{\bmx})-f(\bmx_*))+4L\lrate^2mE(f_{\sigma}(\widetilde{\bmx})-f(\bmx_*))\\
    &=2\left(\frac{1}{\gamma}-2L\lrate^2m\right)E(f_{\sigma}(\widetilde{\bmx})-f(\bmx_*)).
\end{split}
\end{align}
Arithmetic gives the result.
\end{proof}

With the result of the previous lemma in hand, we are finally ready to prove that GSmoothSVRG converges.

\begin{proof}[Proof of Theorem~\ref{thm:ssvrg}]
Using Lemma~\ref{lem:ssvrgstep4}, we have
\begin{align}
\begin{split}
    E(f_{\sigma_s}(\widetilde{\bmx}_s)-f(\bmx_*))
    &\leq\alpha E(f_{\sigma_s}(\widetilde{\bmx}_{s-1})-f(\bmx_*))\\
    &\leq\alpha E(f_{\sigma_{s-1}}(\widetilde{\bmx}_{s-1})-f(\bmx_*))+\frac{Ld}{2}\alpha\max(0,\sigma_{s-1}^2-\sigma_s^2)\\
    &\vdots\\
    &\leq\alpha^s E(f_{\sigma_{0}}(\widetilde{\bmx}_{0})-f(\bmx_*))+\frac{Ld}{2}\sum_{i=1}^{s}\alpha^i\max(0,\sigma_{i-1}^2-\sigma_i^2).
\end{split}
\end{align}
\end{proof}

Note that the only condition that $\tau$ in GSmoothSVRG needs to satisfy is $\sigma_t\geq\tau\geq 0$ at each iteration.
The two obvious choices for $\tau$ are $\sigma_s$ or $0$.
If $\tau = \sigma_s$, then we are performing SVRG on $f_{\sigma_s}$.
On the other hand, if $\tau = 0$, then we are making the control variate of SVRG include information about the gradient of the original, non-smoothed function.

\subsection{Numerical Experiments for GSmoothSVRG}

For the GSmoothSVRG experiments, we repeat the noisy MNIST experiments using a batch size of 1.
Since our primary focus was to compare variance of the norm of the iterative update and the rate of convergence between GSmoothSVRG and SGD, we only trained for one epoch (of 5000 steps in total).
Additionally, we wanted to use the same learning rate for both algorithms; after a hyperparameter search, we chose to use 0.01 for both methods.
A hyperparameter search was also done to pick regularization weights as well as the smoothing parameter for the inner GSmoothSVRG update.
The regularization weights for GSmoothSVRG were $10^{-13}$ for both the first ReLU layer and first dense layer and $10^{-11}$ for both the second ReLU layer and the output layer, and we use the same smoothing parameter for both the inner and outer updates.
The standard deviation of the norm of the gradient update for all of the GSmoothSVRG experiments can be found in Figure~\ref{fig:full_svrg_noise_variation}.

\begin{figure}
    \centering
    \includegraphics[width=\textwidth]{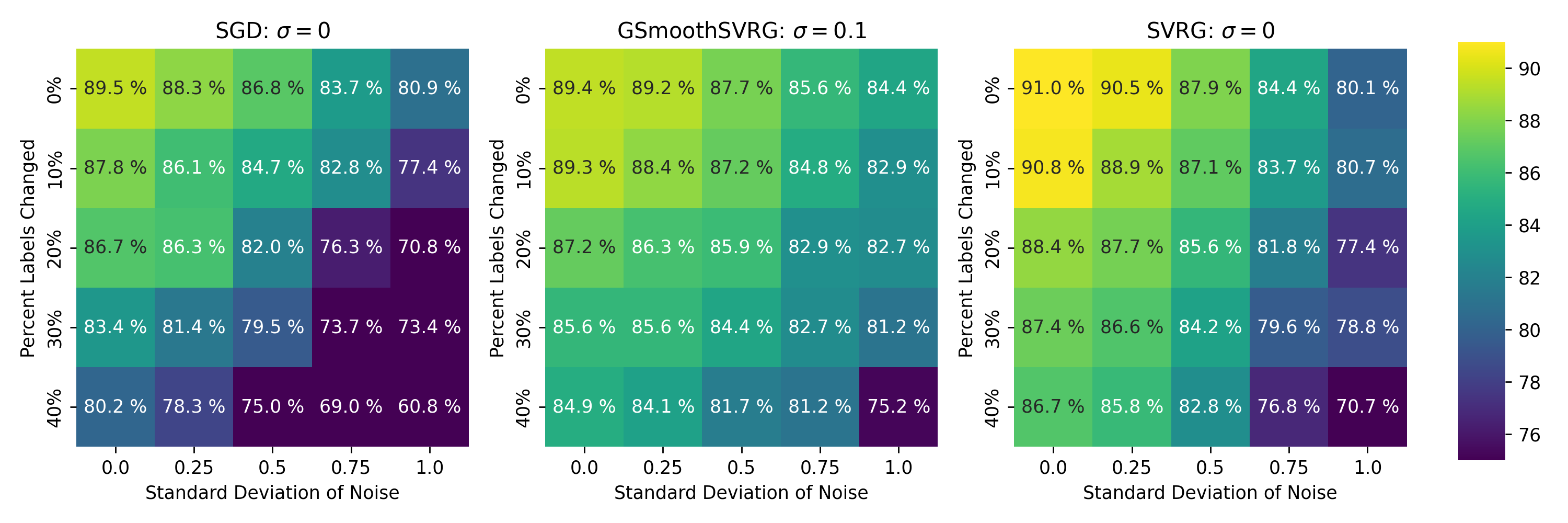}
    \caption{SVRG Noise Heatmaps}
    \label{fig:full_svrg_noise}
\end{figure}

\begin{figure}
    \centering
    \includegraphics[width=\textwidth]{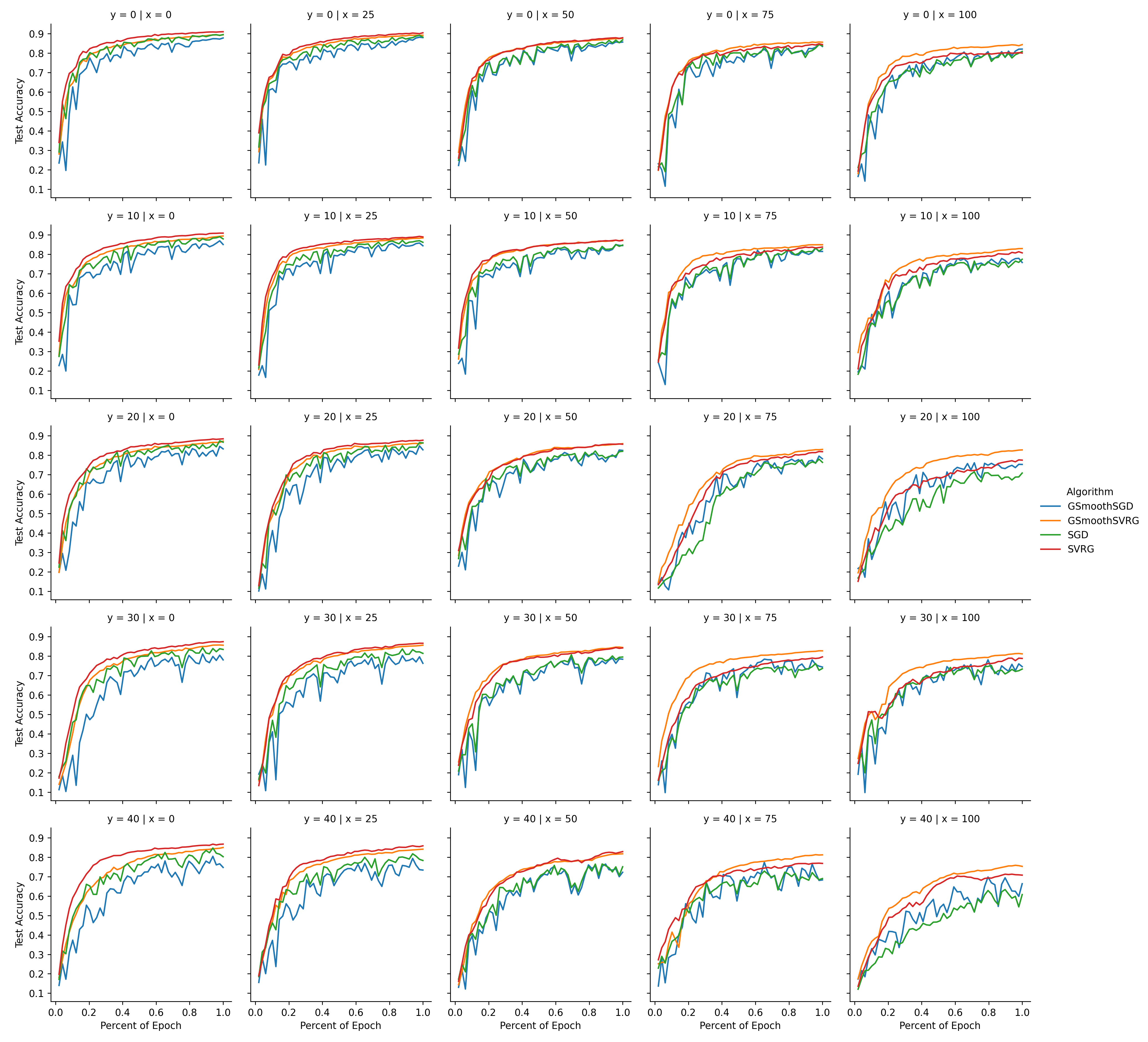}
    \caption{SVRG Noise Heatmap Accuracy}
    \label{fig:full_svrg_noise_accuracy}
\end{figure}

\begin{figure}
    \centering
    \includegraphics[width=\textwidth]{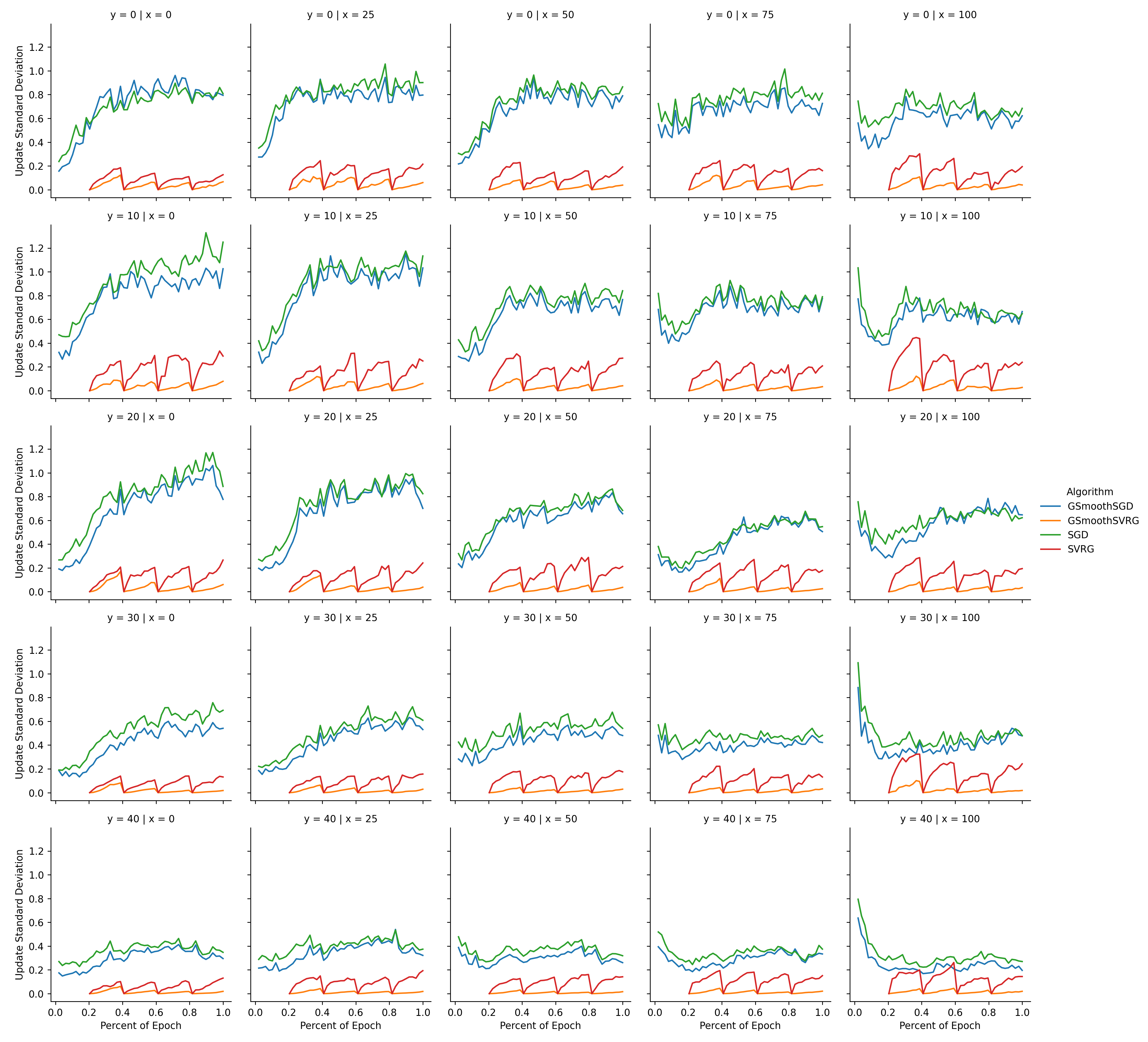}
    \caption{SVRG Noise Heatmap Update Standard Deviation}
    \label{fig:full_svrg_noise_variation}
\end{figure}

\end{document}